\documentclass[12pt]{amsart}
\usepackage{amsmath, amsthm}
\usepackage[margin=1in]{geometry}
 \usepackage{amsmath,amsthm,amssymb}
 \usepackage{lipsum} 
\usepackage{picins} 
\usepackage{graphicx} 
\usepackage{color}
\usepackage{CJK}
\usepackage{mathrsfs}
\usepackage{picinpar}

\newtheorem{thm}{Theorem}[section]
\newtheorem{lem}[thm]{Lemma}
\newtheorem{corollary}[thm]{Corollary}
\newtheorem{proposition}[thm]{Proposition}
\newtheorem{defn}[thm]{Definition}

\newtheorem{rem}[thm]{Remark}

\newcommand{\Across}{\raisebox{-0.25\height}{\includegraphics[width=0.7cm]{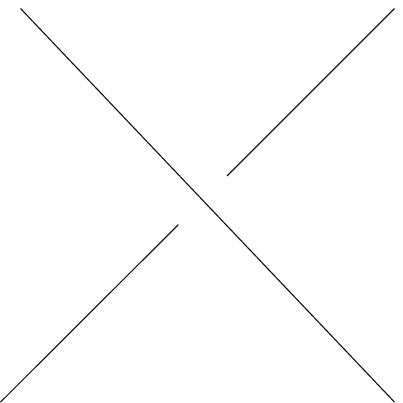}}}
\newcommand{\Bcross}{\raisebox{-0.25\height}{\includegraphics[width=0.7cm]{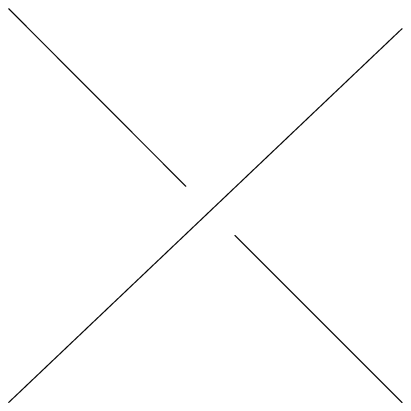}}}

\newcommand{\Asmooth}{\raisebox{-0.25\height}{\includegraphics[width=0.8cm]{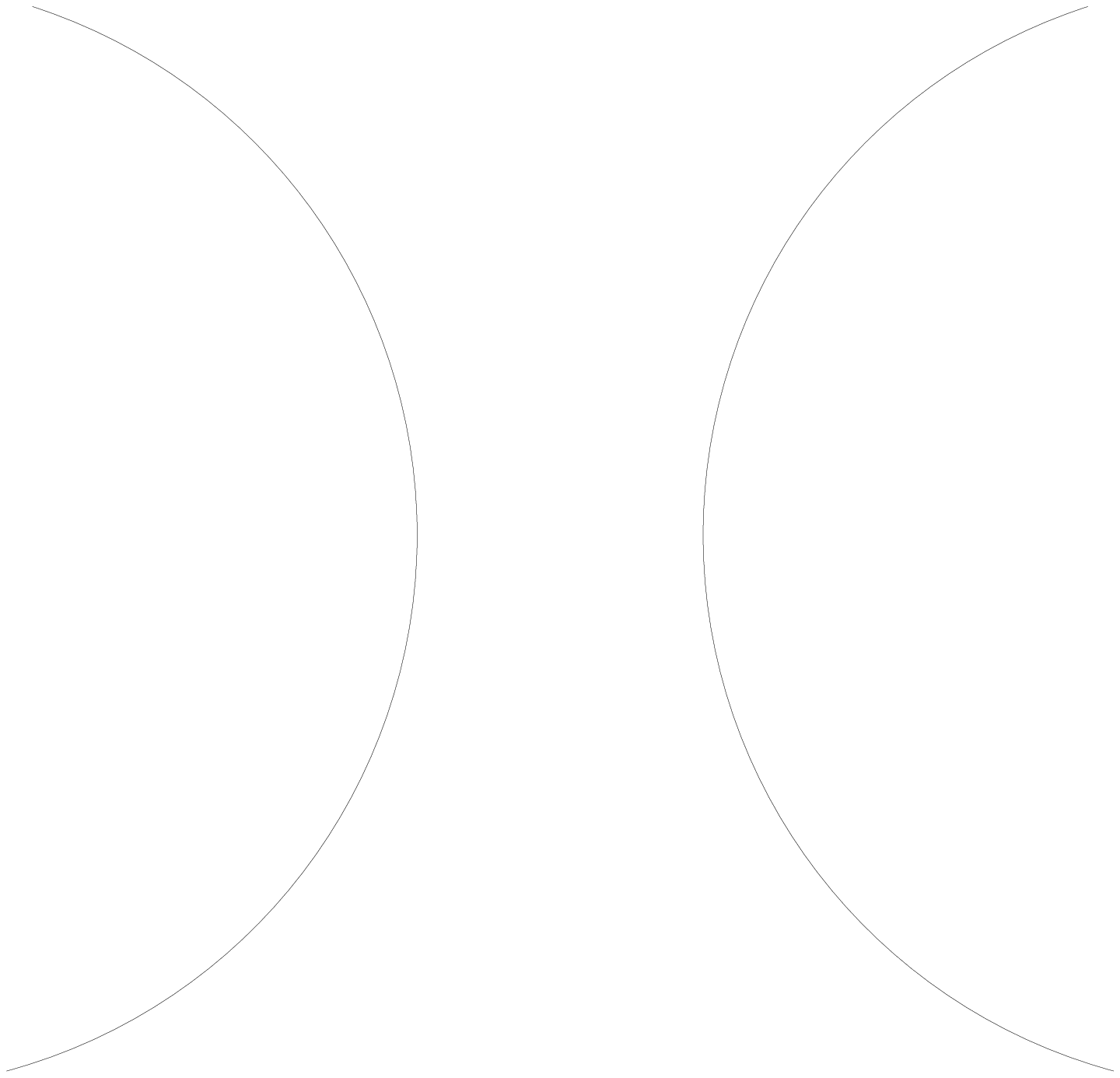}}}
\newcommand{\Bsmooth}{\raisebox{-0.25\height}{\includegraphics[width=0.8cm]{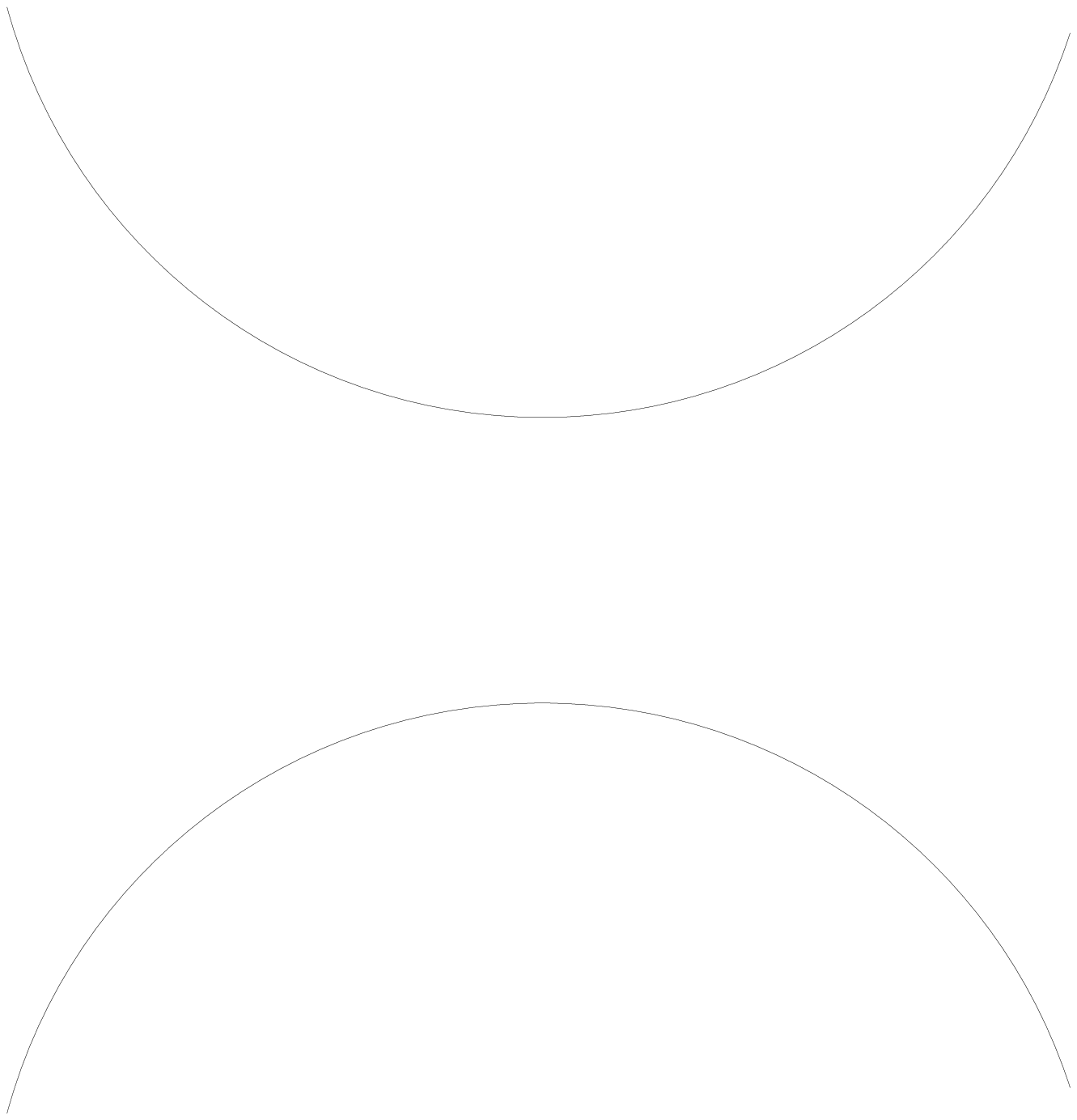}}}
\newcommand{\onevertex}{\raisebox{-0.25\height}{\includegraphics[width=0.8cm]{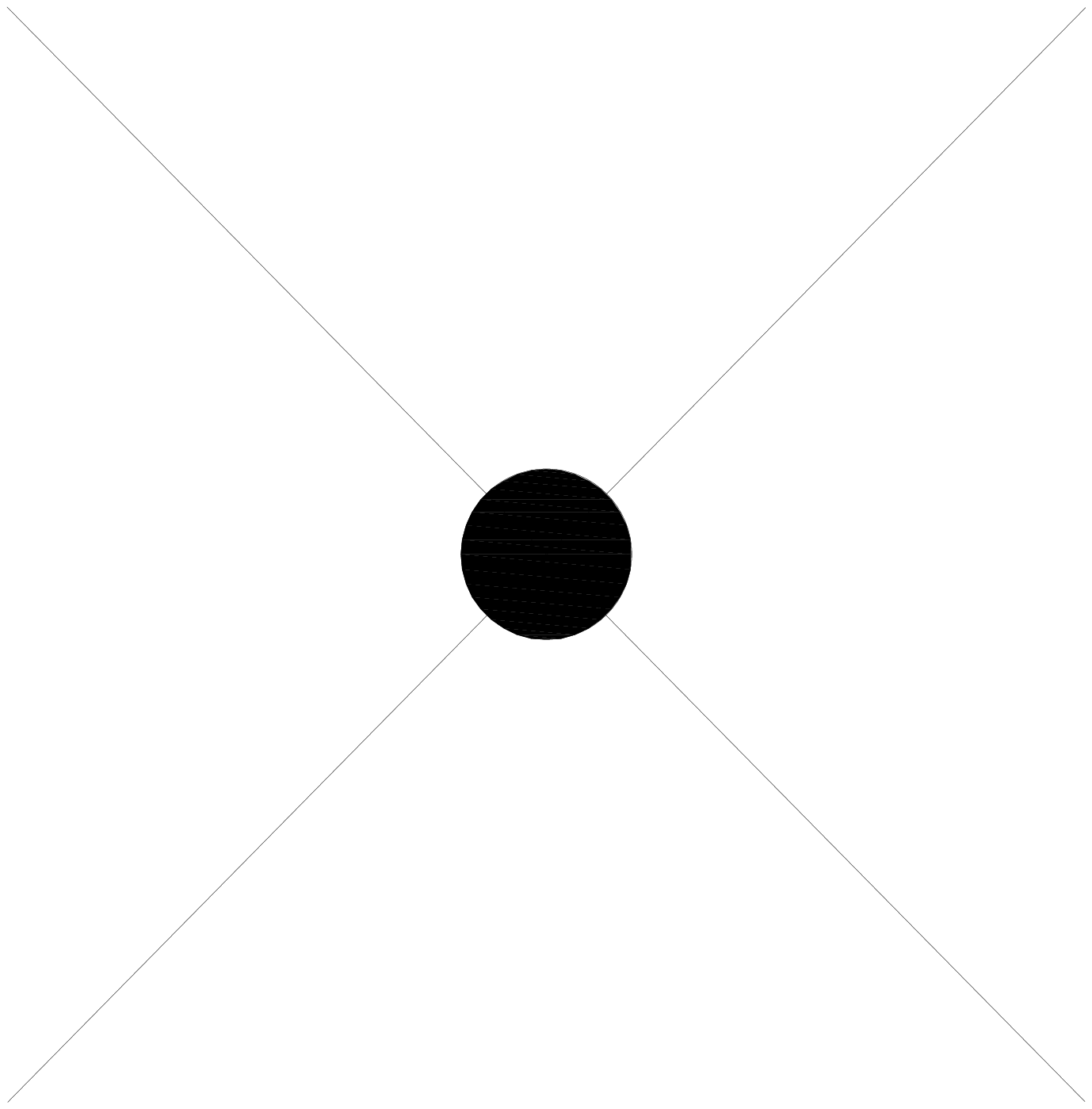}}}
\newcommand{\downvertexloop}{\raisebox{-0.25\height}{\includegraphics[width=0.8cm]{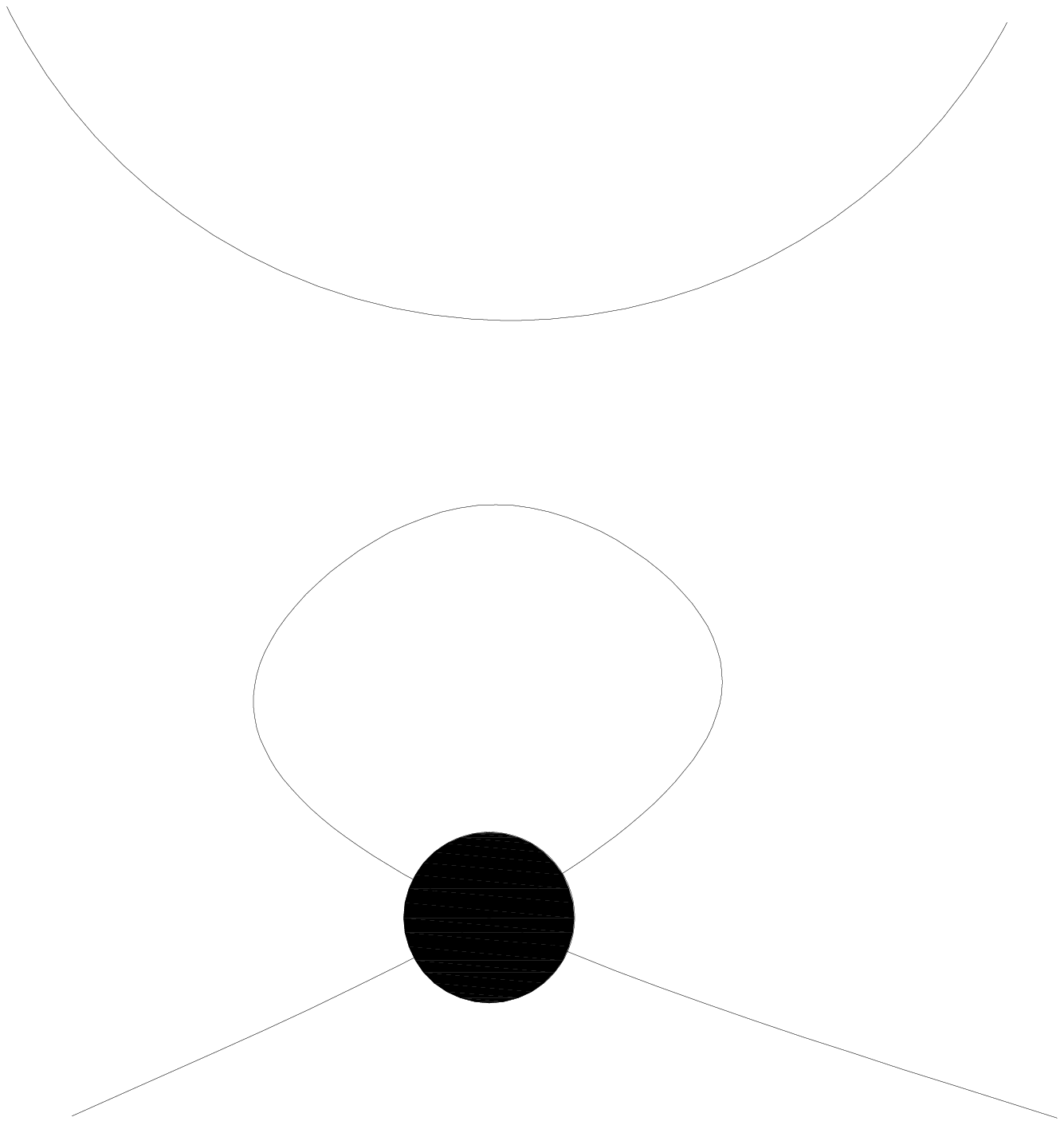}}}
\newcommand{\upvertexloop}{\raisebox{-0.25\height}{\includegraphics[width=0.8cm]{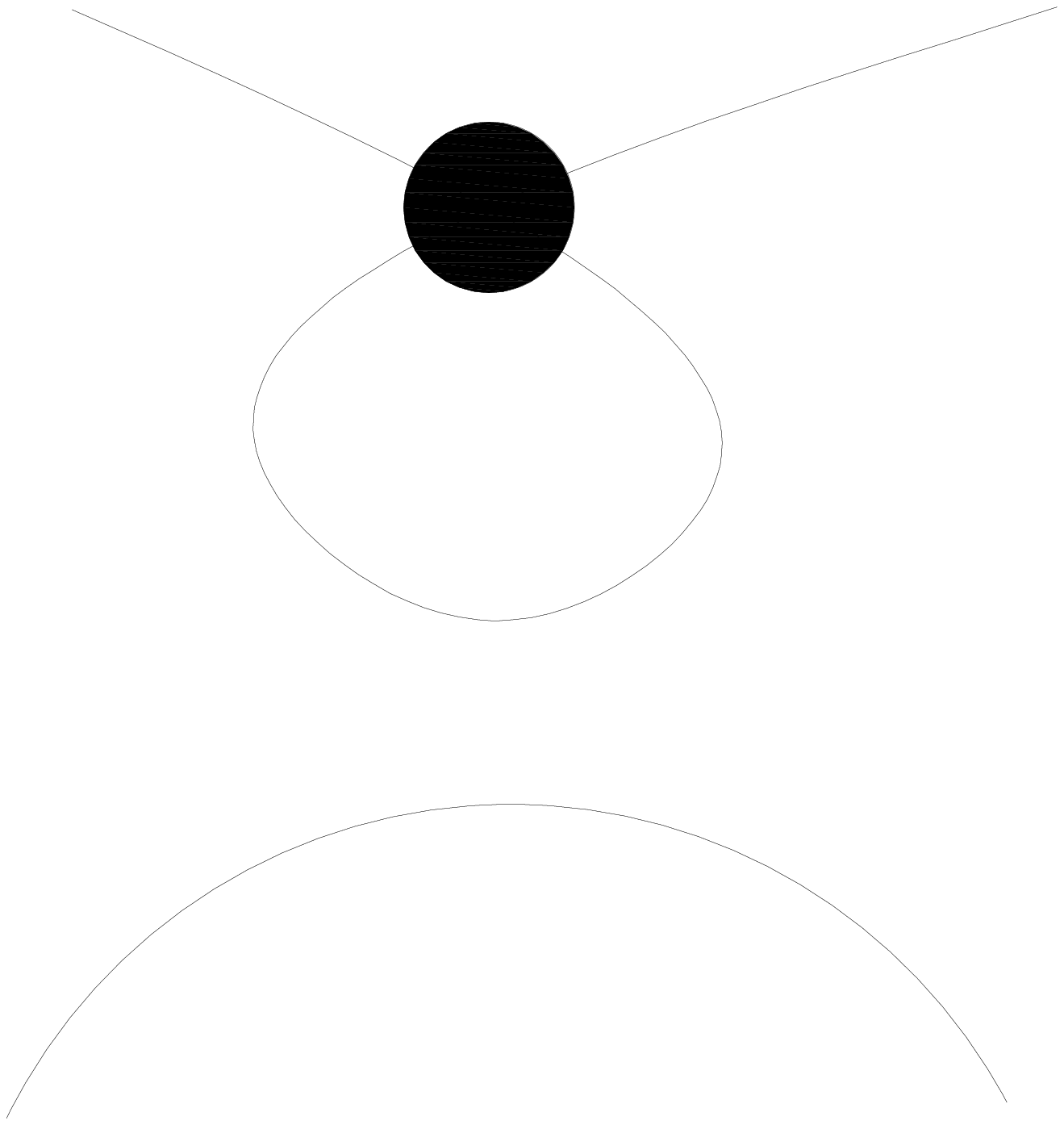}}}
\newcommand{\twovertex}{\raisebox{-0.25\height}{\includegraphics[width=0.8cm]{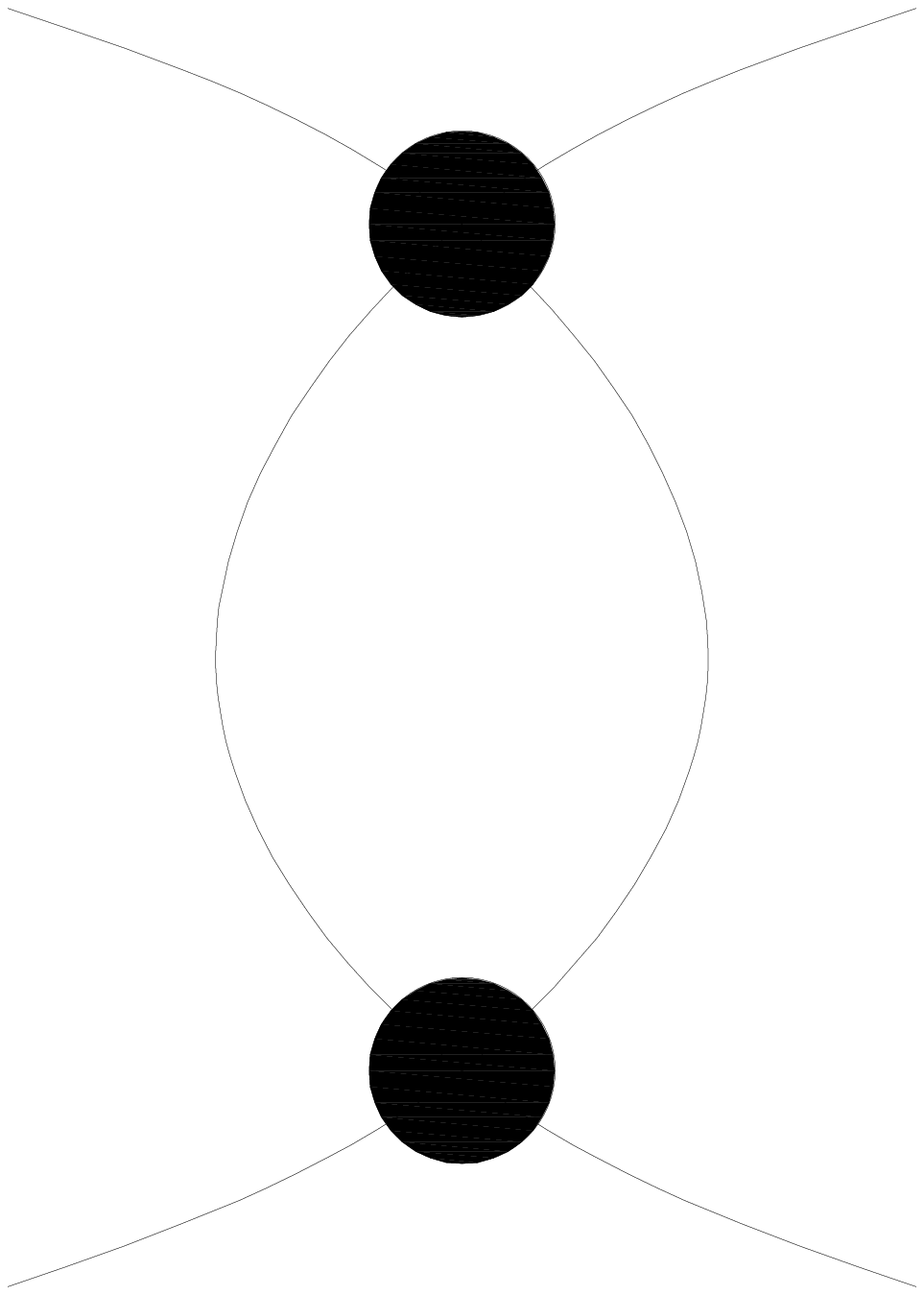}}}
\newcommand{\oneedge}{\raisebox{-0.25\height}{\includegraphics[width=0.8cm]{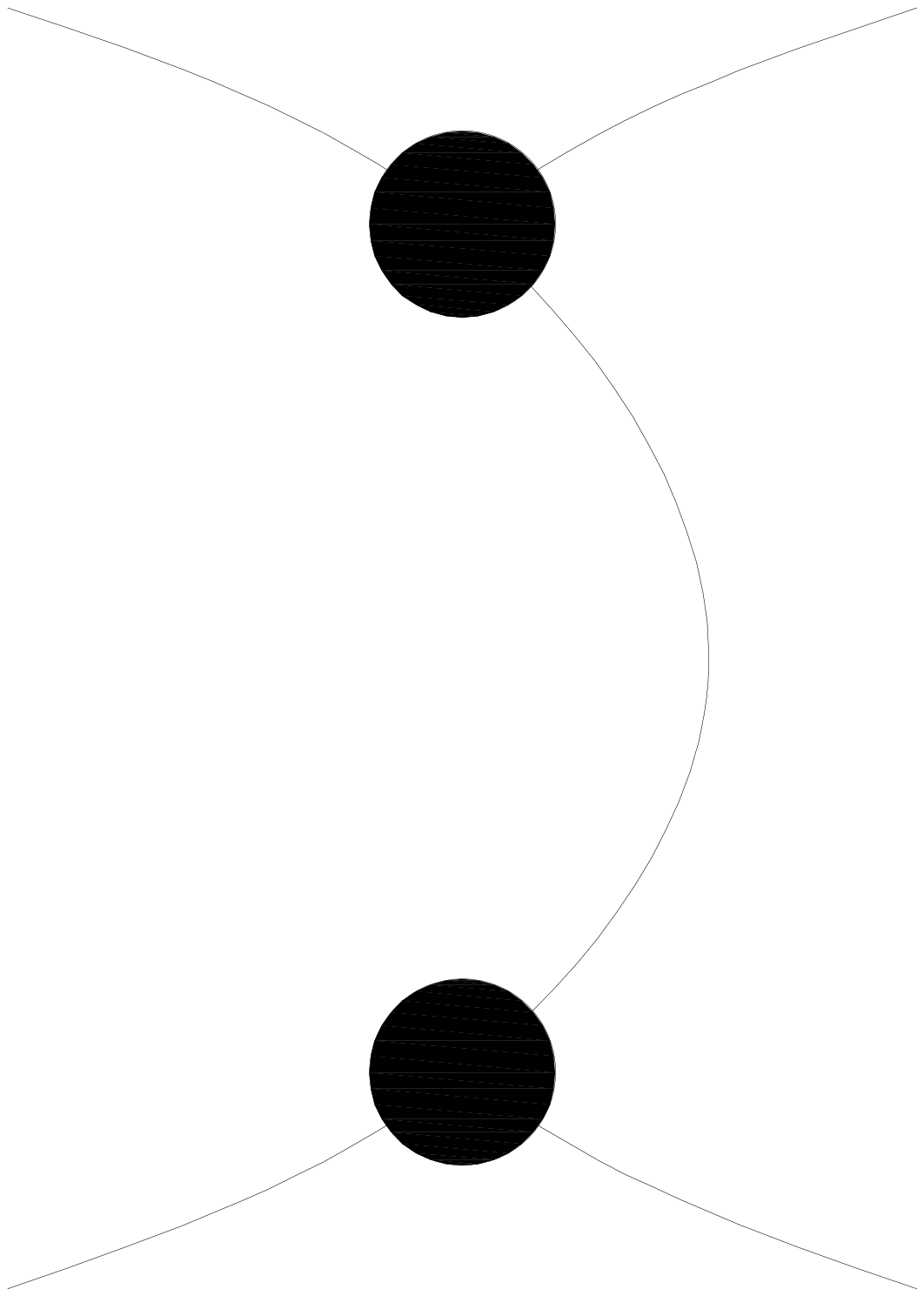}}}
\newcommand{\markedge}{\raisebox{-0.25\height}{\includegraphics[width=0.4cm]{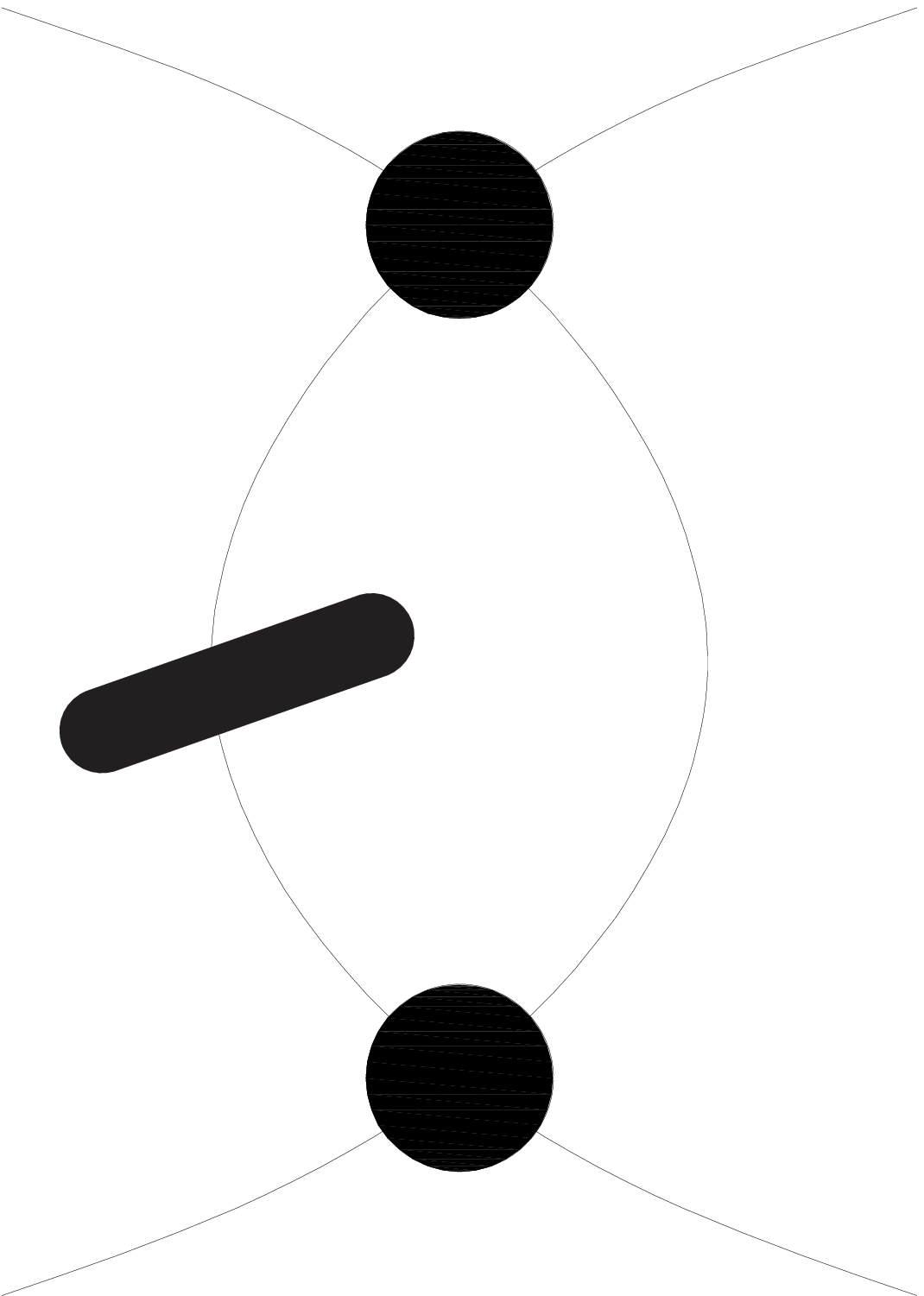}}}
\newcommand{\Bsmoothvertex}{\raisebox{-0.25\height}{\includegraphics[width=0.8cm]{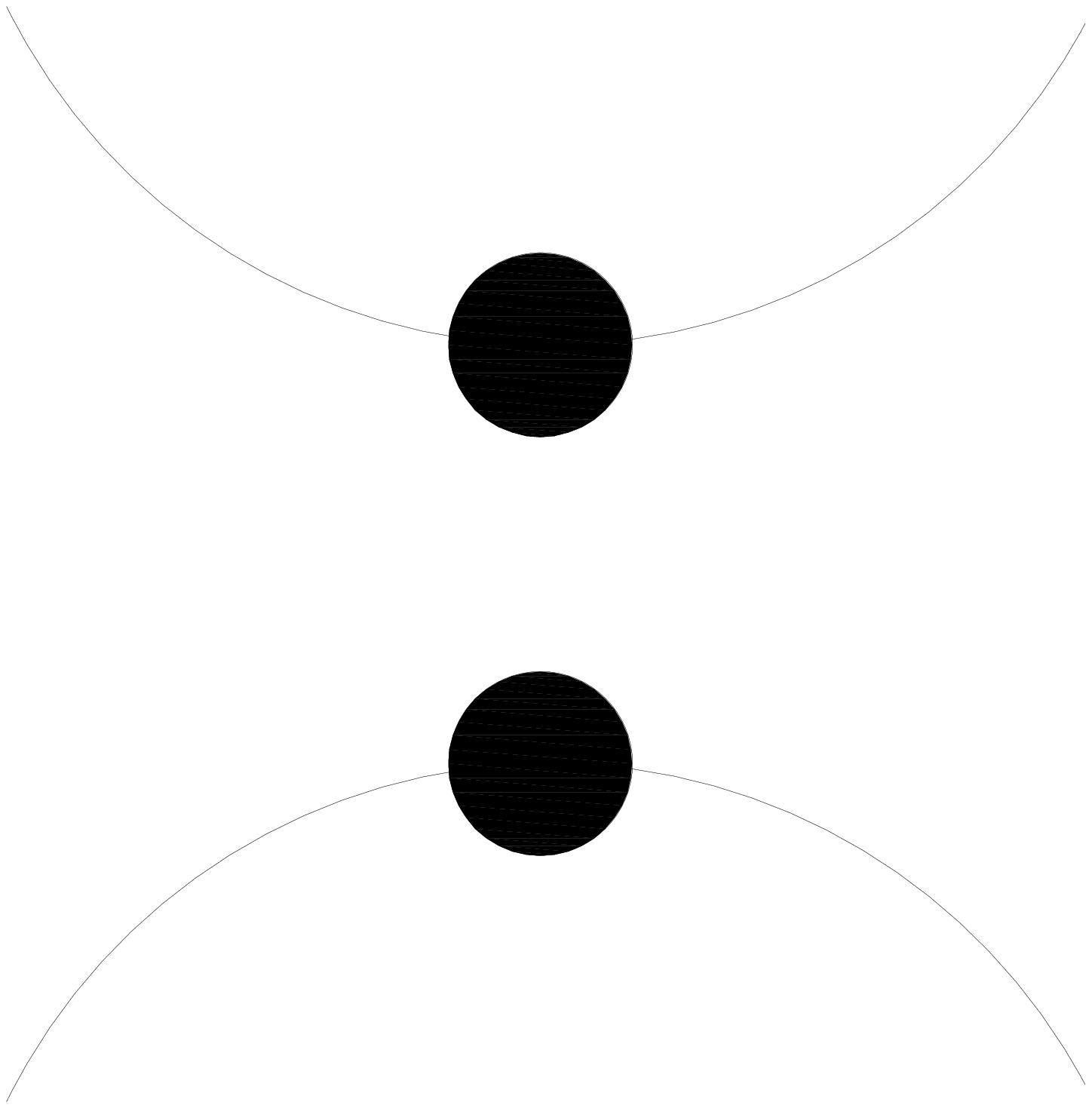}}}
\newcommand{\onerightmark}{\raisebox{-0.25\height}{\includegraphics[width=0.4cm]{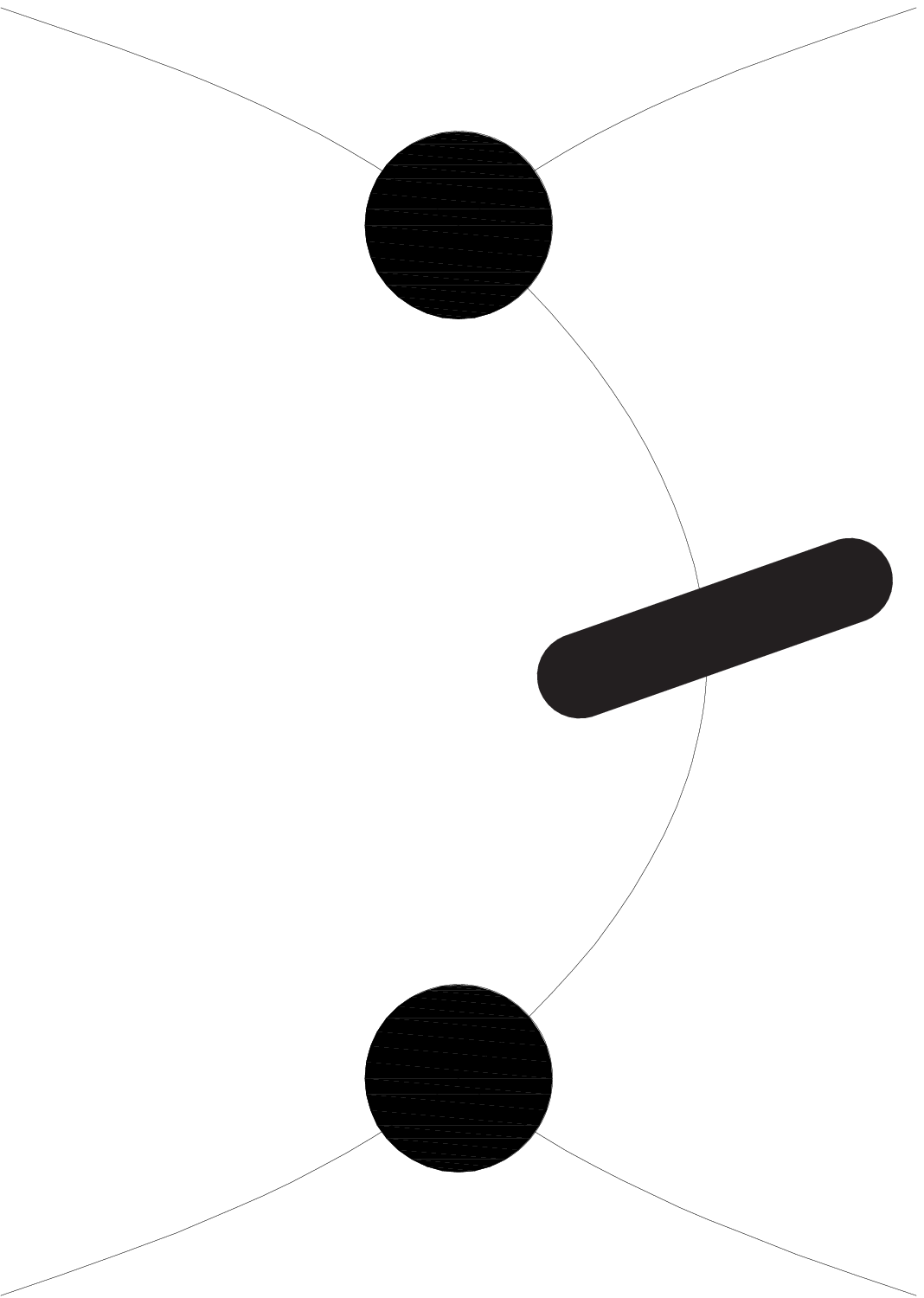}}}
\newcommand{\oneleftmark}{\raisebox{-0.25\height}{\includegraphics[width=0.4cm]{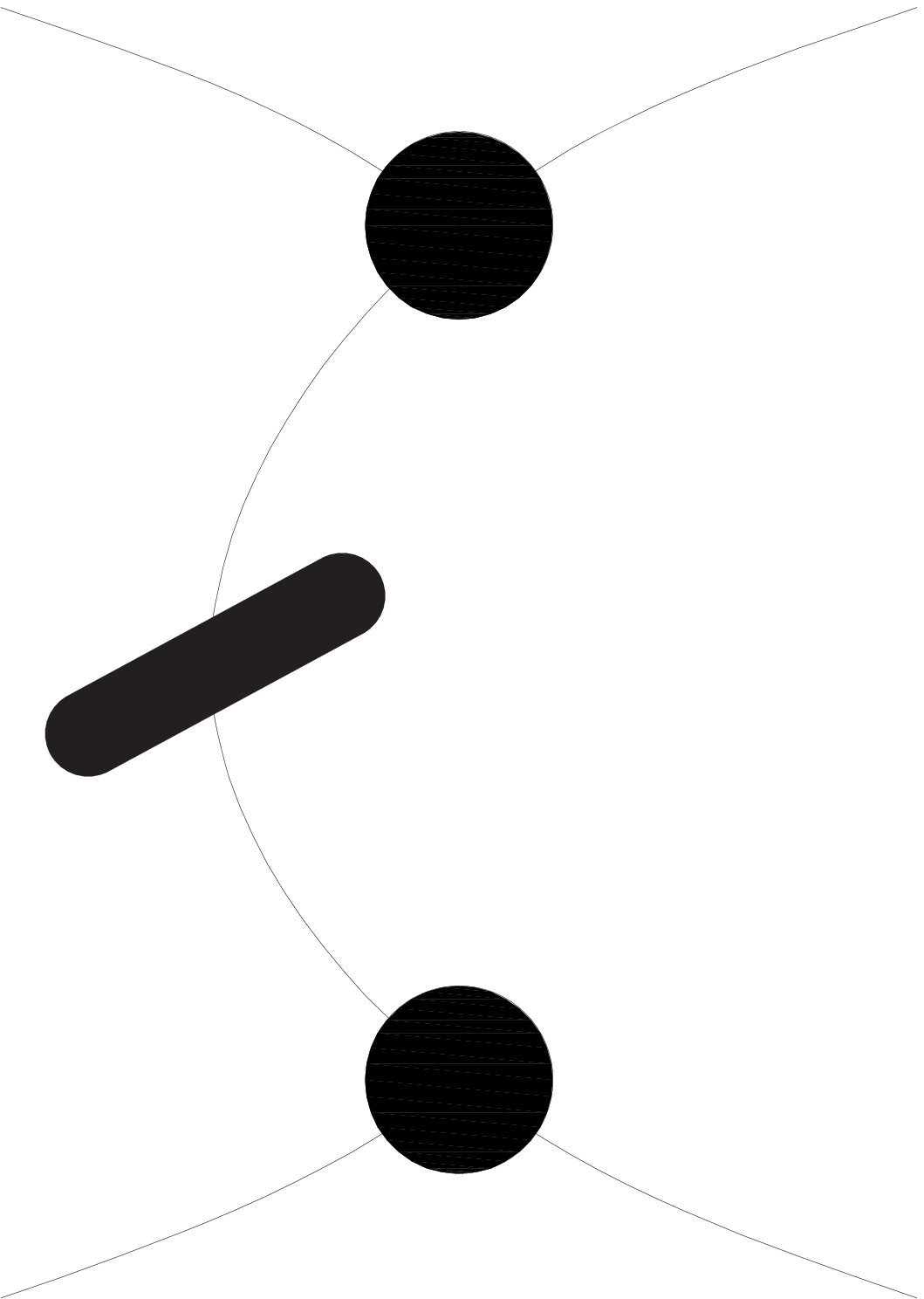}}}
\newcommand{\twomark}{\raisebox{-0.25\height}{\includegraphics[width=0.4cm]{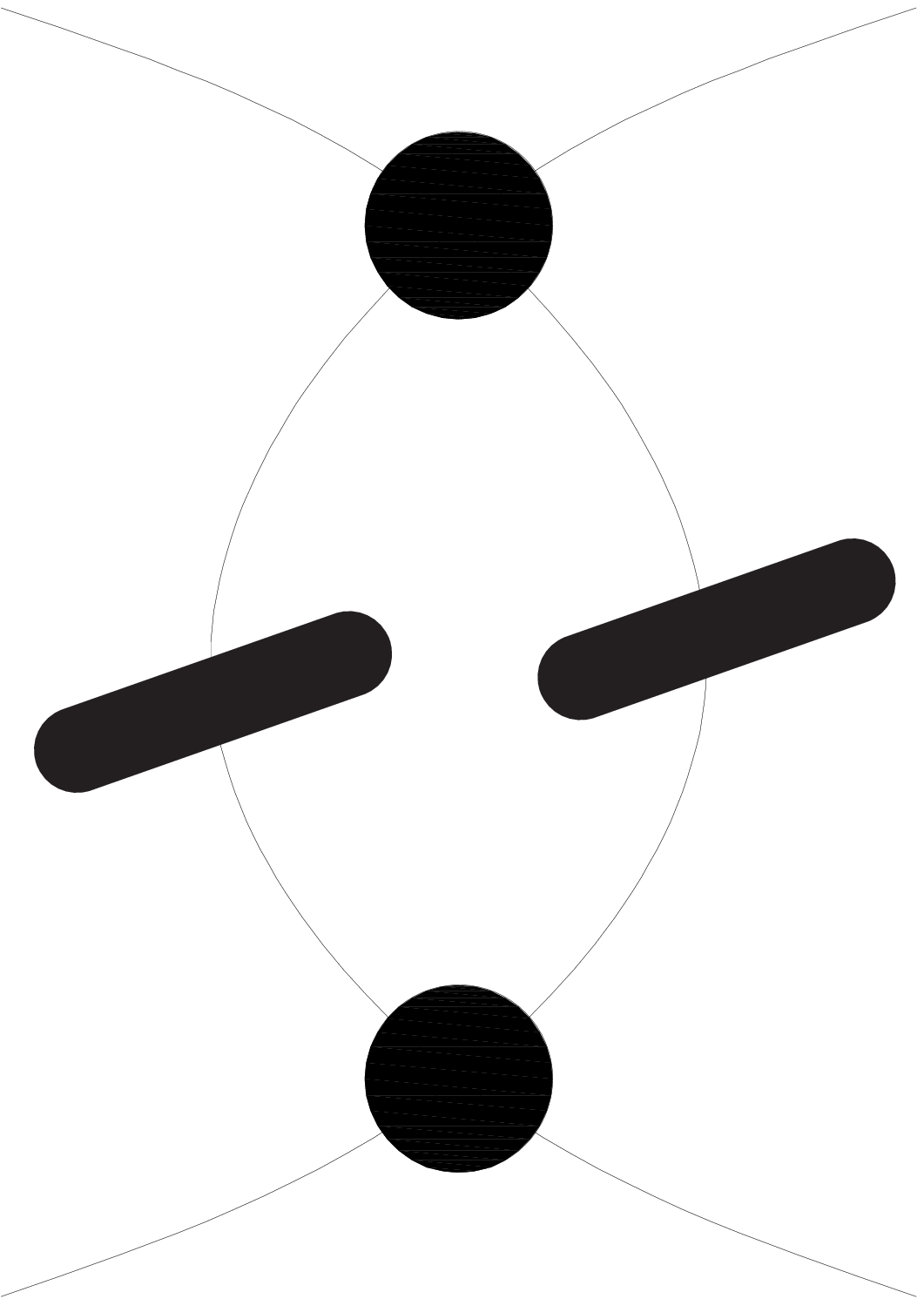}}}
\newcommand{\loopvertex}{\raisebox{-0.25\height}{\includegraphics[width=0.8cm]{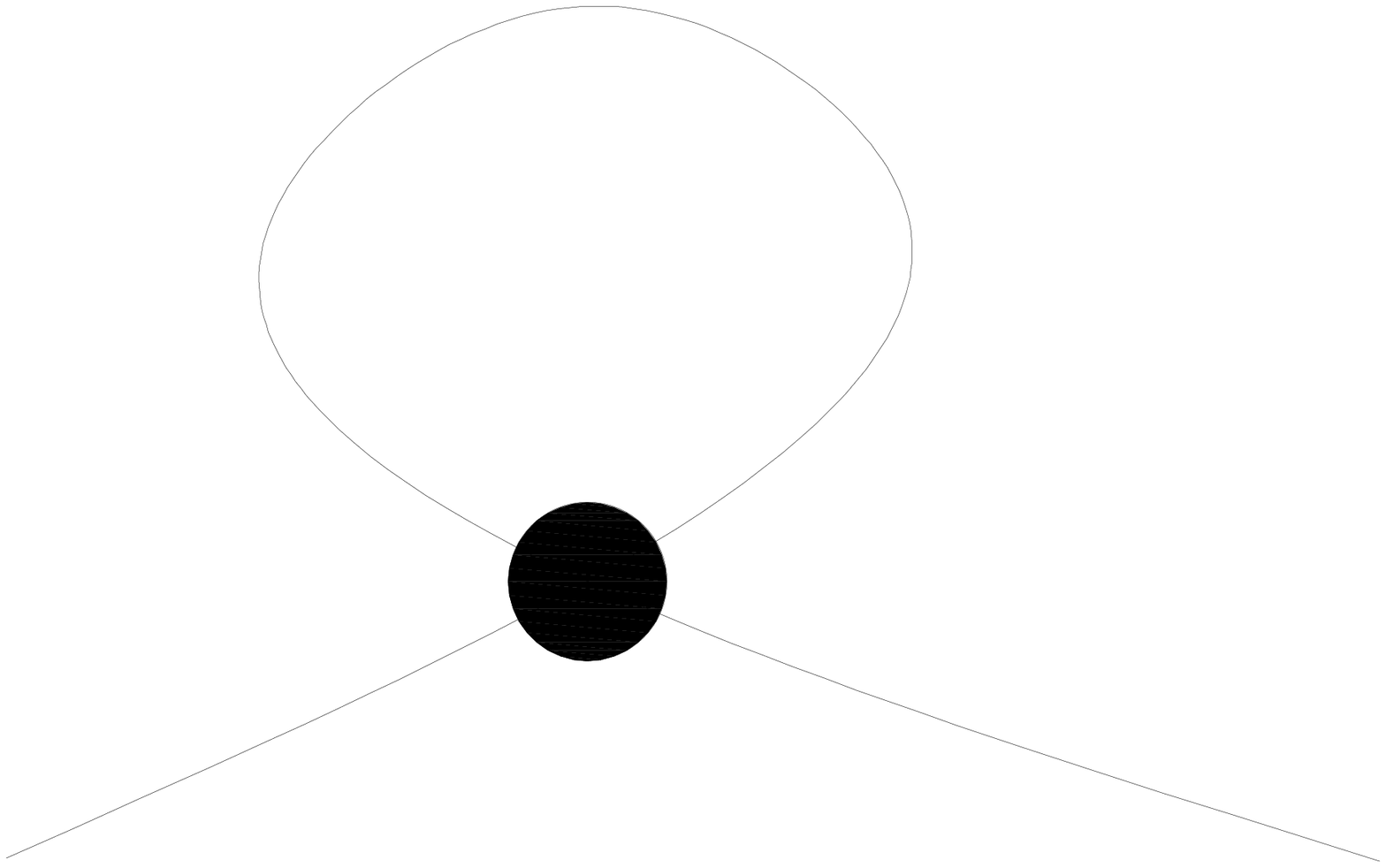}}}
\newcommand{\downarc}{\raisebox{-0.25\height}{\includegraphics[width=0.8cm]{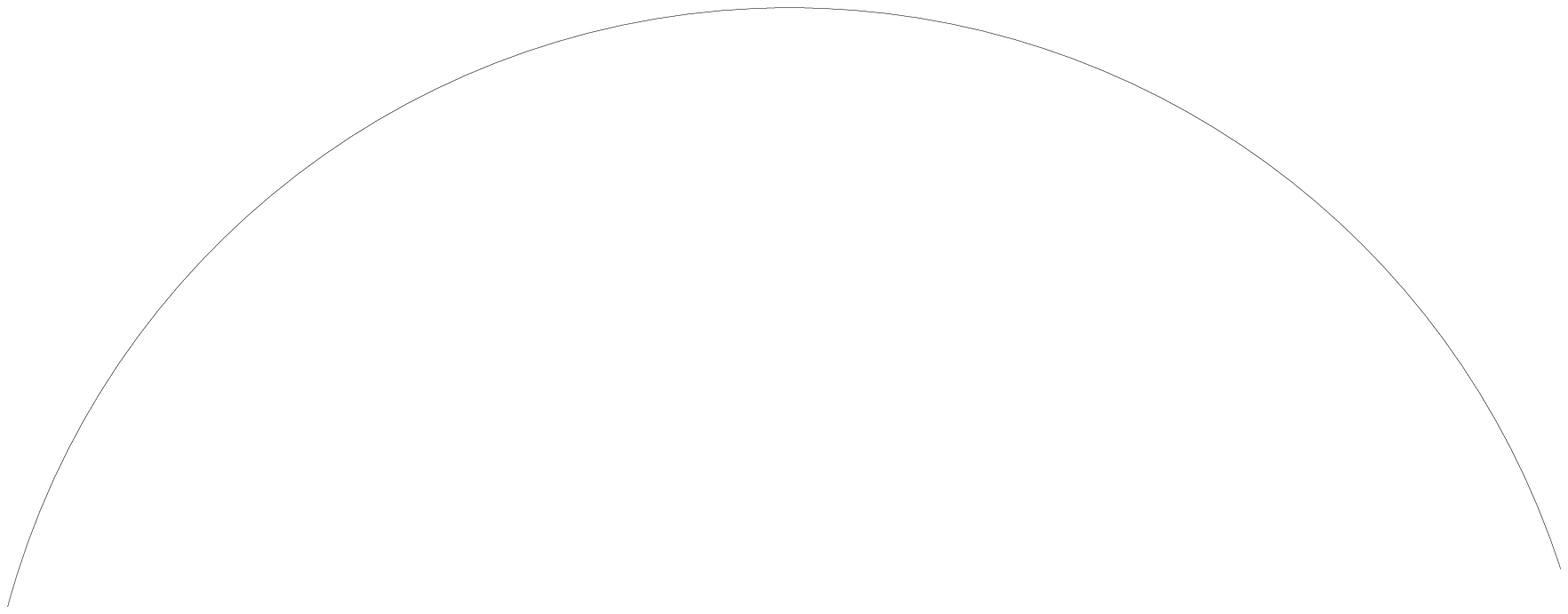}}}

\newcommand{\threevertex}{\raisebox{-0.25\height}{\includegraphics[width=0.8cm]{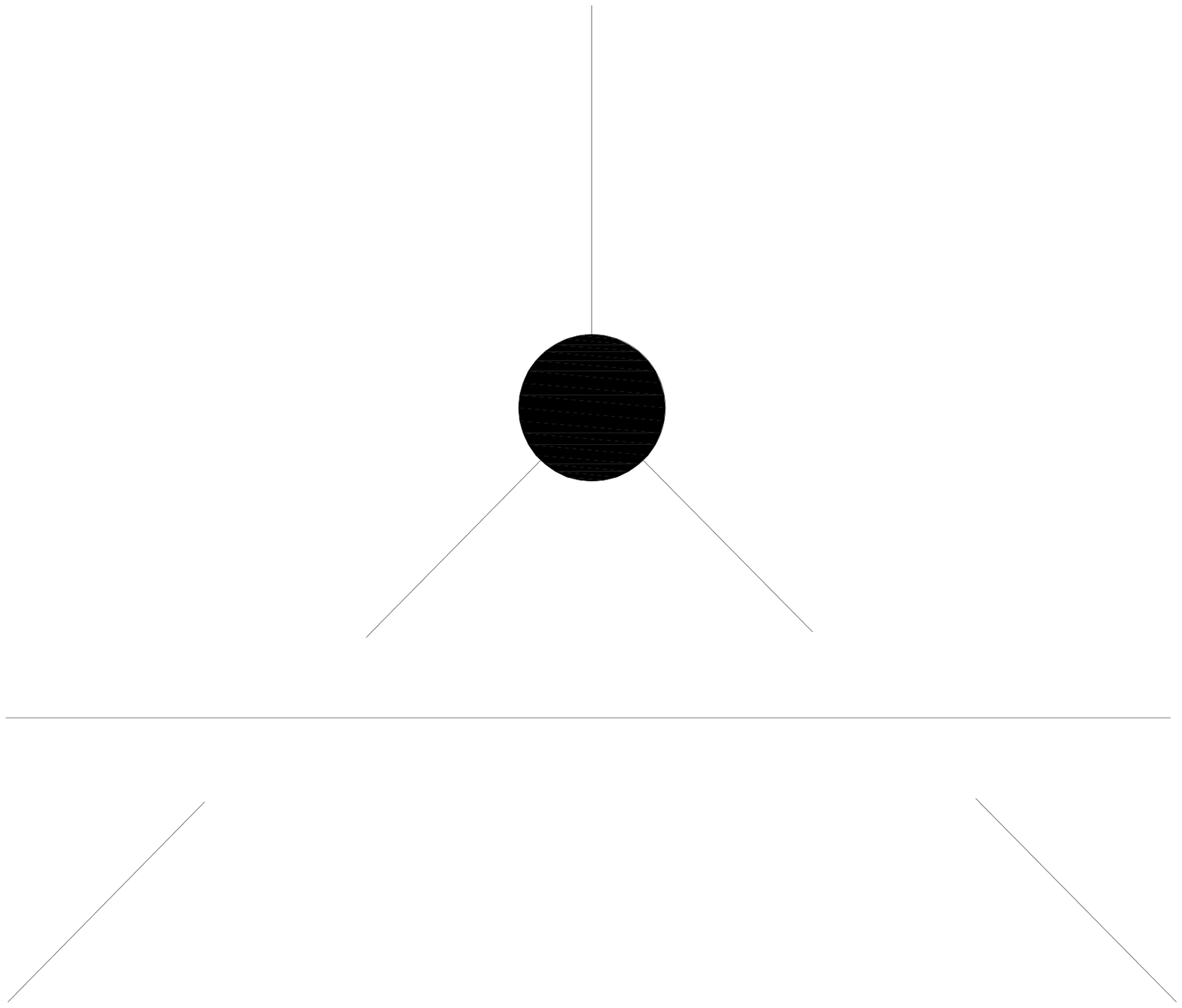}}}
\newcommand{\aav}{\raisebox{-0.25\height}{\includegraphics[width=0.8cm]{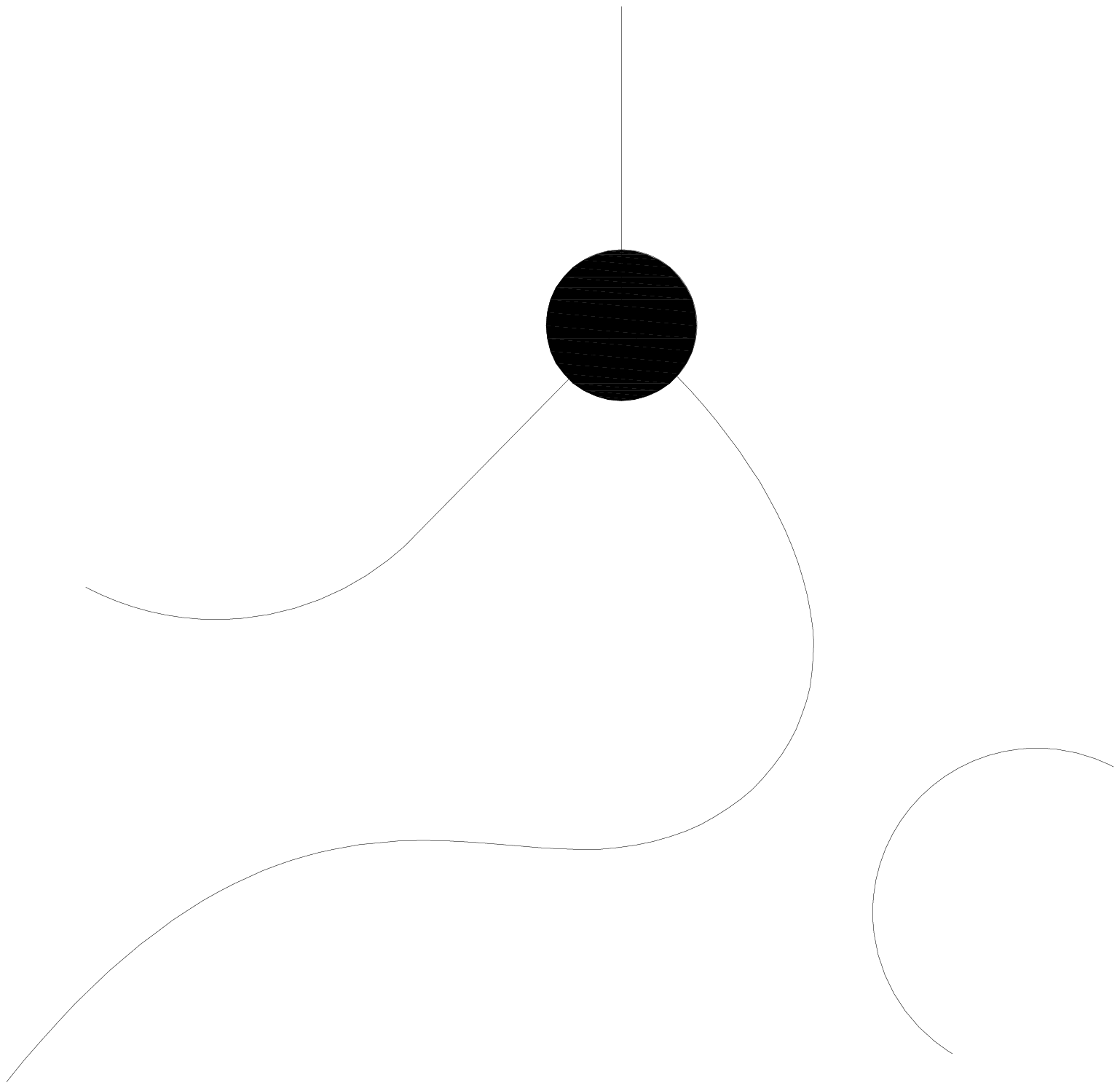}}}
\newcommand{\bbv}{\raisebox{-0.25\height}{\includegraphics[width=0.8cm]{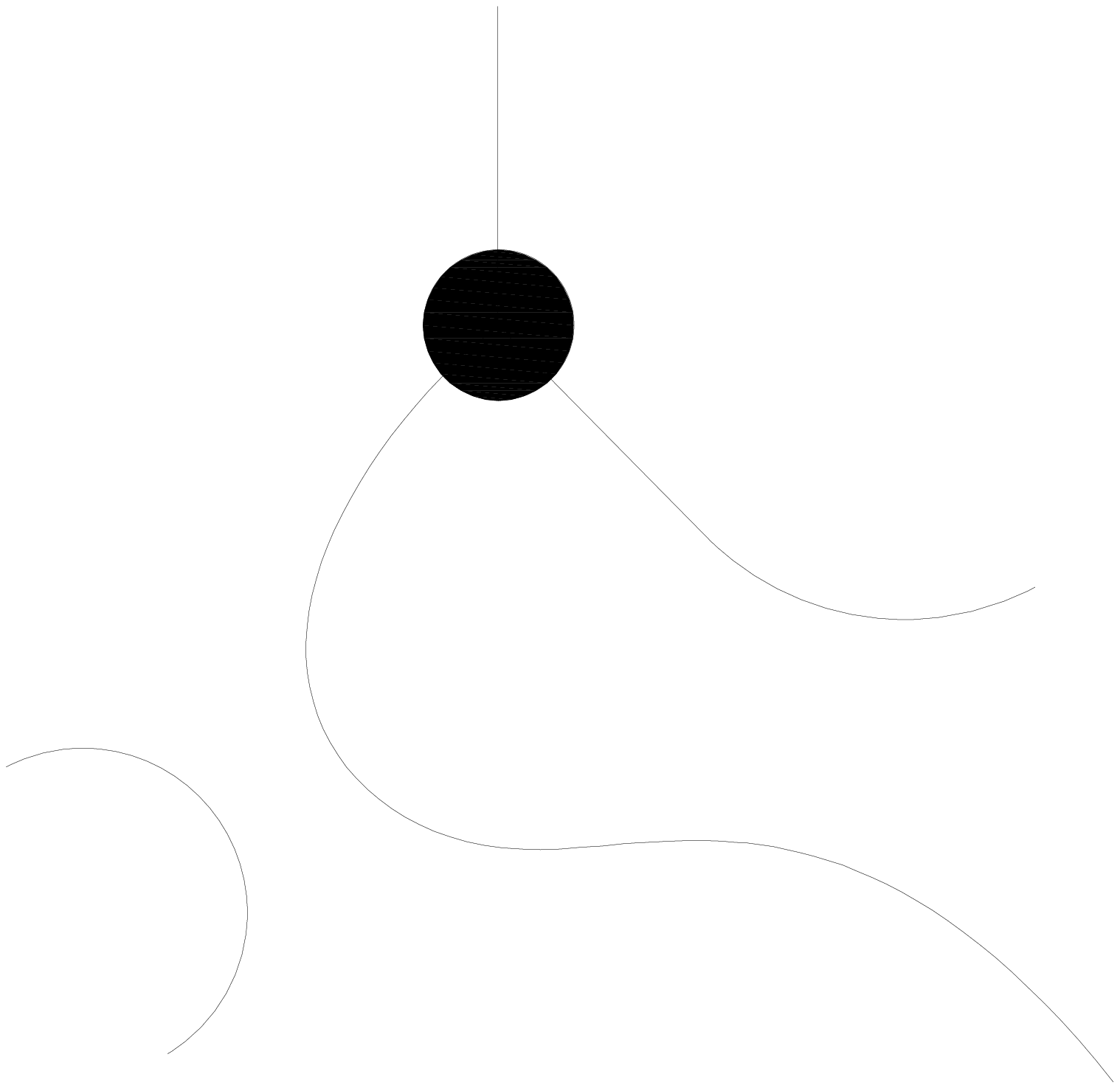}}}
\newcommand{\abv}{\raisebox{-0.25\height}{\includegraphics[width=0.8cm]{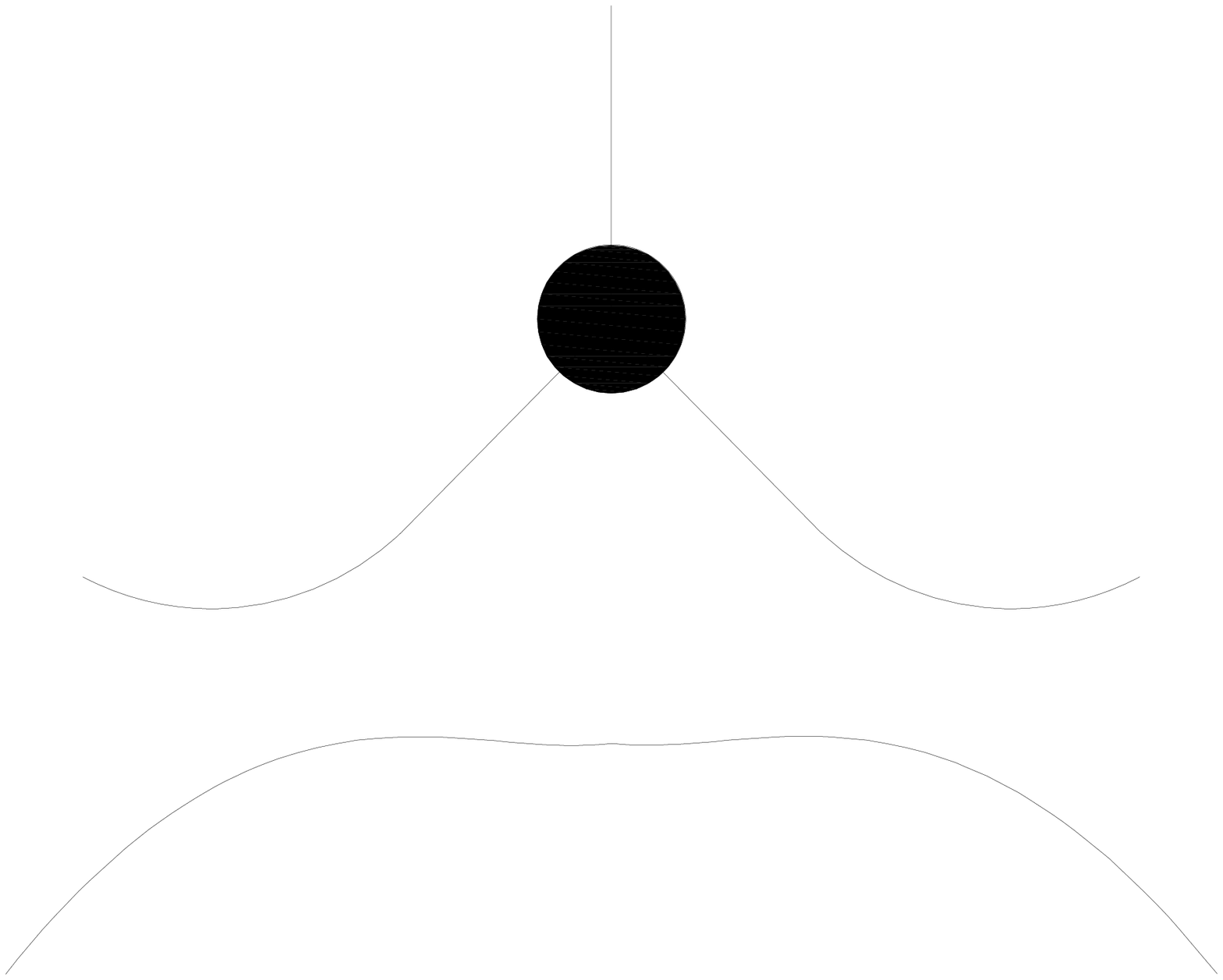}}}
\newcommand{\bav}{\raisebox{-0.25\height}{\includegraphics[width=0.8cm]{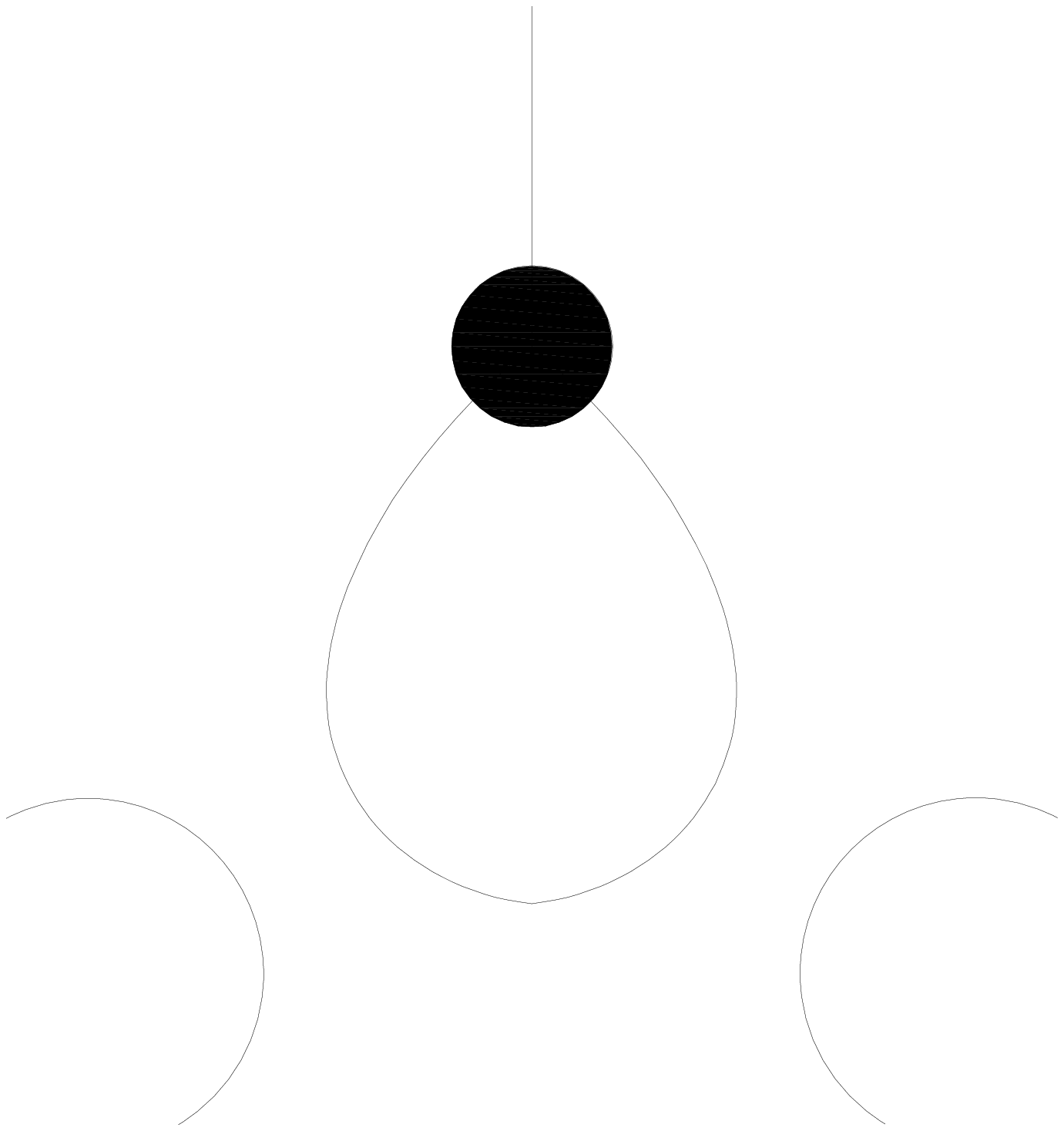}}}
\newcommand{\av}{\raisebox{-0.25\height}{\includegraphics[width=0.8cm]{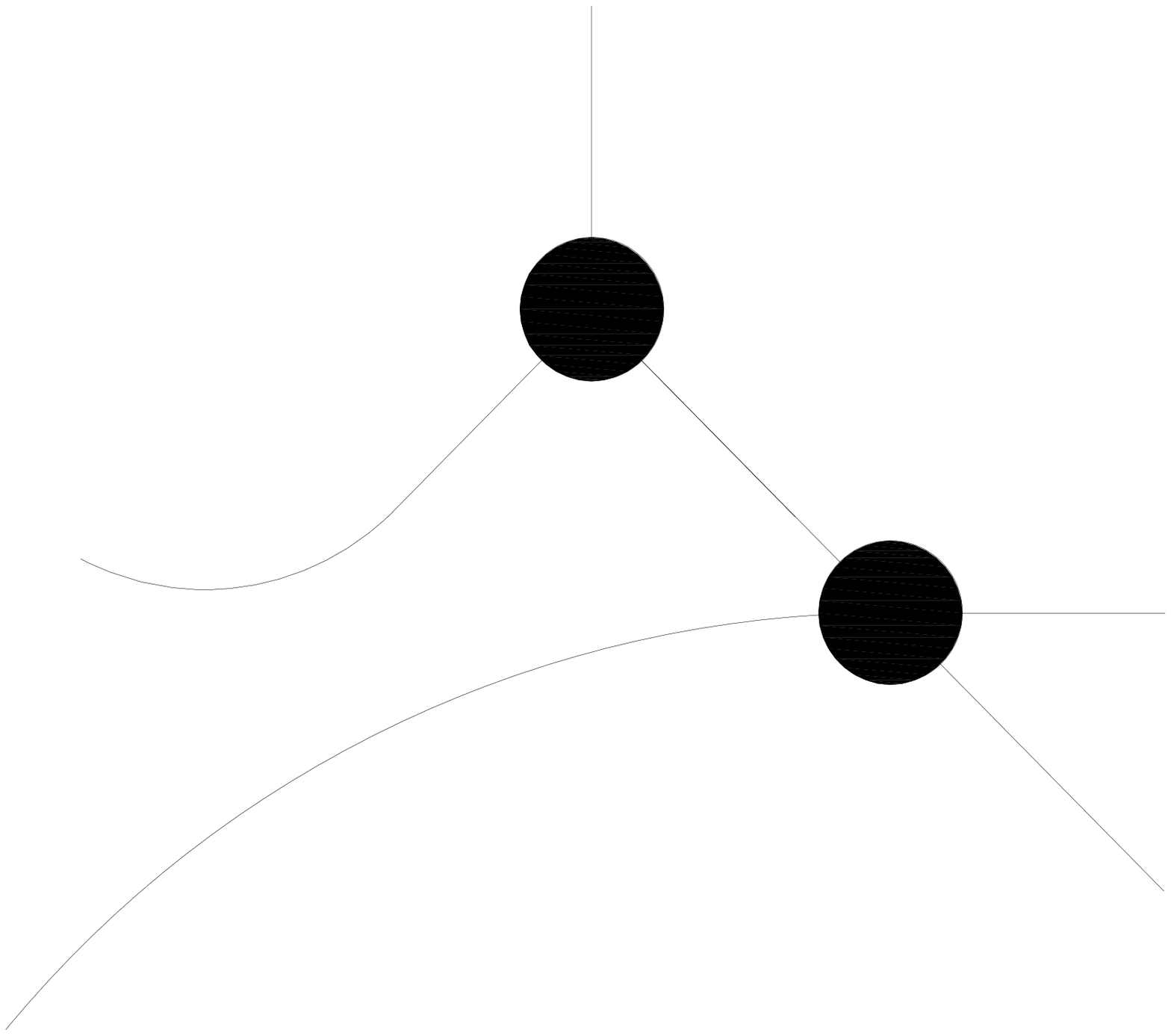}}}
\newcommand{\bv}{\raisebox{-0.25\height}{\includegraphics[width=0.8cm]{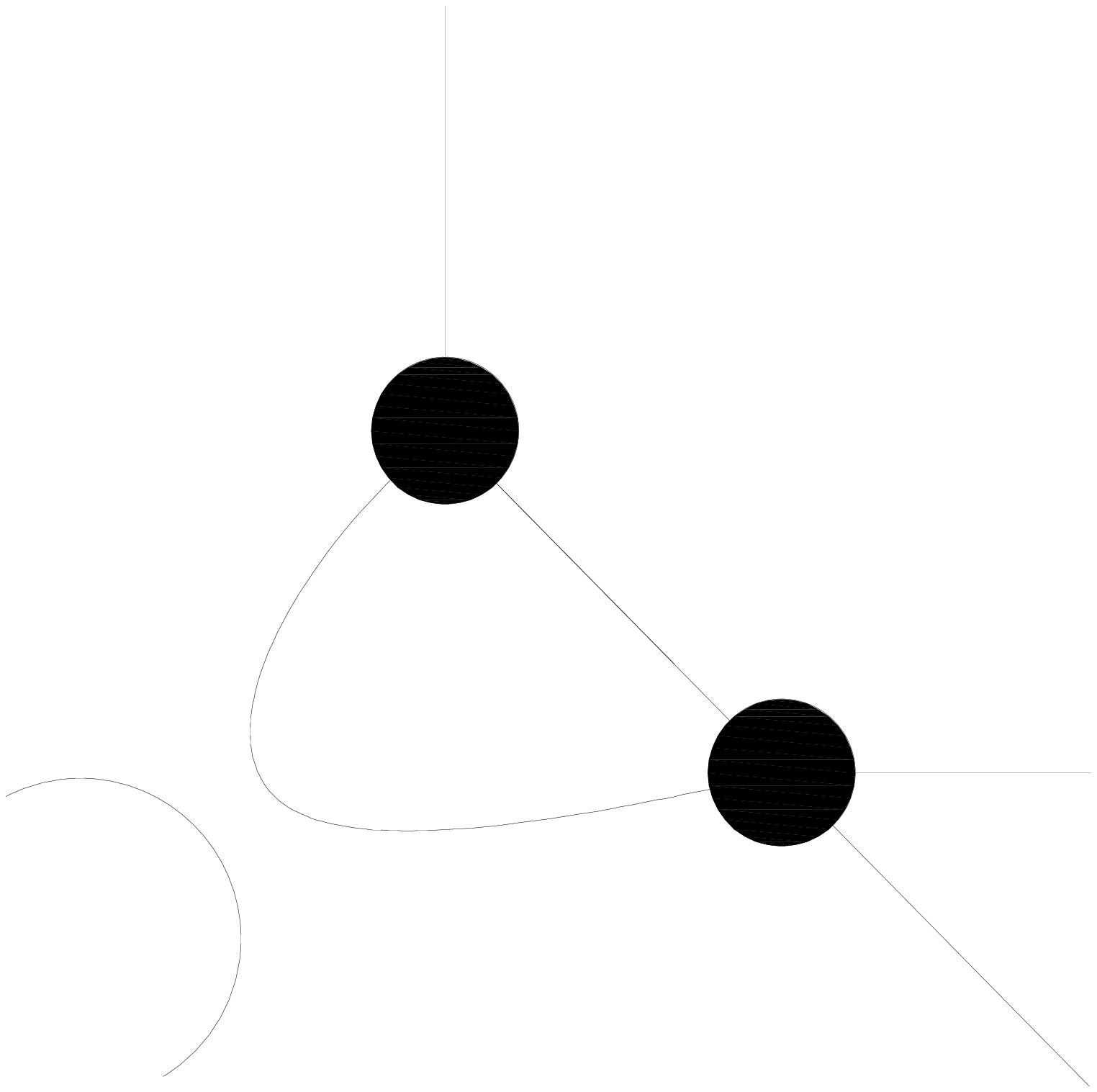}}}
\newcommand{\vb}{\raisebox{-0.25\height}{\includegraphics[width=0.8cm]{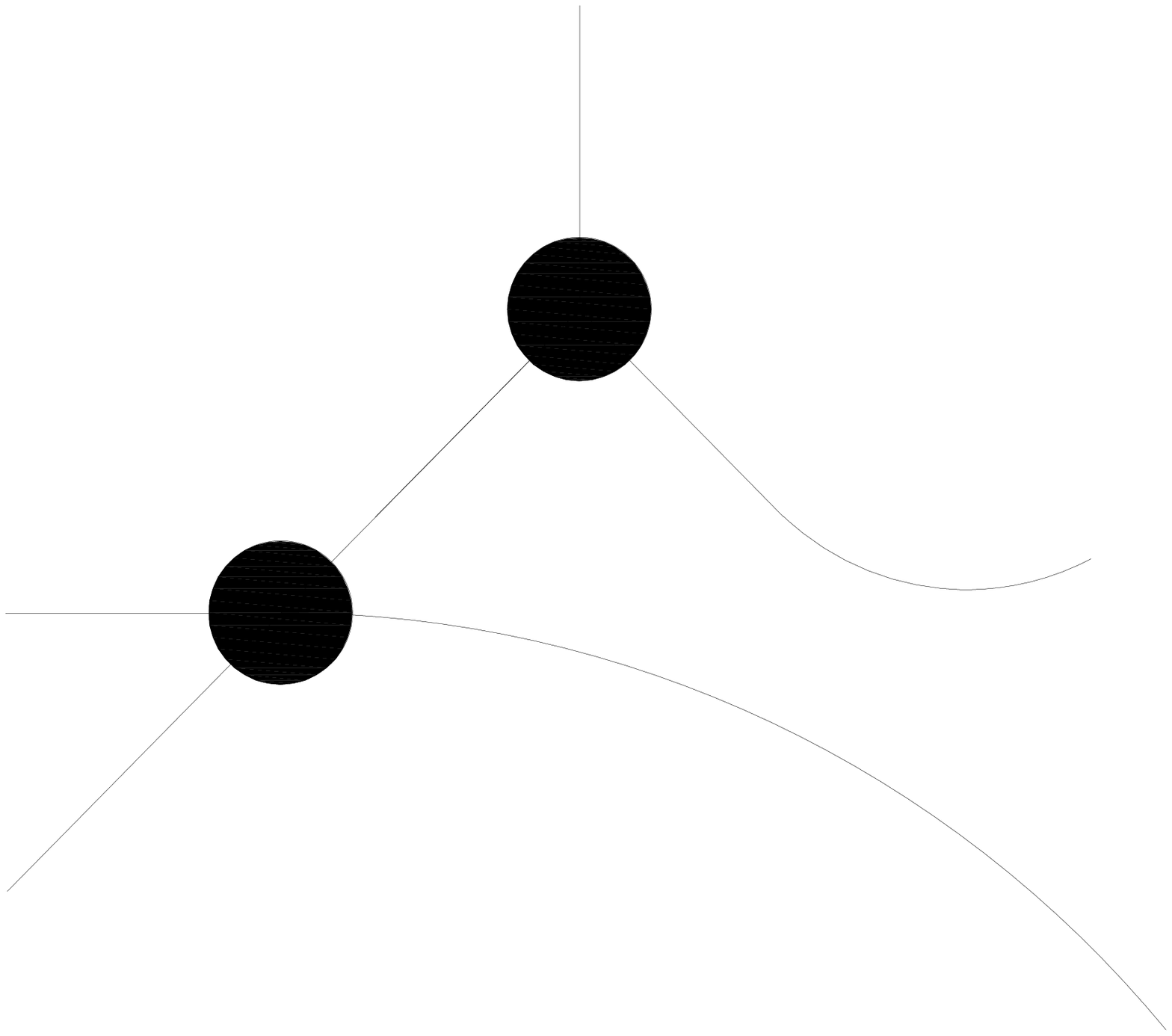}}}
\newcommand{\va}{\raisebox{-0.25\height}{\includegraphics[width=0.8cm]{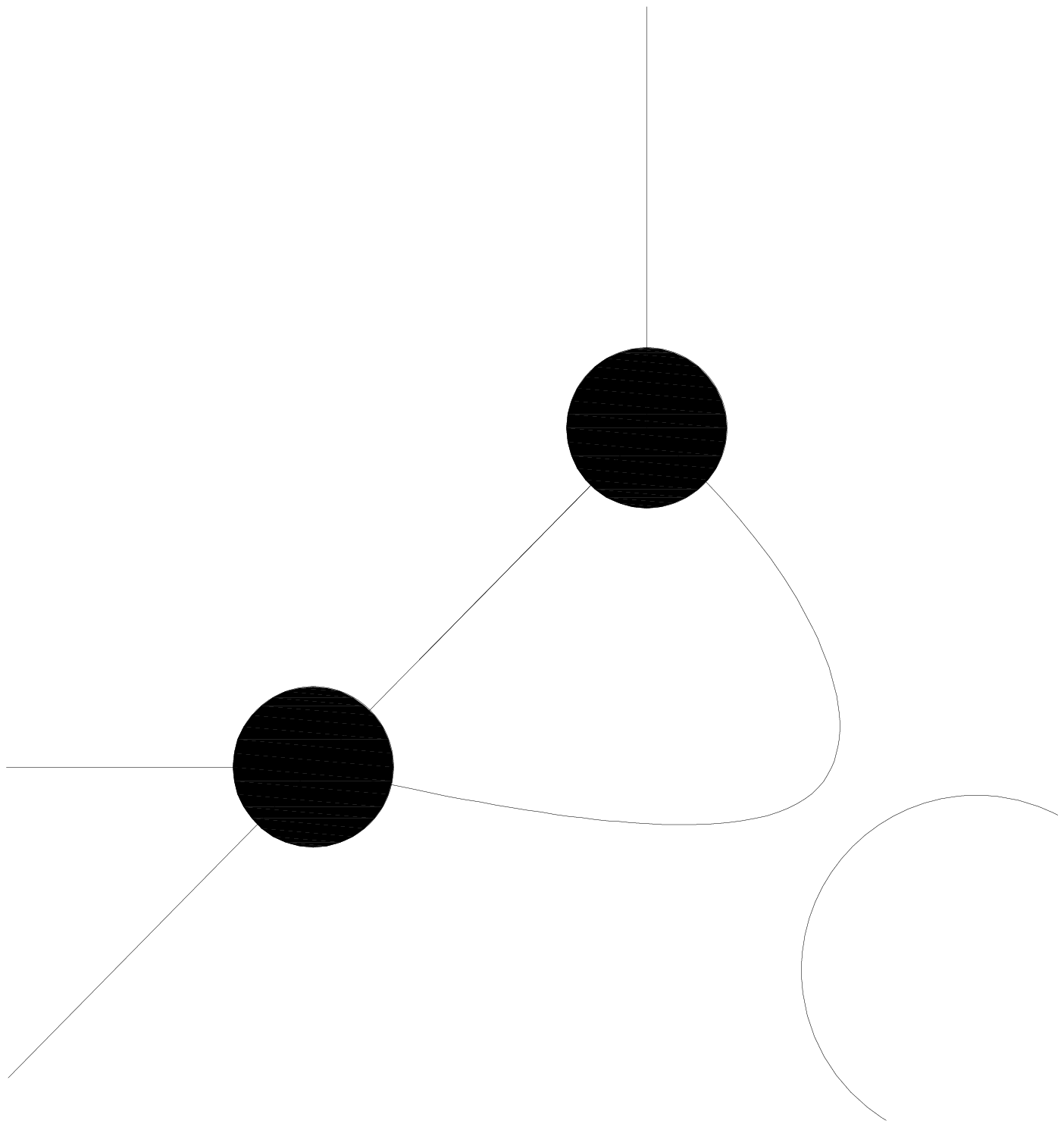}}}
\newcommand{\vv}{\raisebox{-0.25\height}{\includegraphics[width=0.8cm]{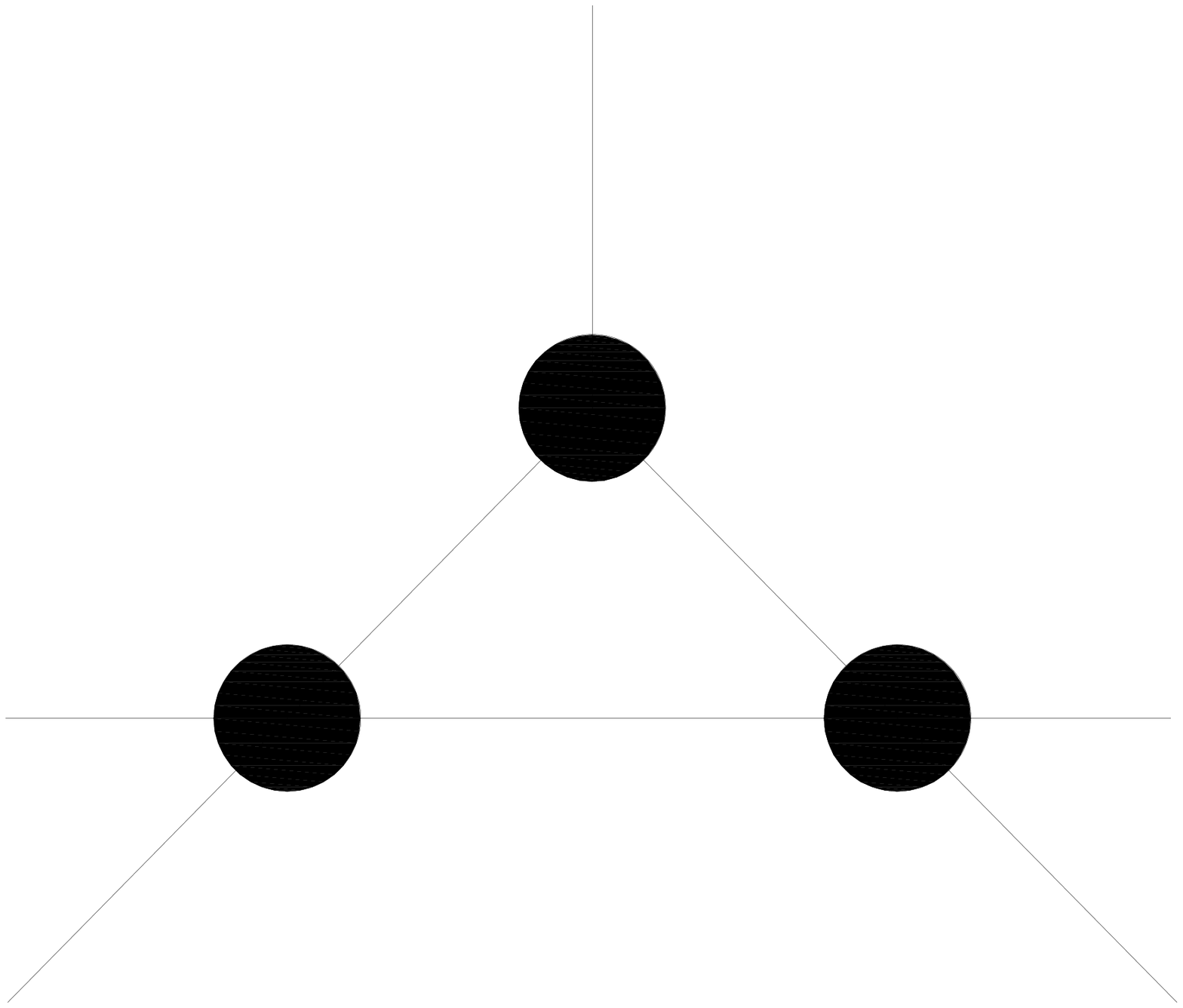}}}
\newcommand{\twothree}{\raisebox{-0.25\height}{\includegraphics[width=0.8cm]{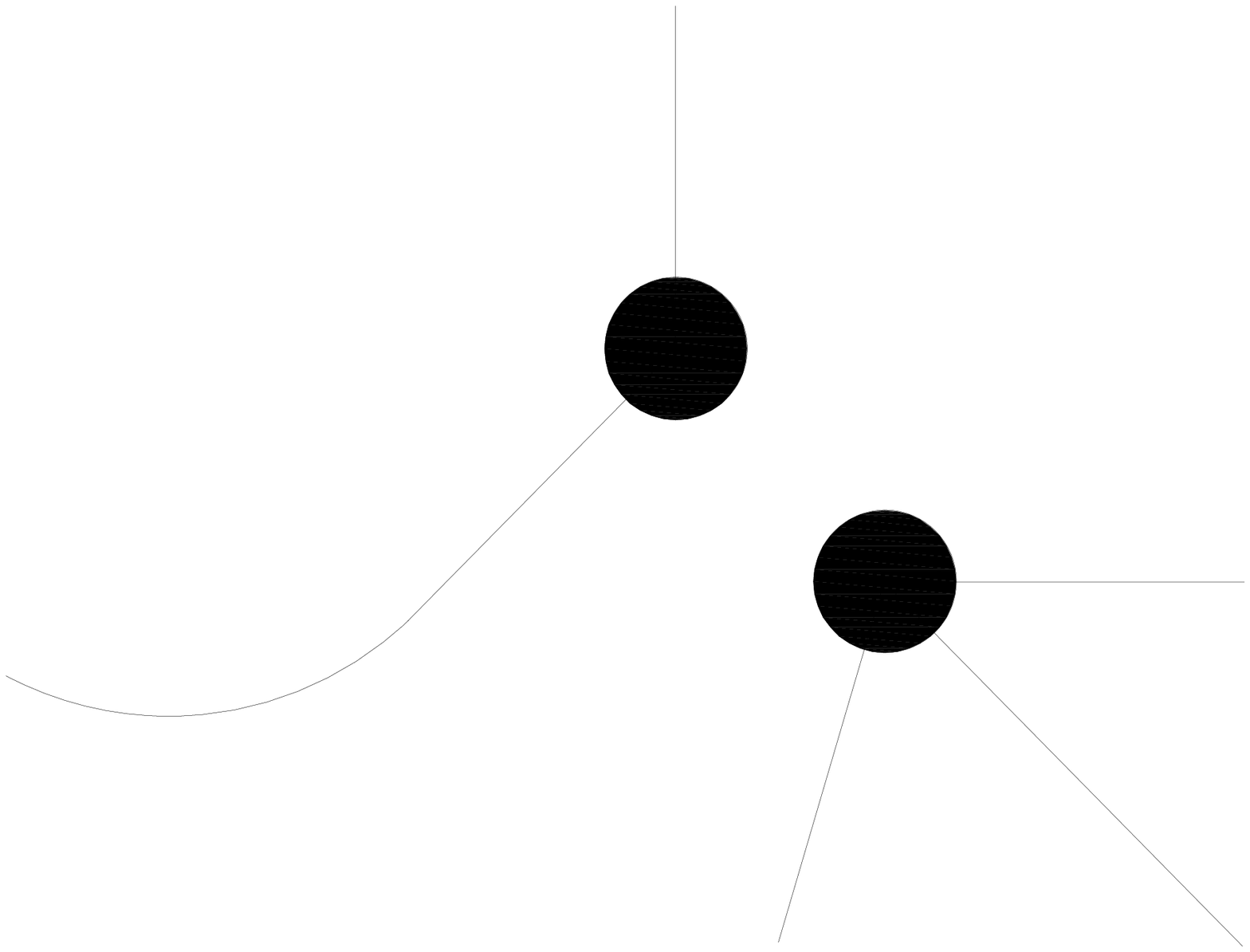}}}
\newcommand{\onearc}{\raisebox{-0.25\height}{\includegraphics[width=1.0cm]{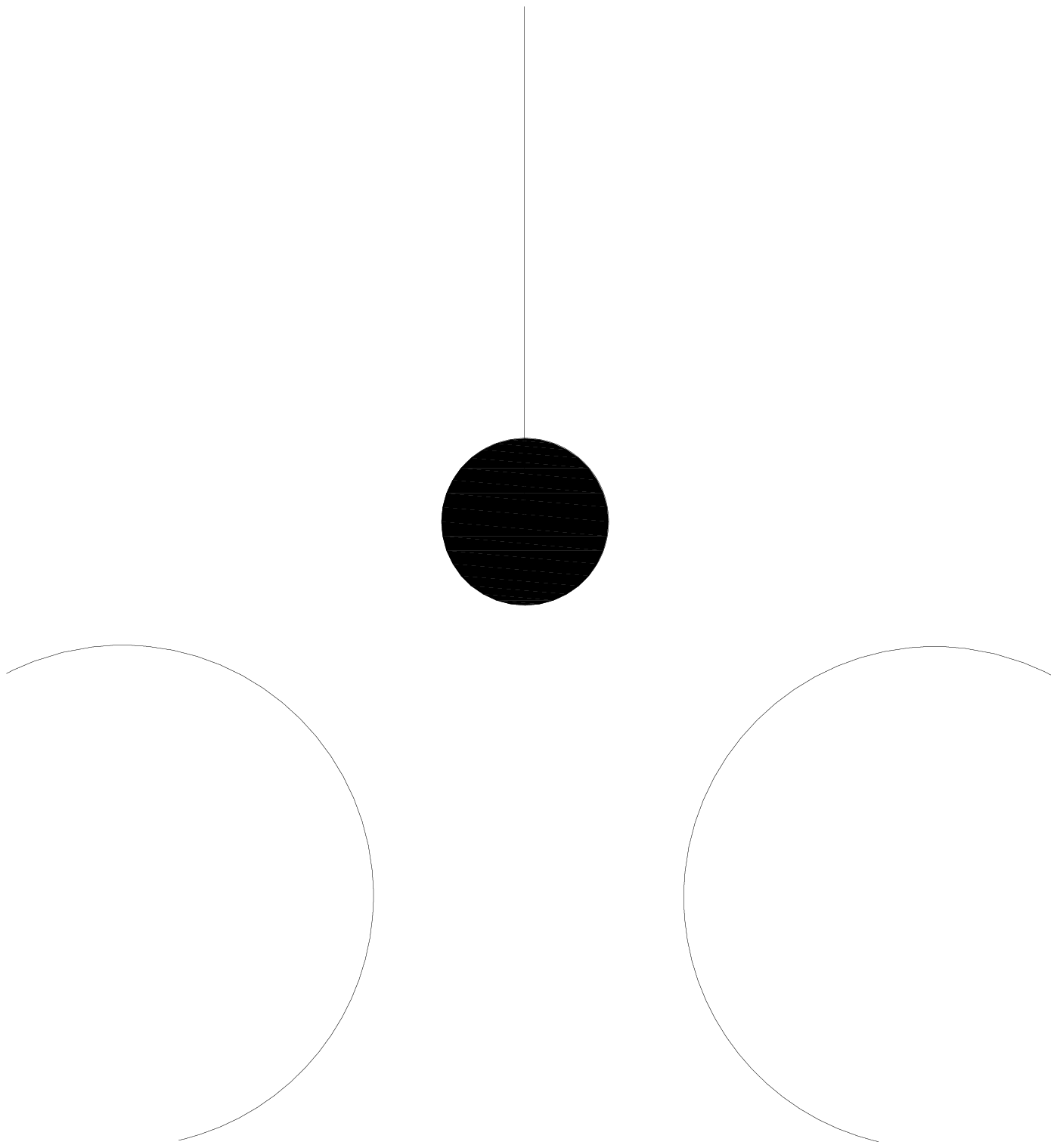}}}
\newcommand{\threetwo}{\raisebox{-0.25\height}{\includegraphics[width=1.0cm]{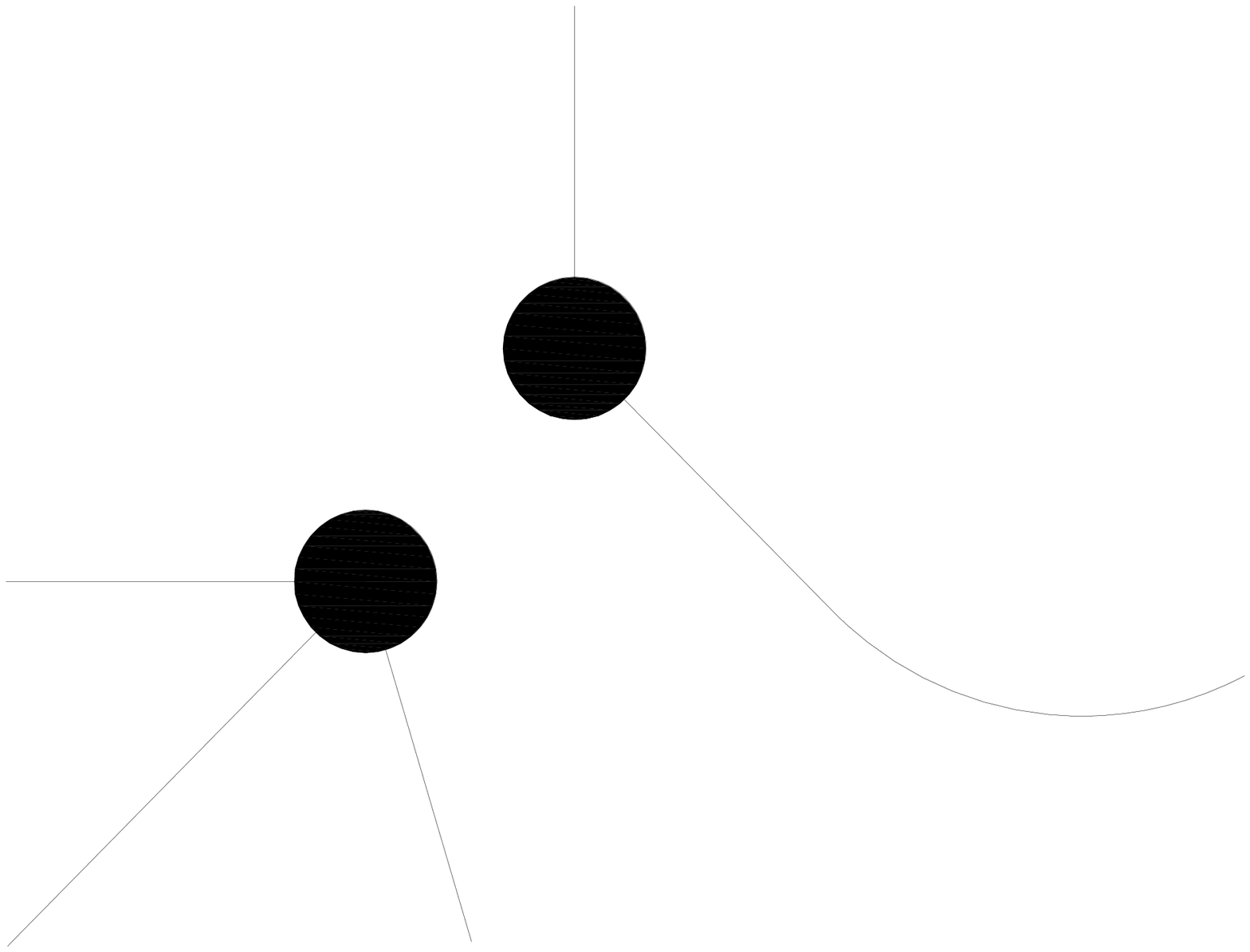}}}
\newcommand{\five}{\raisebox{-0.25\height}{\includegraphics[width=1.0cm]{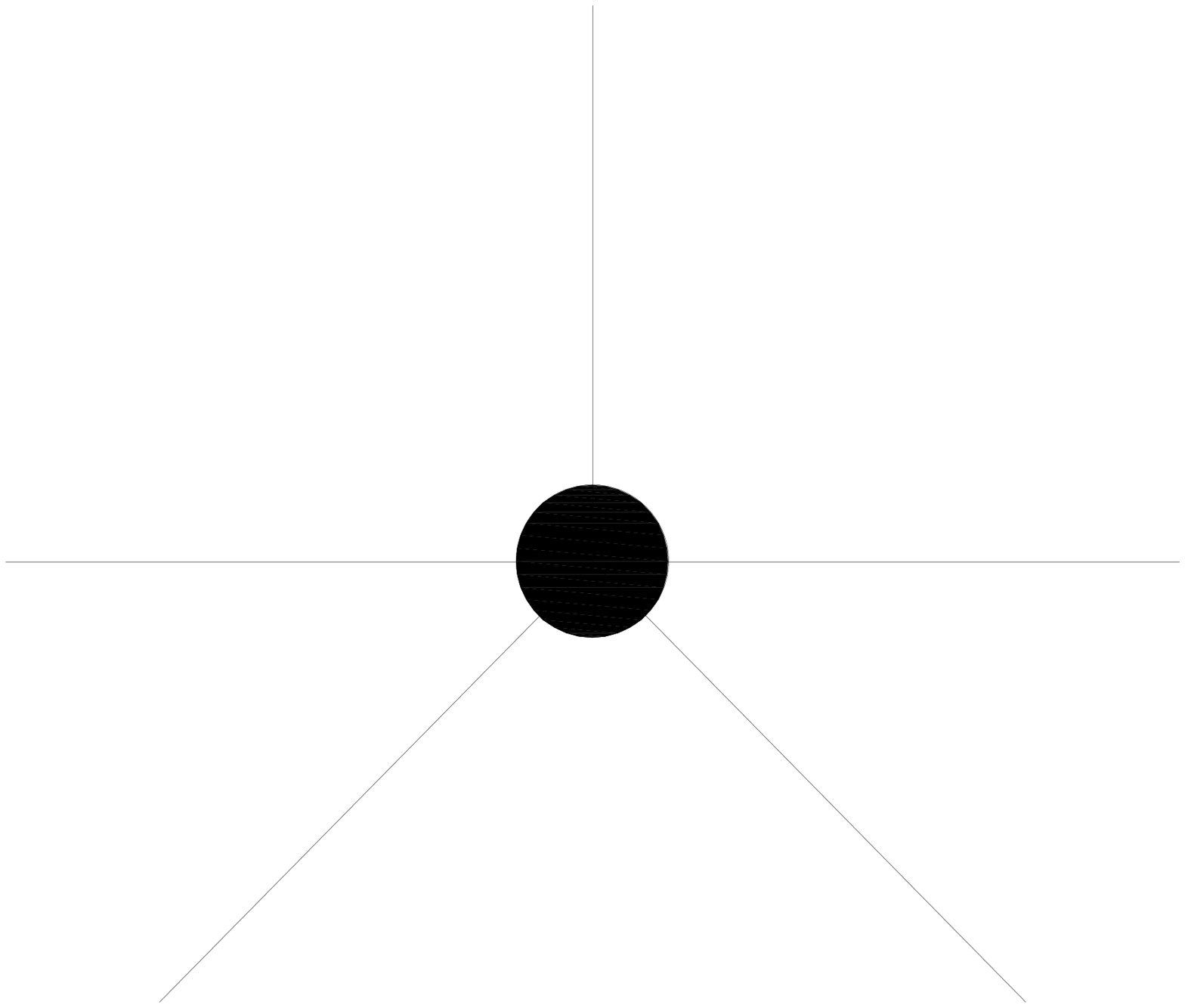}}}

\newcommand{\Gte}{\raisebox{-0.25\height}{\includegraphics[width=1.0cm]{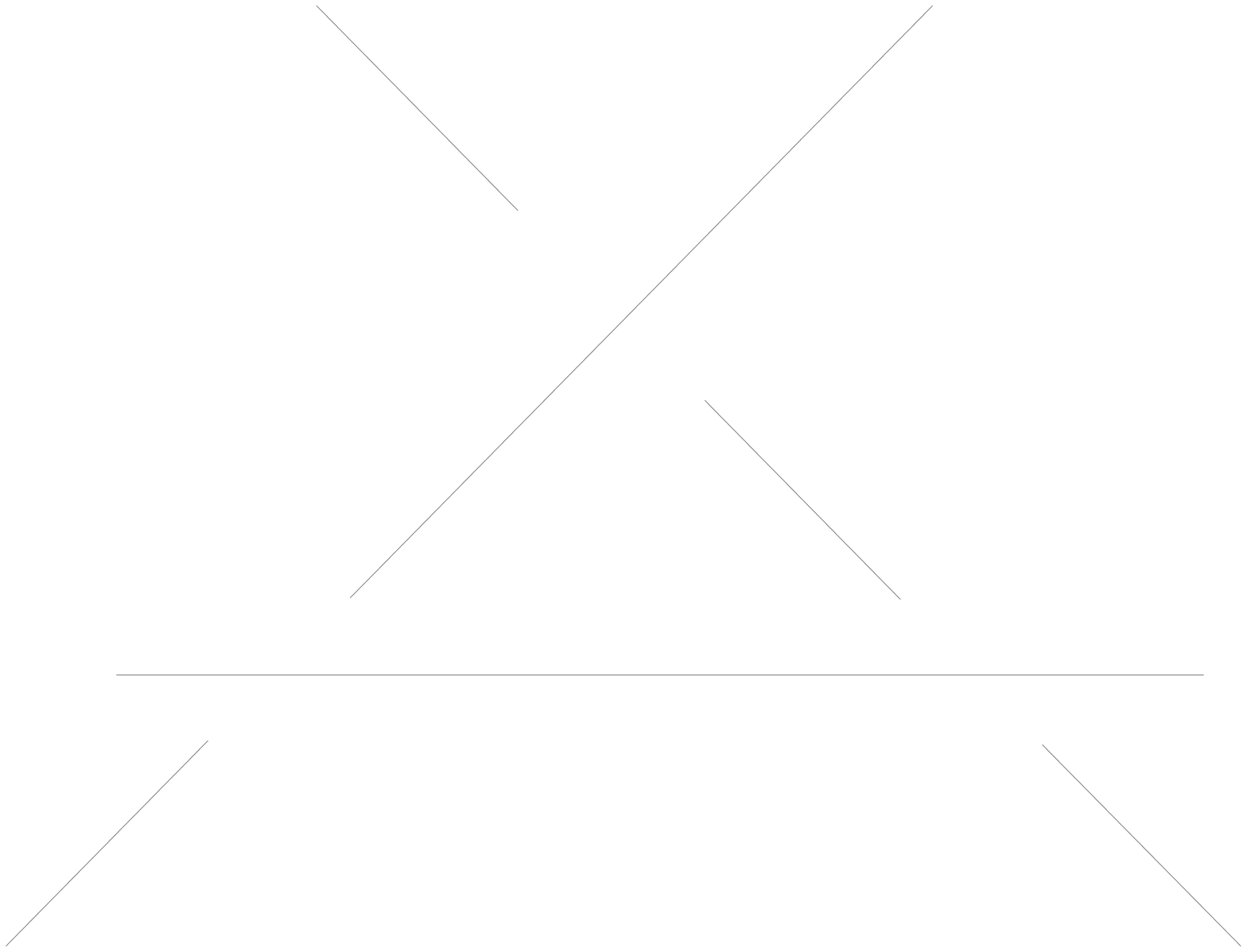}}}
\newcommand{\Gts}{\raisebox{-0.25\height}{\includegraphics[width=1.0cm]{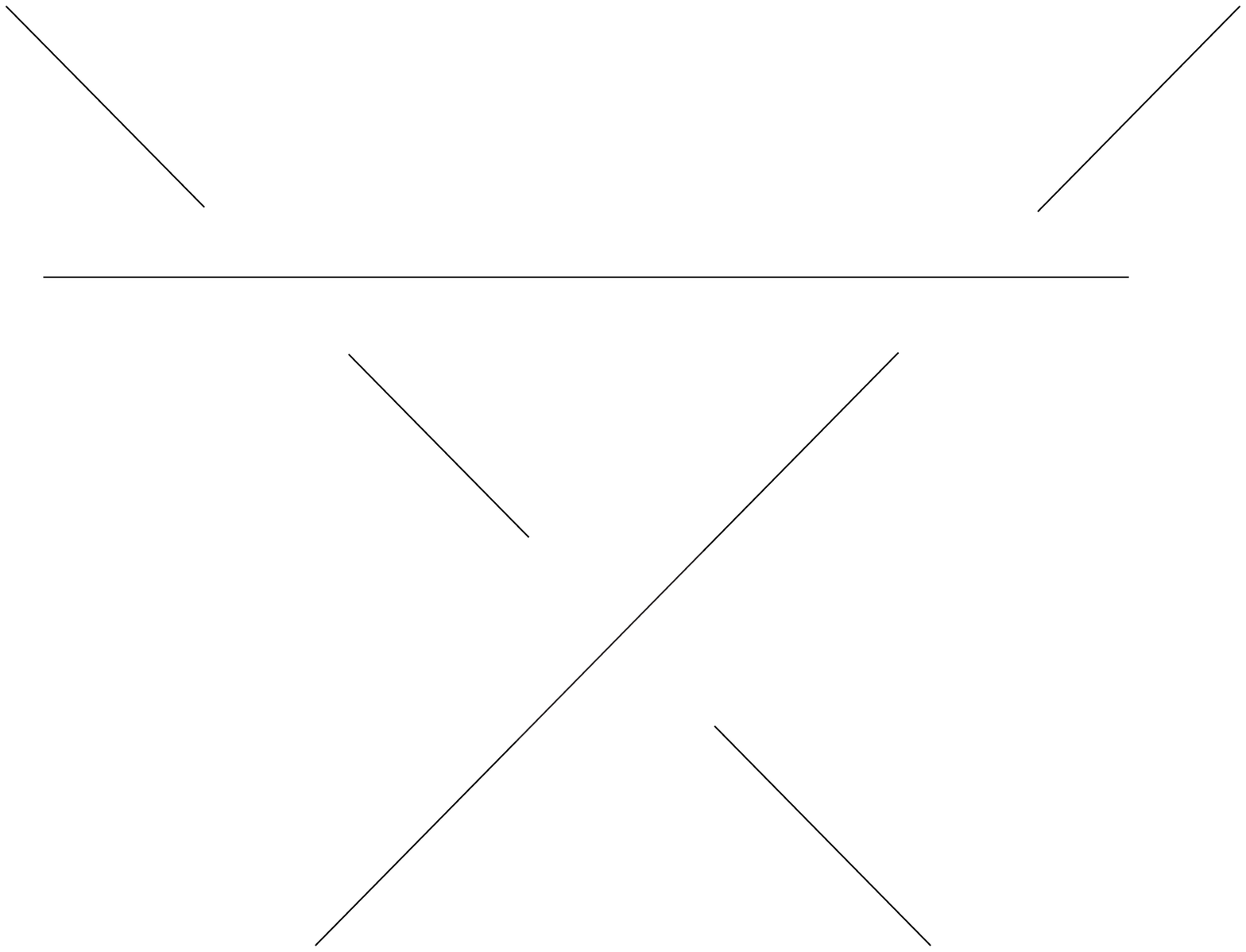}}}
\newcommand{\graphone}{\raisebox{-0.25\height}{\includegraphics[width=1.0cm]{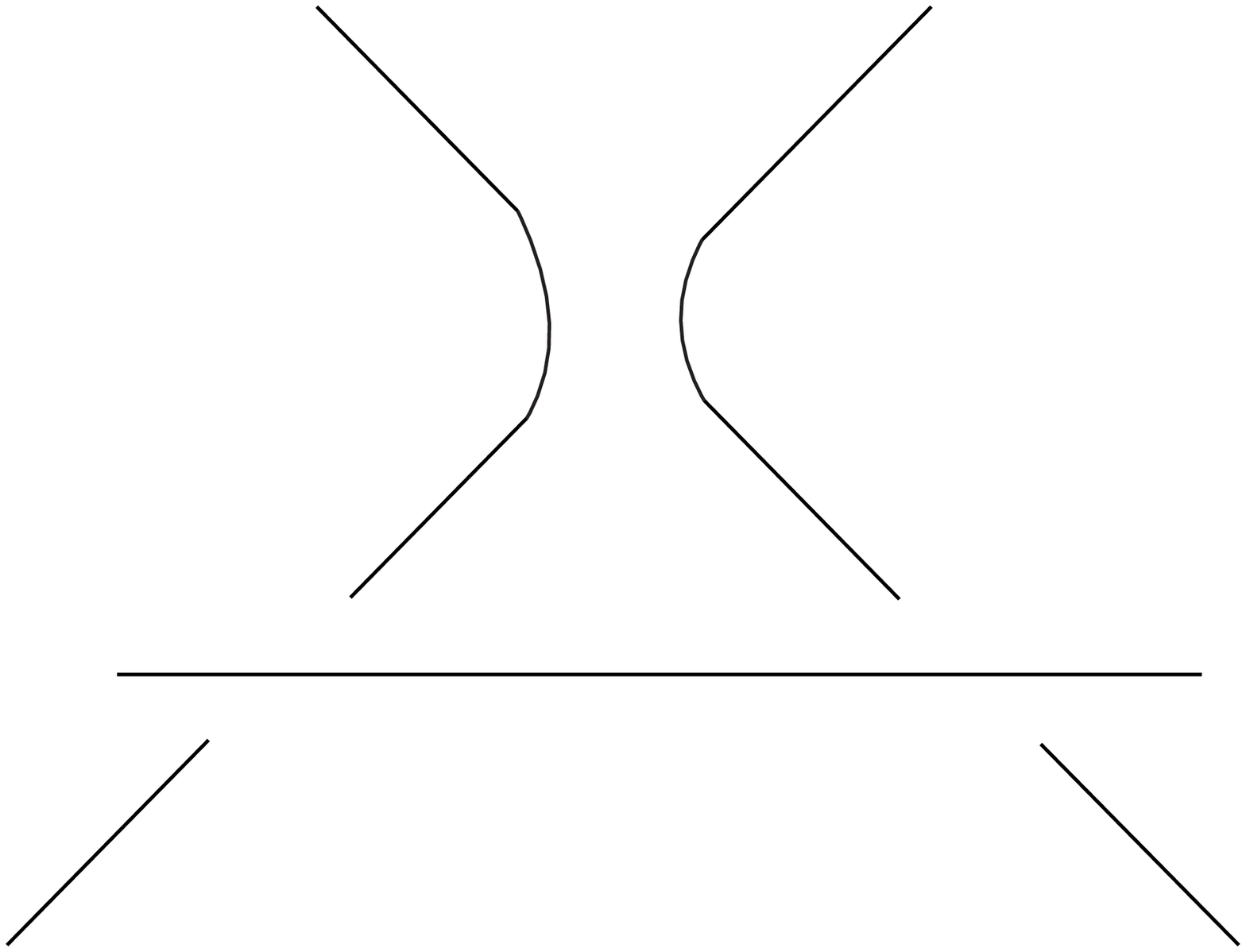}}}
\newcommand{\graphtwo}{\raisebox{-0.25\height}{\includegraphics[width=1.0cm]{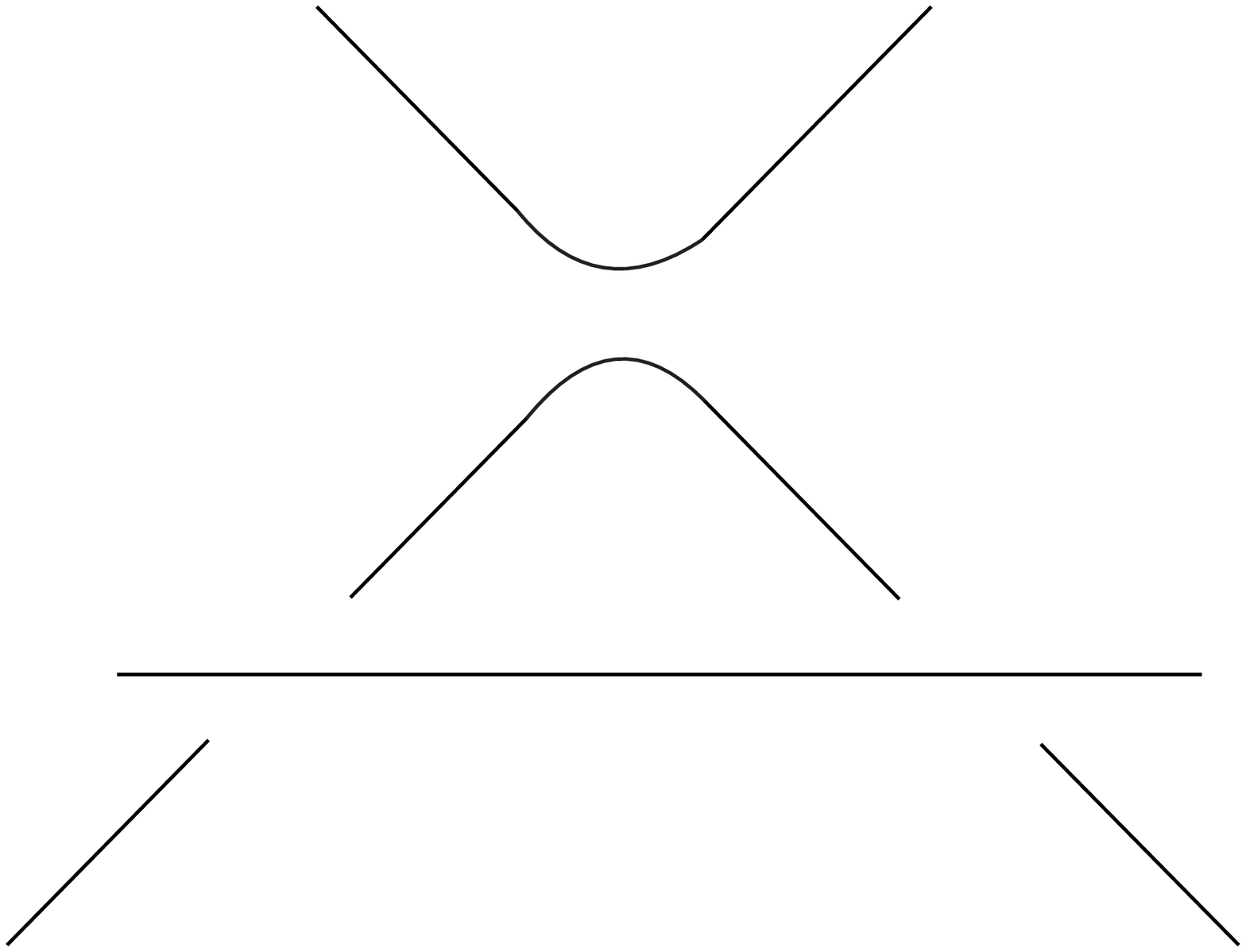}}}
\newcommand{\graphthree}{\raisebox{-0.25\height}{\includegraphics[width=1.0cm]{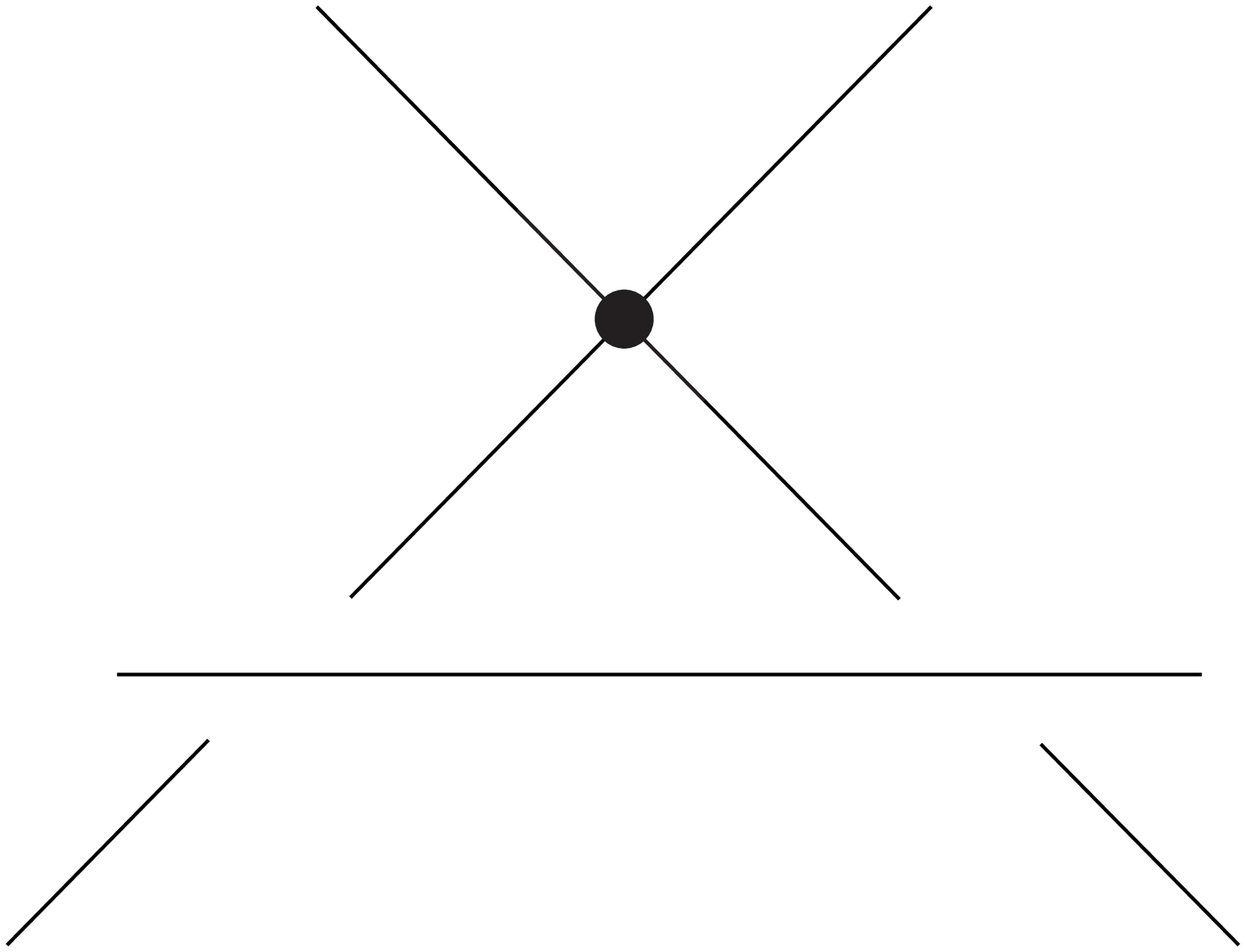}}}
\newcommand{\graphfour}{\raisebox{-0.25\height}{\includegraphics[width=1.0cm]{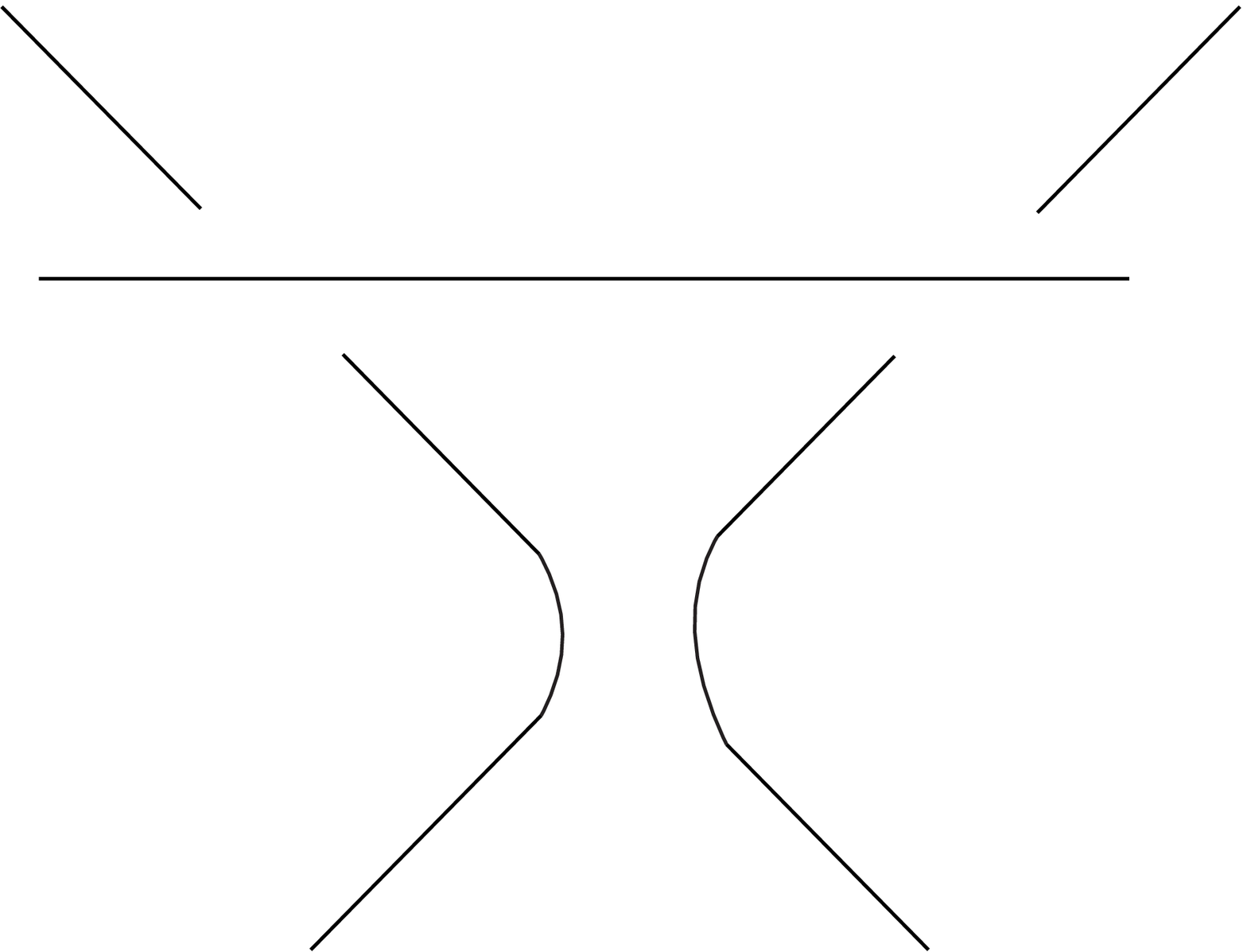}}}
\newcommand{\graphfive}{\raisebox{-0.25\height}{\includegraphics[width=1.0cm]{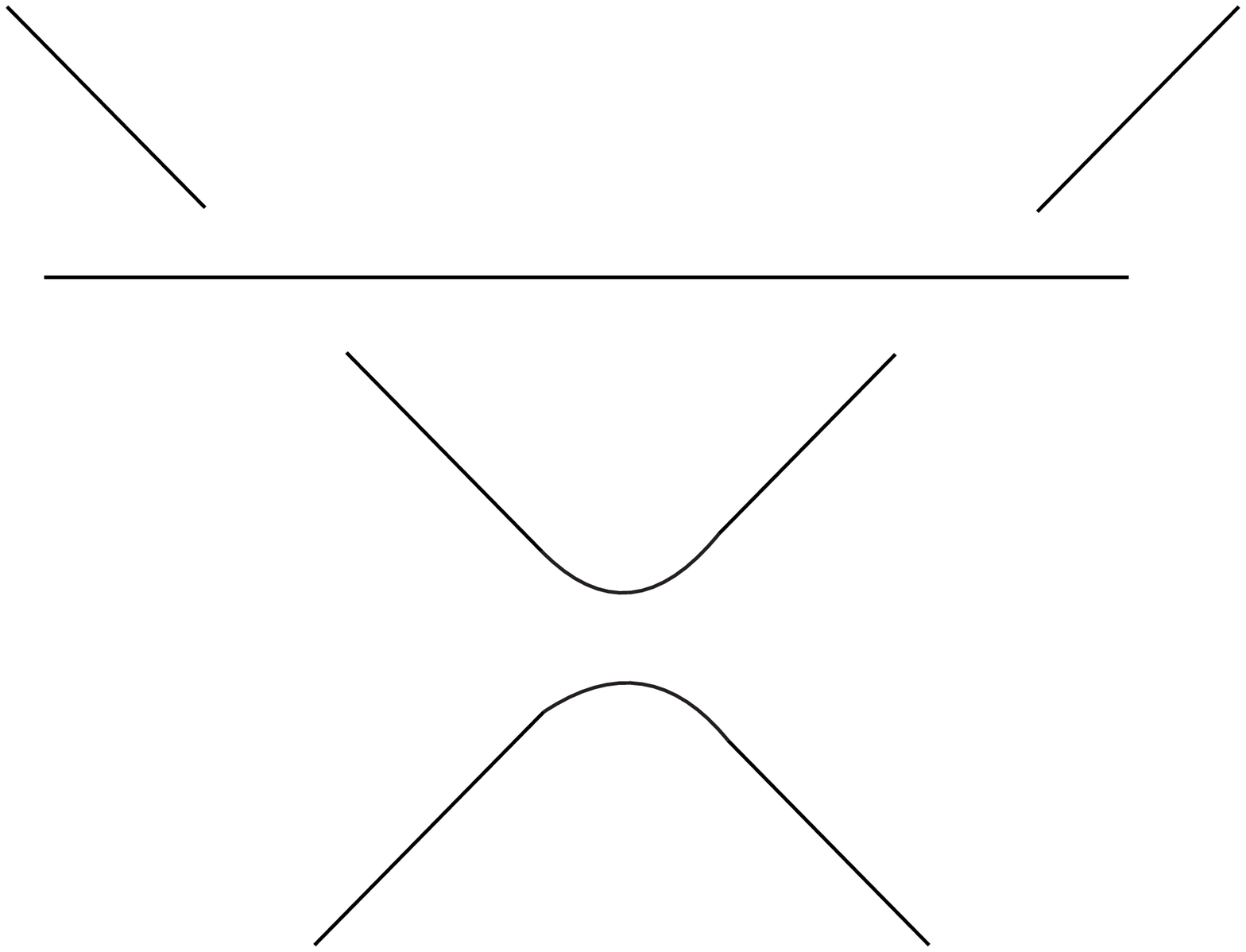}}}
\newcommand{\graphsix}{\raisebox{-0.25\height}{\includegraphics[width=1.0cm]{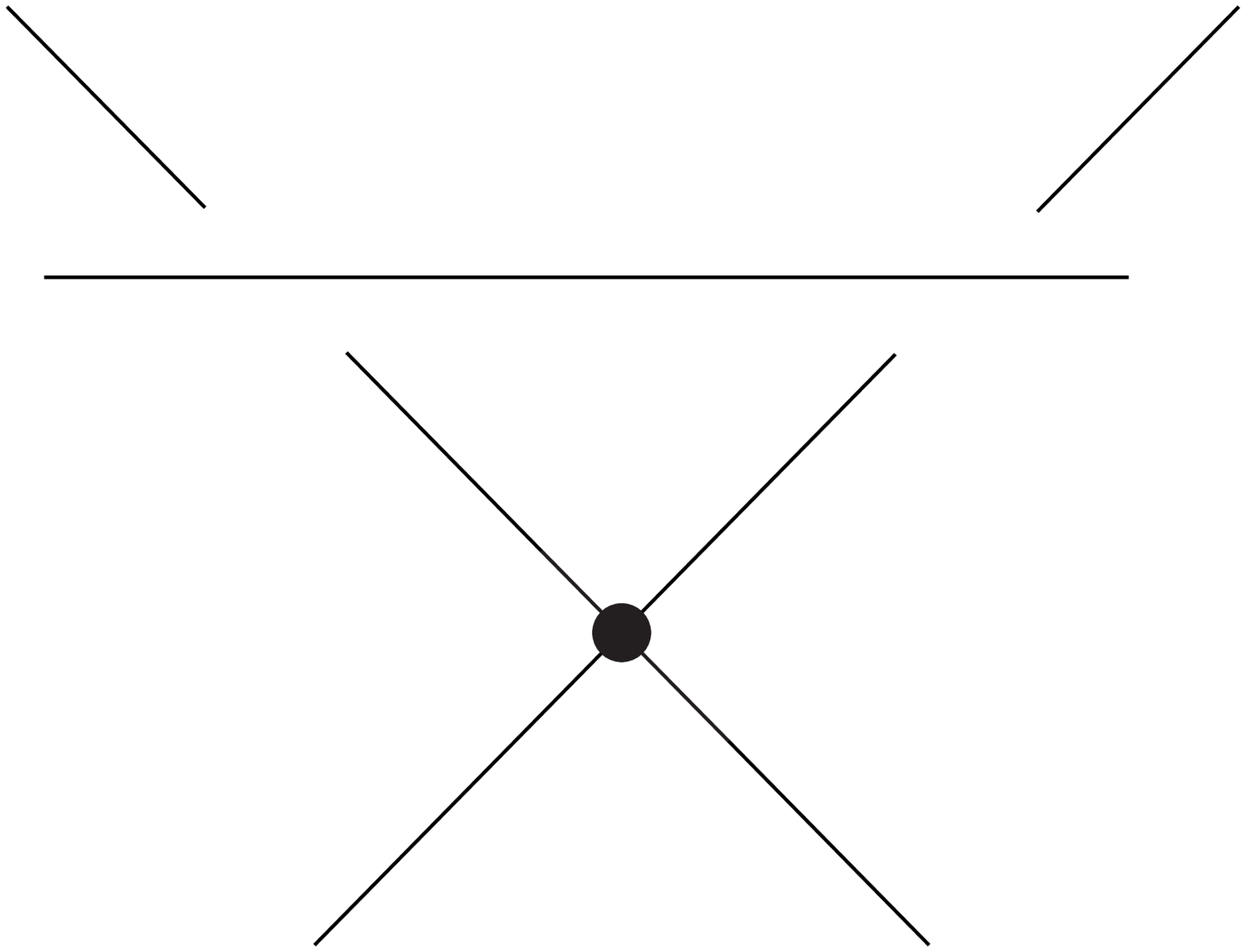}}}

\newcommand{\vertex}{\raisebox{-0.25\height}{\includegraphics[width=0.1cm]{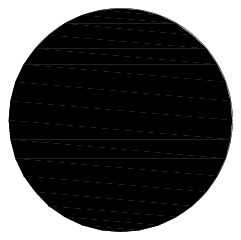}}}
\newcommand{\loopmaker}{\raisebox{-0.25\height}{\includegraphics[width=0.4cm]{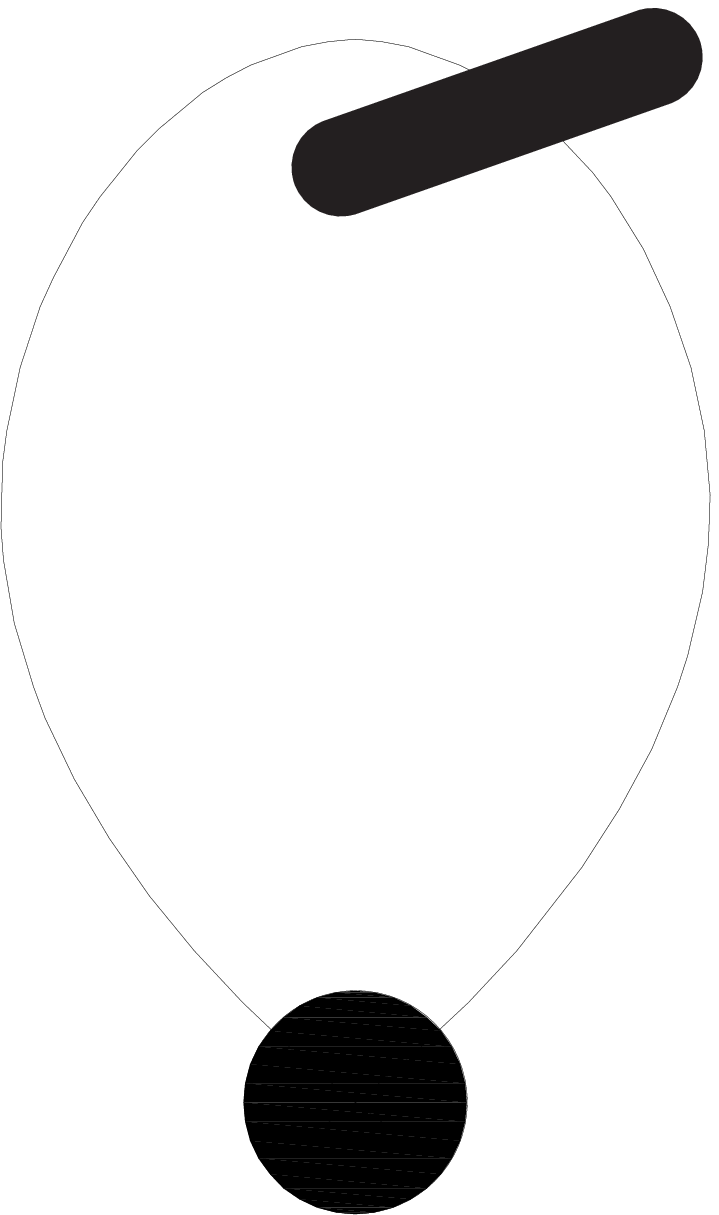}}}
\newcommand{\selfloop}{\raisebox{-0.25\height}{\includegraphics[width=0.4cm]{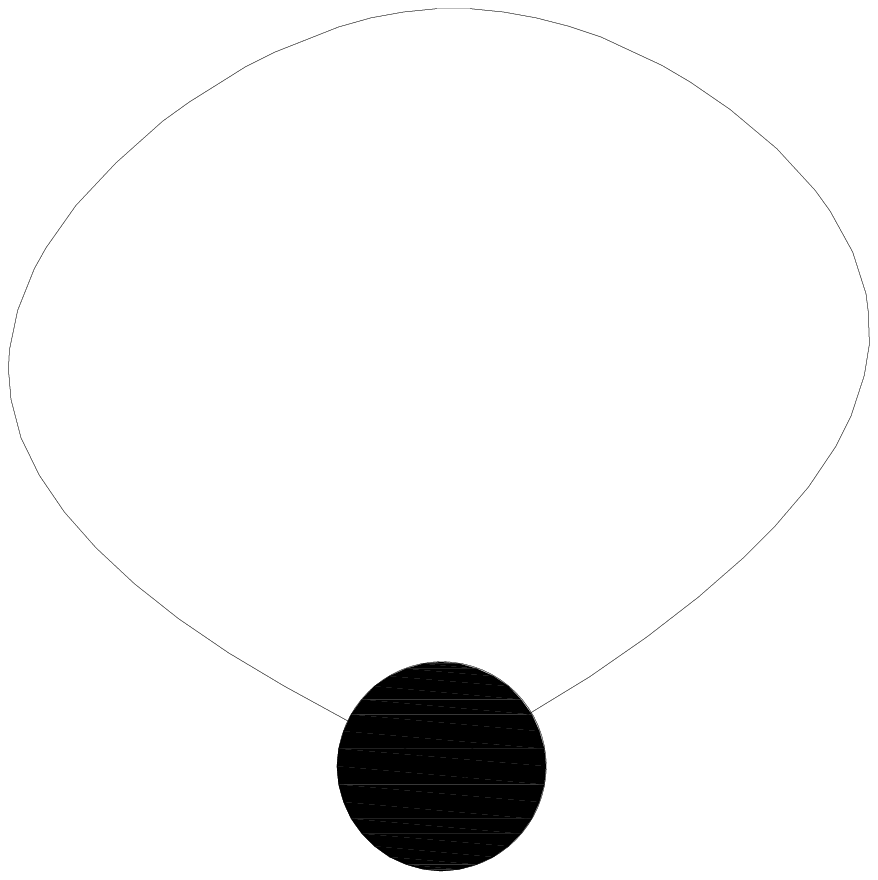}}}

\newcommand{\crossarc}{\raisebox{-0.25\height}{\includegraphics[width=0.8cm]{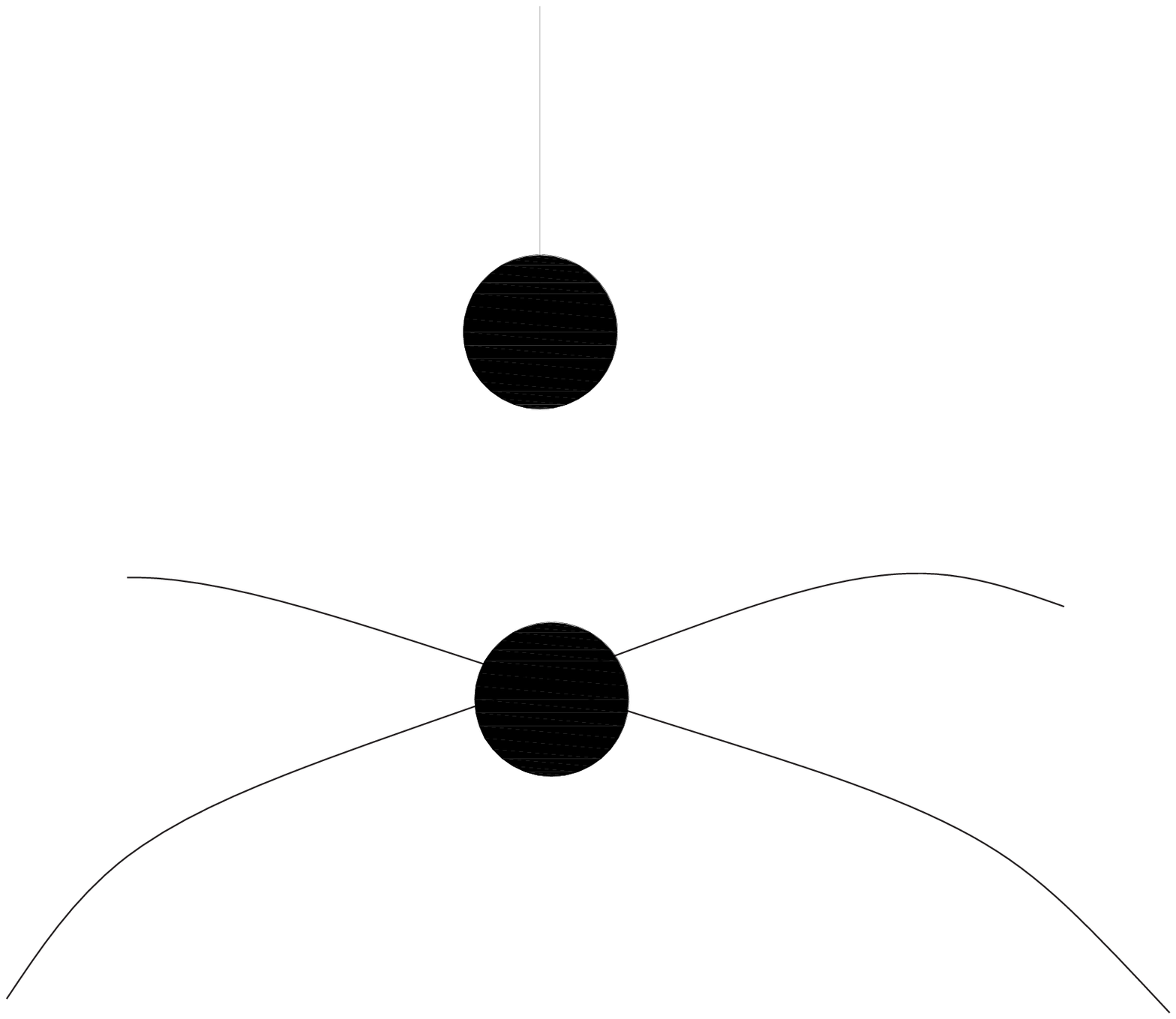}}}
\newcommand{\ReidMove}{\raisebox{-0.25\height}{\includegraphics[width=0.8cm]{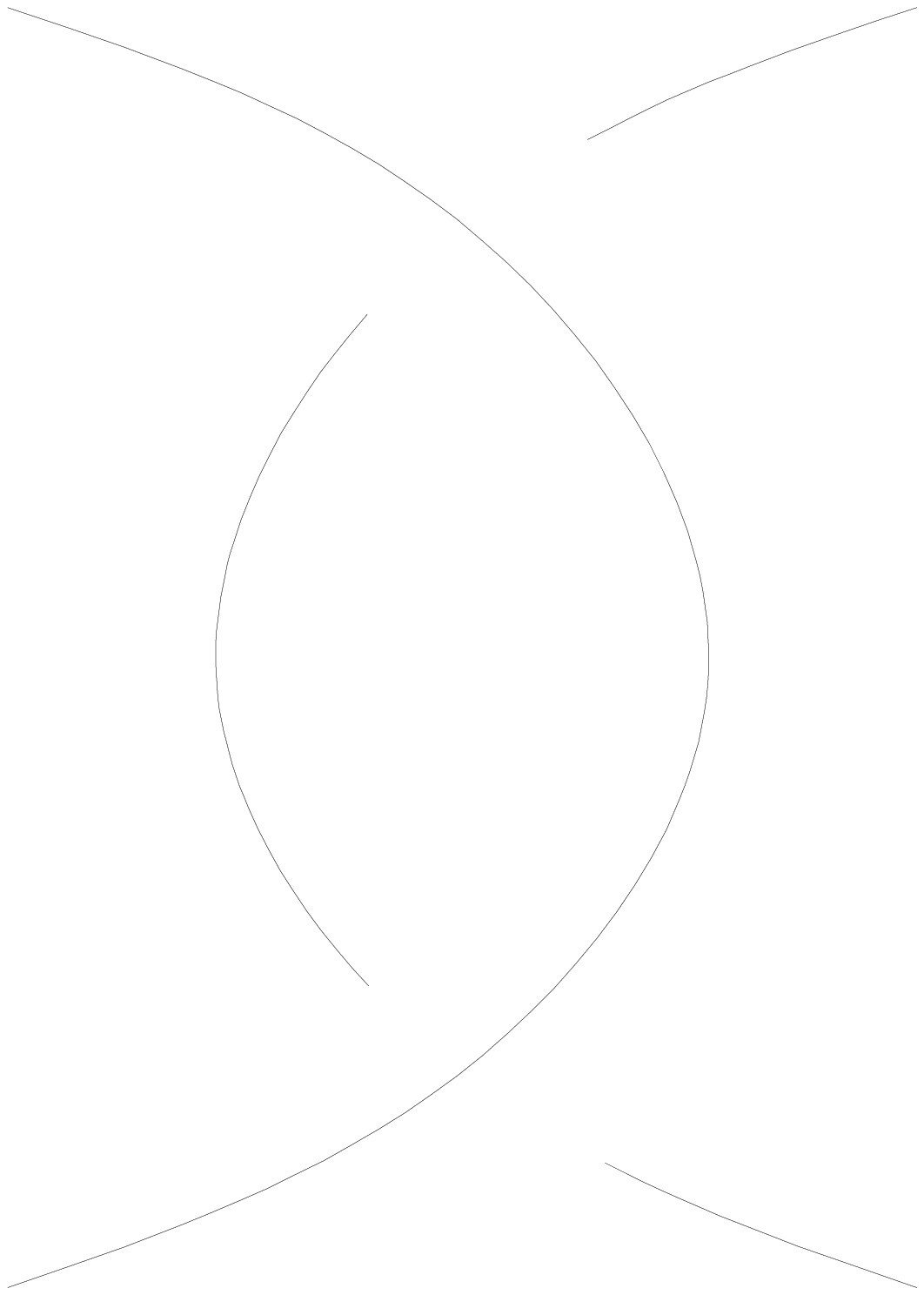}}}
\newcommand{\onee}{\raisebox{-0.25\height}{\includegraphics[width=1.2cm]{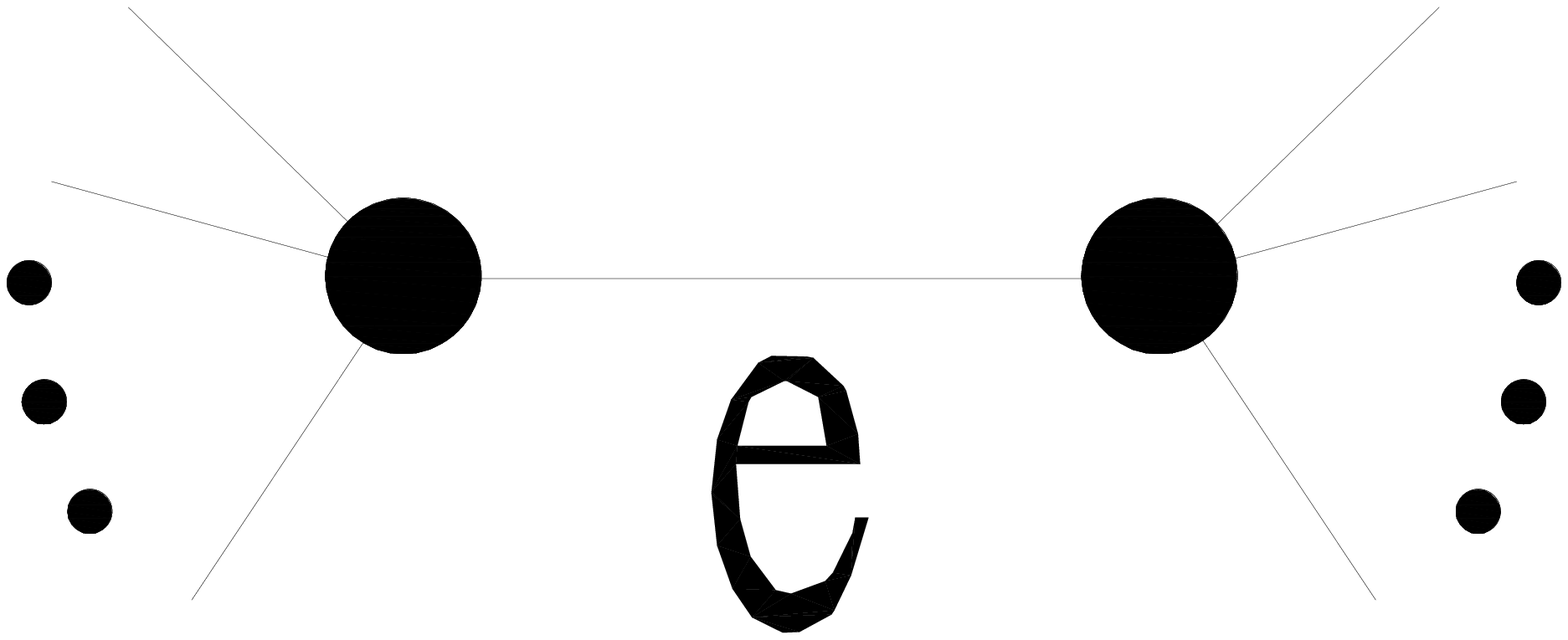}}}
\newcommand{\deleteedge}{\raisebox{-0.25\height}{\includegraphics[width=1.2cm]{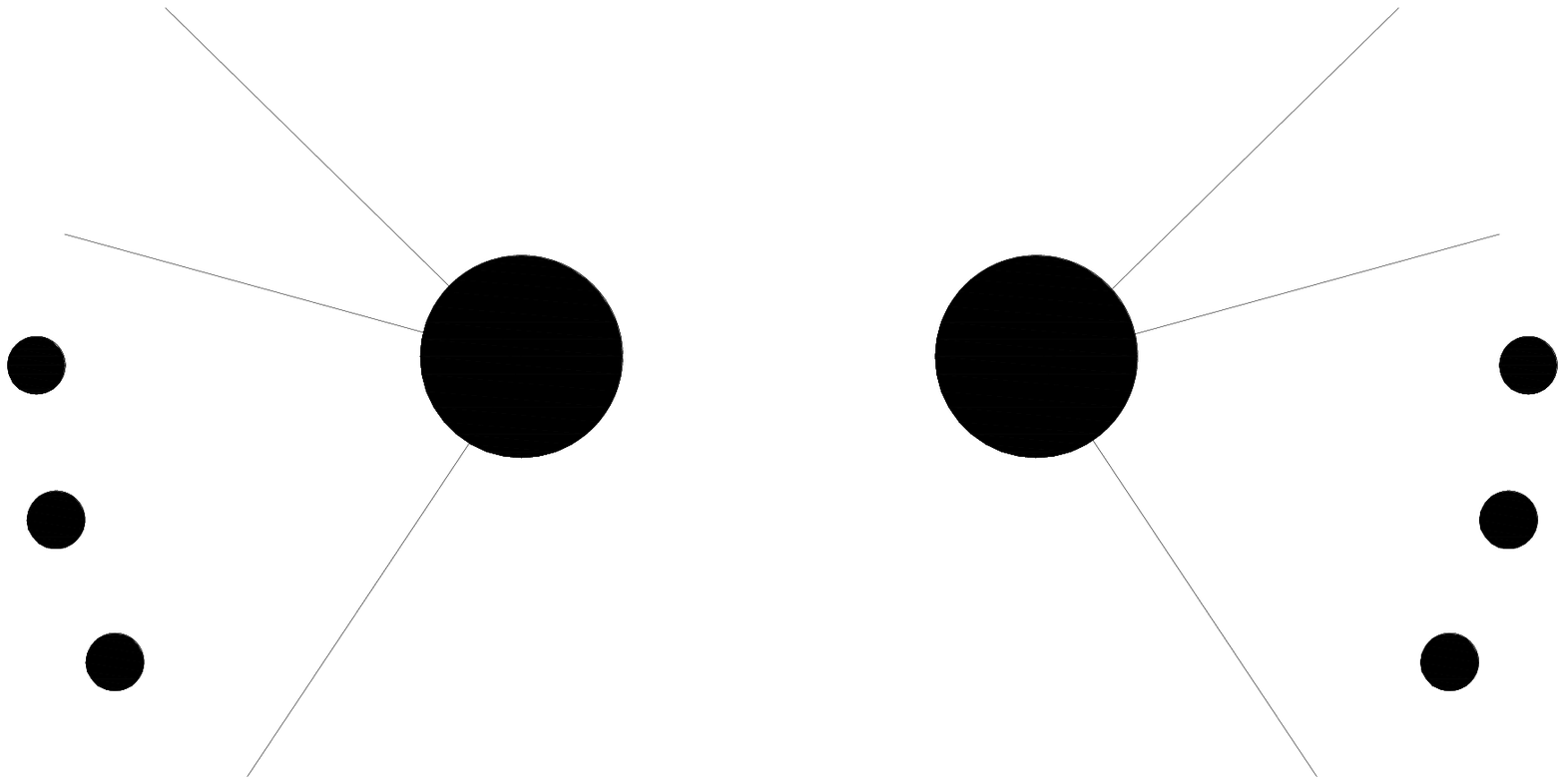}}}
\newcommand{\markingedge}{\raisebox{-0.25\height}{\includegraphics[width=1.2cm]{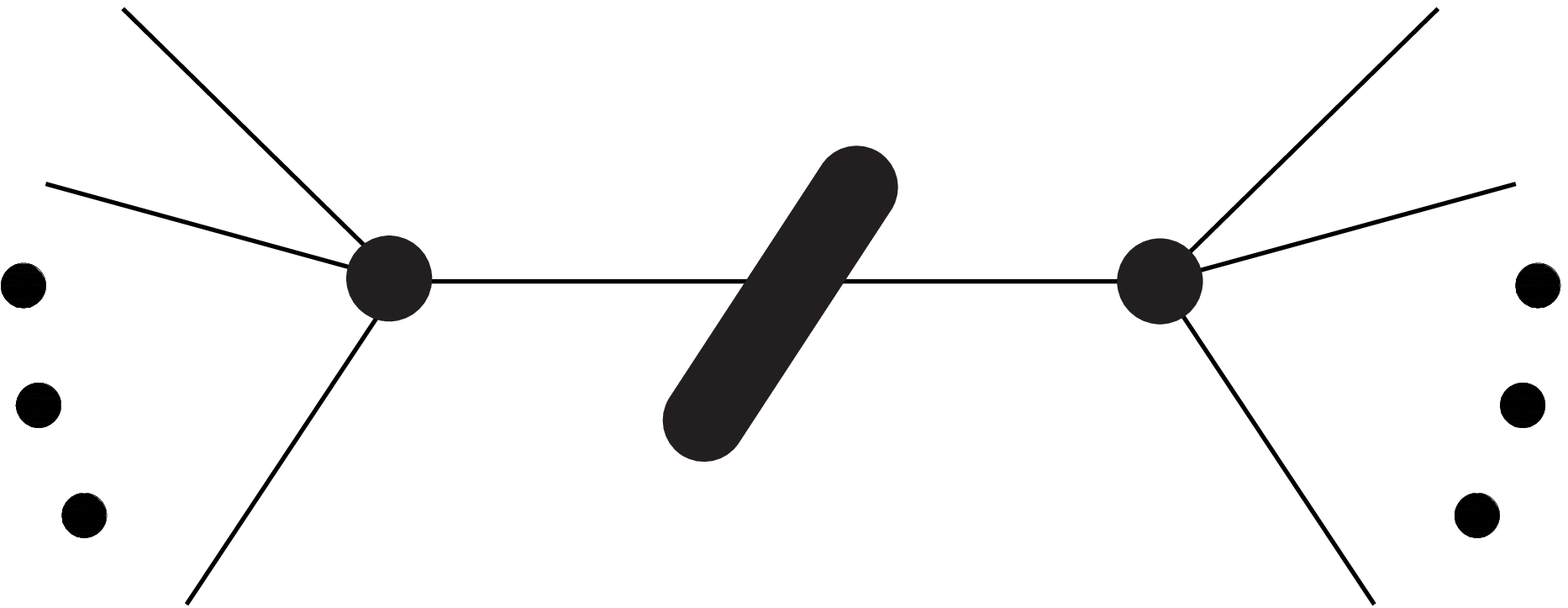}}}
\newcommand{\upthree}{\raisebox{-0.25\height}{\includegraphics[width=0.8cm]{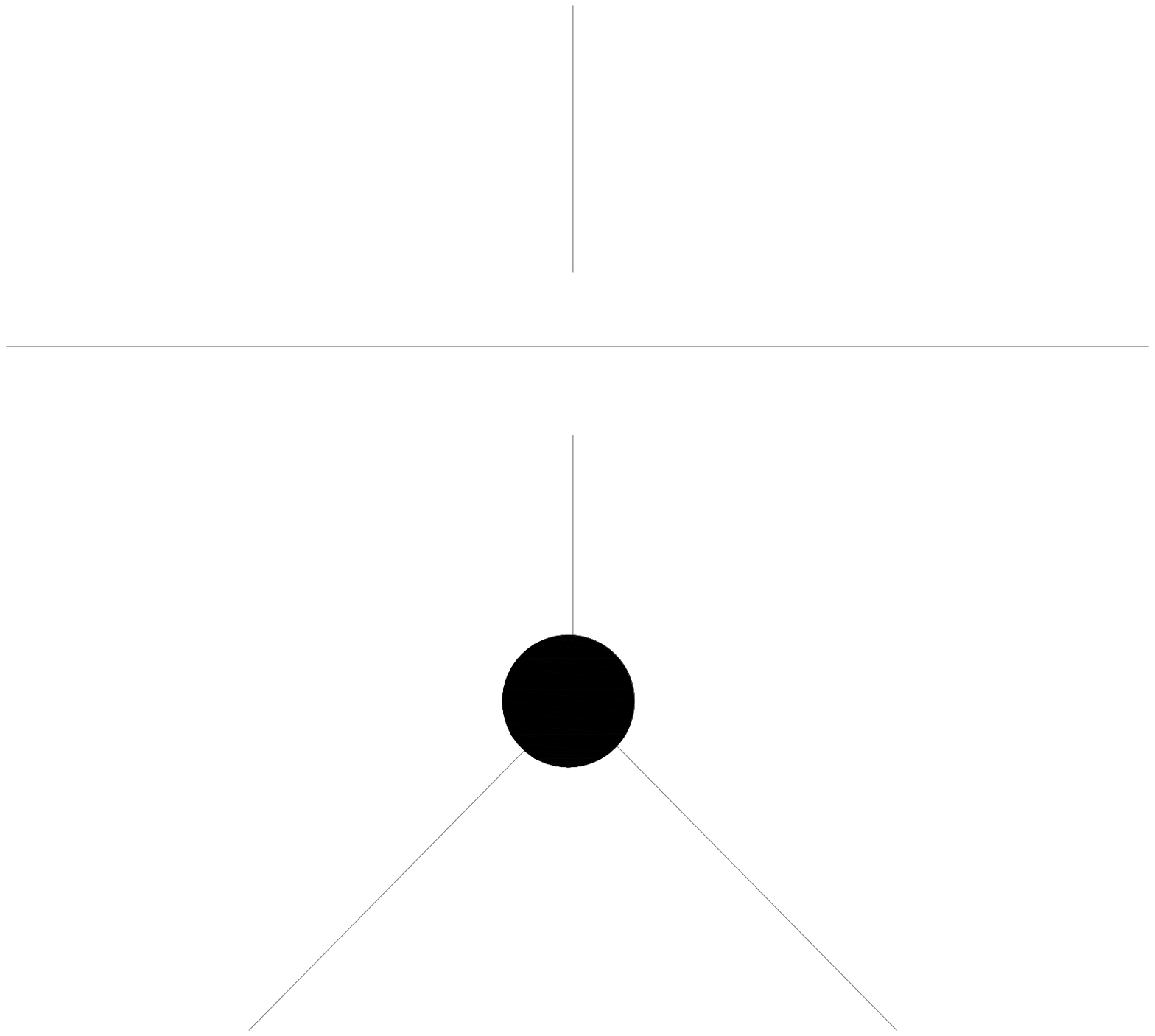}}}

\newcommand{\makrboundary}{\raisebox{-0.25\height}{\includegraphics[width=0.8cm]{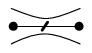}}}

\newcommand{\ndv}{\raisebox{-0.25\height}{\includegraphics[width=1.0cm]{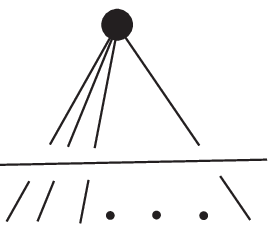}}}
\newcommand{\ndvae}{\raisebox{-0.25\height}{\includegraphics[width=1.1cm]{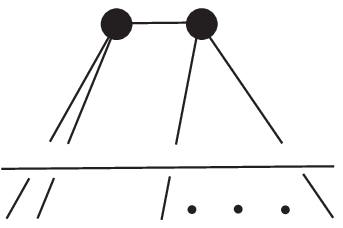}}}
\newcommand{\ndvde}{\raisebox{-0.25\height}{\includegraphics[width=1.1cm]{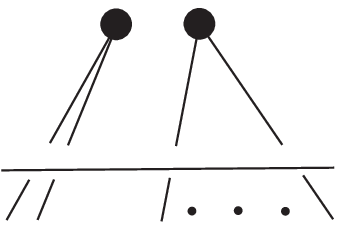}}}
\newcommand{\ndvdes}{\raisebox{-0.25\height}{\includegraphics[width=1.0cm]{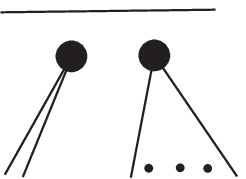}}}
\newcommand{\ndvaes}{\raisebox{-0.25\height}{\includegraphics[width=1.0cm]{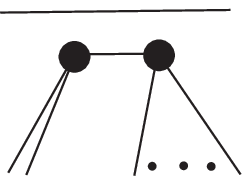}}}
\newcommand{\ndvs}{\raisebox{-0.25\height}{\includegraphics[width=0.8cm]{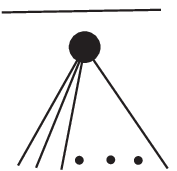}}}

\newcommand{\positivecurl}{\raisebox{-0.25\height}{\includegraphics[width=0.7cm]{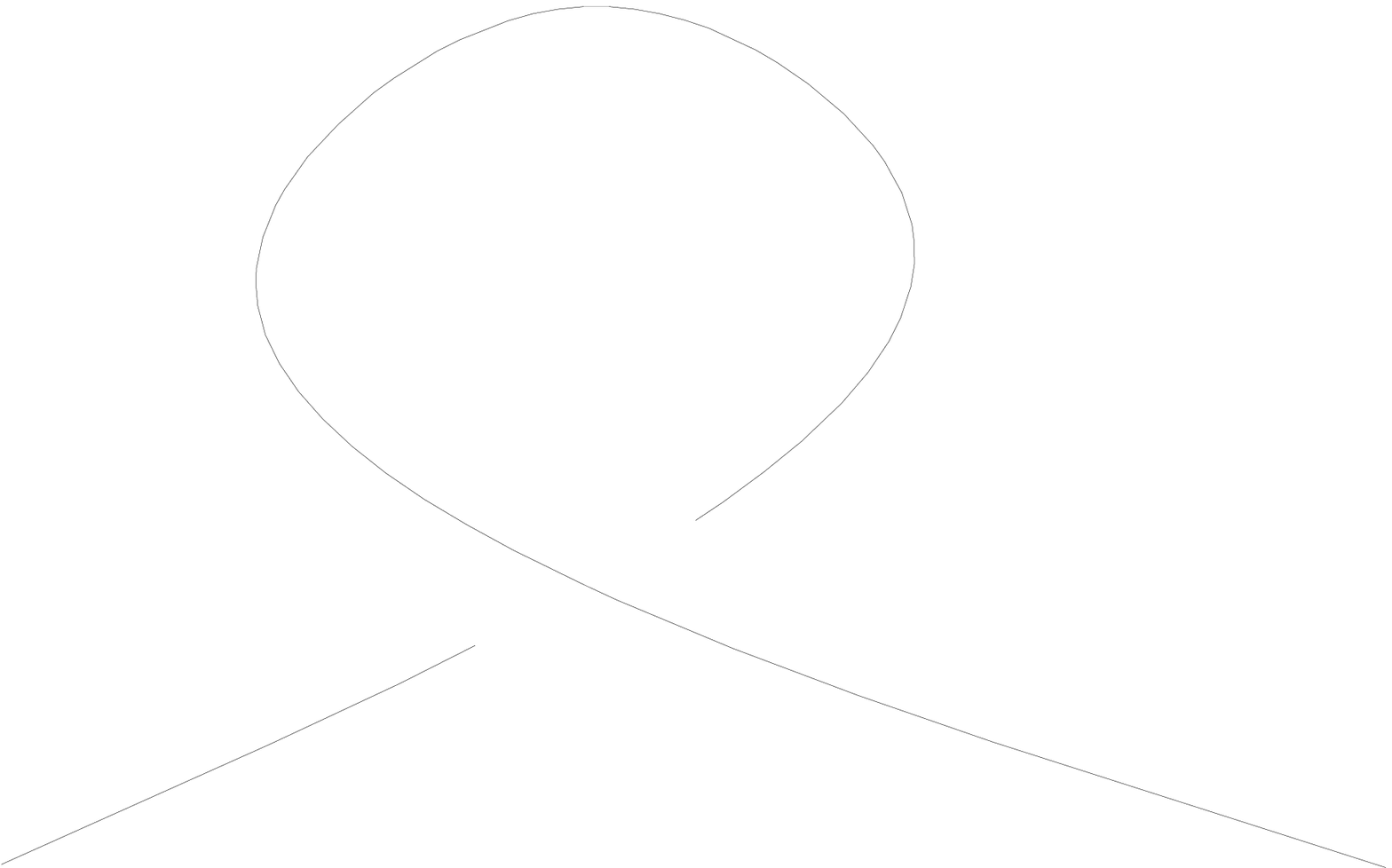}}}
\newcommand{\negativecurl}{\raisebox{-0.25\height}{\includegraphics[width=0.7cm]{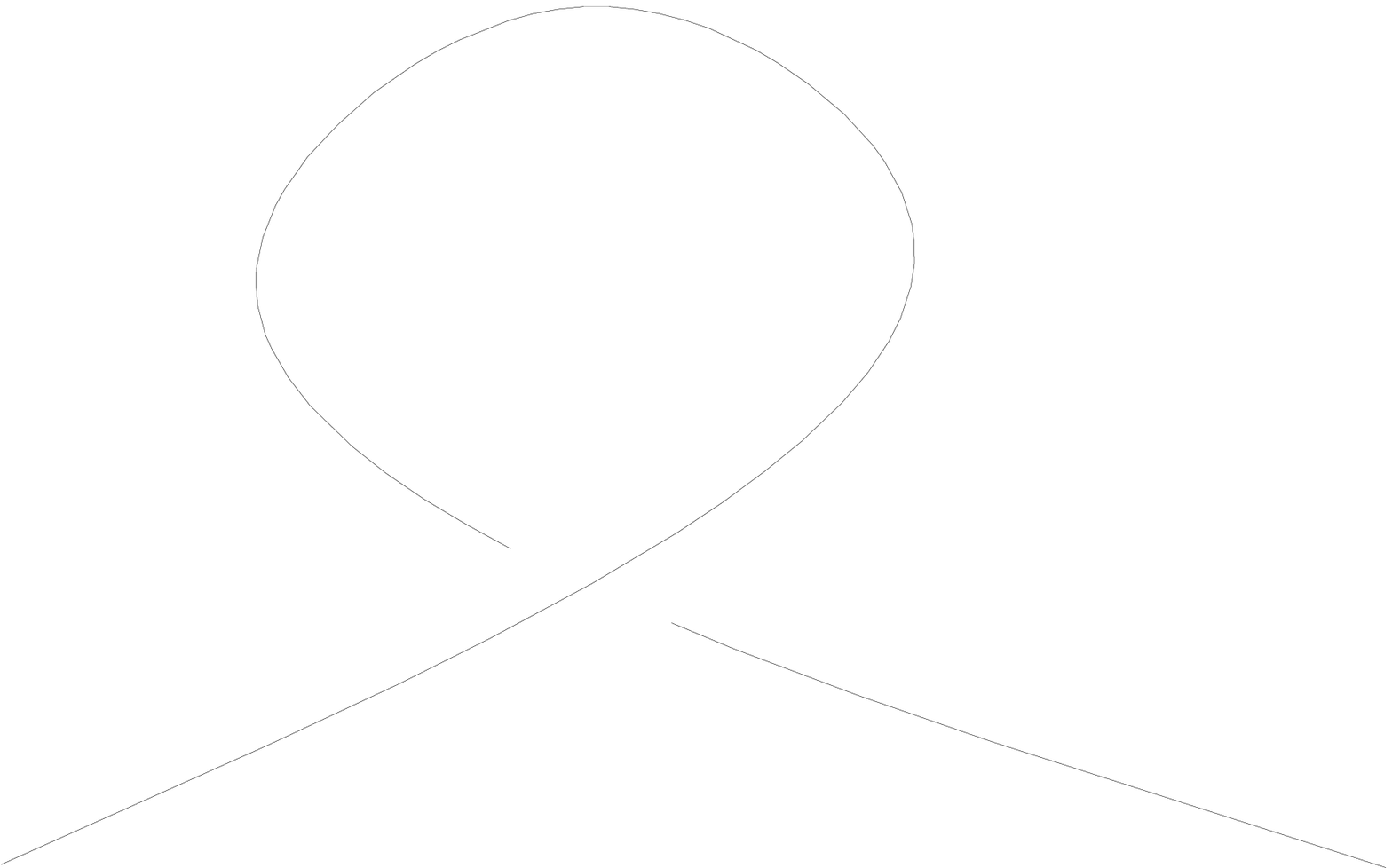}}}
\newcommand{\virtualhopflink}{\raisebox{-0.25\height}{\includegraphics[width=1.0cm]{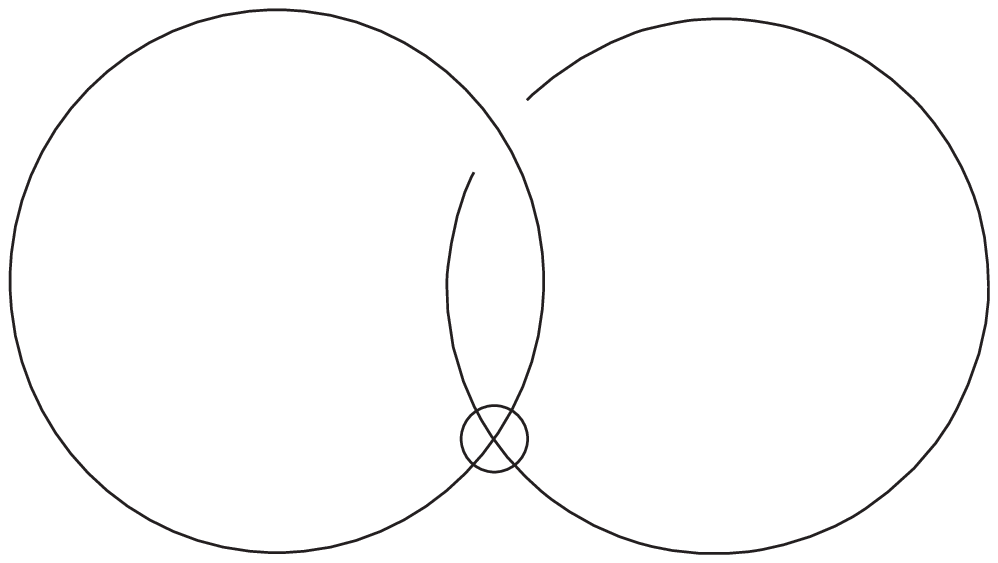}}}
\newcommand{\virtualtrefoil}{\raisebox{-0.25\height}{\includegraphics[width=1.0cm]{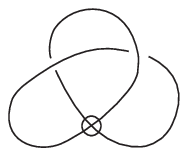}}}

\newcommand{\cablecrossing}{\raisebox{-0.25\height}{\includegraphics[width=1.0cm]{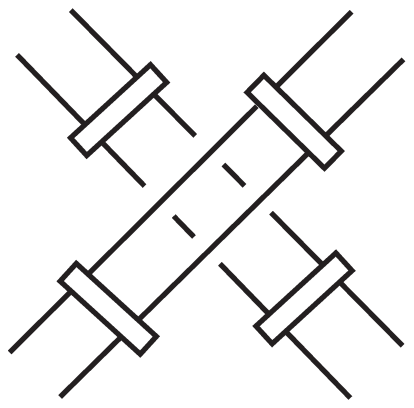}}}
\newcommand{\cableA}{\raisebox{-0.25\height}{\includegraphics[width=1.0cm]{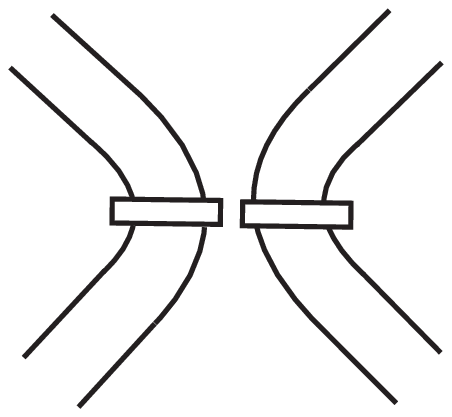}}}
\newcommand{\cableB}{\raisebox{-0.25\height}{\includegraphics[width=1.0cm]{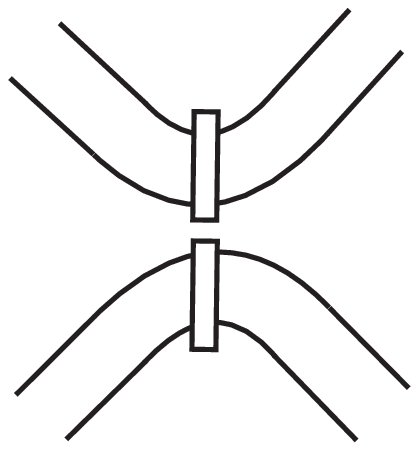}}}
\newcommand{\cablev}{\raisebox{-0.25\height}{\includegraphics[width=1.0cm]{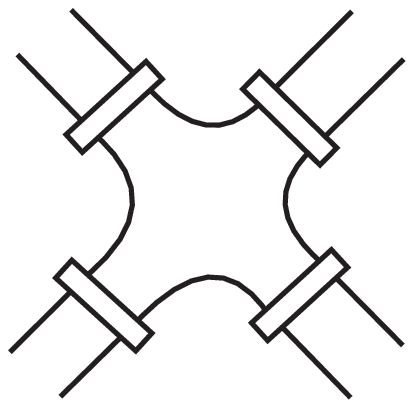}}}

\newcommand{\virtualizedtrefoil}{\raisebox{-0.25\height}{\includegraphics[width=1.0cm]{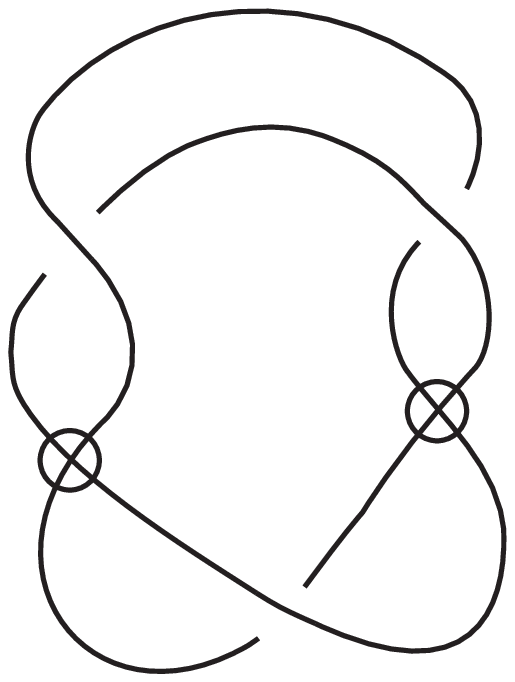}}}

\begin{document}
\title
{The generalized Yamada polynomials of virtual spatial graphs}

\author{Qingying Deng}
\address{School of Mathematical Sciences,Xiamen University, Xiamen, Fujian 361005,
P. R. China. E-mail: qydeng.xmu@gmail.com (Q. Y. Deng)}

\author{Xian'an Jin}
\address{School of Mathematical Sciences,Xiamen University, Xiamen, Fujian 361005,
P. R. China. E-mail: xajin@xmu.edu.cn (X. Jin)}

\author{Louis H. Kauffman}
\address{Department of Mathematics, Statistics and Computer Science \\ 851 South Morgan Street   \\ University of Illinois at Chicago\\
Chicago, Illinois 60607-7045\\ and\\ Department of Mechanics and Mathematics\\ Novosibirsk State University\\Novosibirsk, Russia\\E-mail: kauffman@uic.edu (L. Kauffman)}

\begin{abstract}
Classical knot theory can be generalized to virtual knot theory and spatial graph theory. In 2007, Fleming and Mellor combined virtual knot theory and spatial graph theory to form, combinatorially, virtual spatial graph theory. In this paper, we introduce a topological definition of virtual spatial graphs that is similar to the topological definition of a virtual link. Our main goal is to generalize the classical Yamada polynomial that is defined for a spatial graph. We define a generalized Yamada polynomial for a virtual spatial graph and prove that it can be normalized to a rigid vertex isotopic invariant and to a pliable vertex isotopic invariant for graphs with maximum degree at most 3. We consider the connection and difference between the generalized Yamada polynomial and the Dubrovnik polynomial of a classical link. The generalized Yamada polynomial specializes to a version of the Dubrovnik polynomial for virtual links such that it can be used to detect the non-classicality of some virtual links. We obtain a specialization for the generalized Yamada polynomial (via the Jones-Wenzl projector $P_2$ acting on a virtual spatial graph diagram), that can be used to write a program for calculating it based on Mathematica Code.
\end{abstract}



\maketitle

\section{Introduction}

\noindent

A knot is a closed 1-manifold embedded in $\mathbb{R}^3$.
A link is a disjoint union of finitely many closed 1-manifolds embedded in $\mathbb{R}^3$.
Knot theory is the study of isotopy classes of circles embedded in $\mathbb{R}^3$.
Any link can be described by its \textit{diagram}, the result of projecting the embedding in $\mathbb{R}^3$ to a plane, retaining information about over-strand and under-strand on each crossing.
The notion of studying topological embeddings of spatial graphs has been in the literature for a long time. For example, see Section 5 in \cite{RHFox}.

In 1989, Kauffman \cite{Kau89} and Yamada \cite{Yam} independently introduced the notion of a diagrammatic \textit{spatial graph} (that is, a graph in $\mathbb{R}^3$), which extended  diagrammatic knot theory.
The theory of spatial graphs studies two types of isotopy classes of graphs embedded in $\mathbb{R}^3$: graphs with rigid vertices and graphs with nonrigid (topological) vertices.
Later, Kauffman \cite{VKT} introduced the notion of a \textit{virtual knot}, which is an extension of classical knot theory.
Moreover, Kamada \cite{Kam} introduced the notion of an abstract link diagram, and it is proved that there is a bijection from the equivalence classes of virtual link diagrams to an appropriate equivalence relation on abstract link diagrams.
Fleming and Mellor introduced the notion of a \textit{virtual spatial graph} (see Section \ref{S:virtualsg}) by using a combinatorial method in \cite{Fleming07}, which combined a spatial graph and a virtual knot.
They also defined the Yamada polynomial of virtual spatial graphs based on the idea of the original Yamada polynomial for spatial graphs.
But it is very weak in the sense that Miyazawa \cite{Miyazawa} defined a new type Yamada polynomial for virtual spatial graphs via adding a new parameter $\varepsilon$ with value $\pm1$. Miyazawa's polynomial will correspond to our generalized Yamada polynomial $R(D;\alpha,x)$ with $\alpha=A$ and $x=\pm1$.

Furthermore, the theory of virtual spatial graphs studies two kinds of isotopy classes of graphs embedded in orientable thickened surface $S_g\times I$, where $S_g$ is a compact orientable surface with genus $g$ and $I$ is a unit interval of $[0,1]$.
In \cite{KauMishra}, Kauffman and Mishra proposed the concept of \textit{a virtual graph} which is equivalent to a 4-regular virtual spatial graph as virtual rigid vertex graph (see Section \ref{S:virtualsg}).
In \cite{KauMishra}, degree 4 vertices are considered as special crossings, i.e., in a virtual graph, there are three types of transverse double points.
However, Kauffman and Manturov in \cite{KauMan} use the term ``virtual graph'' to mean a 4-regular graph with virtual crossings only.
We therefore note that from our point of view an acceptable usage is ``virtual graph'' for arbitrary cyclic graph (See Definition \ref{D:cyclic}) with virtual crossings immersed in the plane. The term ``spatial'' is reserved for diagrams that have classical crossings.

Very recently, Snyder and Miller \cite{SnyderMiller} defined the notion of a virtual graph (equivalently, cyclic graph) in a topological way.
And they define the $S$-polynomial for virtual graphs which is a little different from  (but similar with) our generalized Yamada polynomial for cyclic graphs.
They also defined the Yamada polynomial for virtual spatial graphs which is based on the $S$-polynomial.

For related studies of spatial graphs, please refer to \cite{Hanaki,Harvey,KKKP,LLLV,Murakami,Yam}.
For related studies of virtual spatial graphs, please refer to
\cite{Fleming07,fm3,KauMishra,Miyazawa,SnyderMiller}.

Understanding the equivalence classes of the virtual spatial graphs is an important issue.
Therefore, people try to generalize the related research in knot theory, virtual knot theory, and spatial graph theory to virtual spatial graphs.
Fleming and Mellor's article \cite{Fleming07} generalized some concepts in classical and virtual knot theory such as fundamental groups and quandles.
And they proved that the fundamental group and quandle are both pliable vertex (topological) isotopy invariants.

The paper is organized as follows:
In Section \ref{S:virtualsg}, we introduce a topological definition of the virtual spatial graph which is similar to the topological definition of a virtual link.
We define the generalized Yamada polynomial of virtual spatial graphs (as virtual rigid vertex graph and virtual pliable vertex graph) via their diagrams in Section \ref{S:GYP}.
We shall consider the connection and difference between it and the Dubrovnik polynomial of classical links in Section \ref{S:Dubrovnik}.
As an application of the generalized Yamada polynomial, by the Jones-Wenzl projector $P_2$ acting on a virtual spatial graph diagram,
we get a special parameter for the generalized Yamada polynomial, which can be used to write a program for calculating the generalized Yamada polynomial based on \textit{Mathematica} Code in Section \ref{S:Application}.

\section{Virtual spatial graphs}\label{S:virtualsg}

\noindent

Let us start with the definitions and introduce the notation.

\begin{defn}
Let $G=(V,E)$ be a graph embedded in $\mathbb{R}^3$. We say $G$ is a \textit{spatial graph}.
\end{defn}

\begin{defn}\cite{Yam,Kau89,KV}\label{kvd}
A \textit{spatial graph diagram} is a planar representation of a spatial graph and it is analogous to the link diagram of a link.
In the spatial graph diagram, vertices of the graph are represented by planar nodes as in Figure \ref{F:classicalmoves}, and all crossings are in generic form as in a classical link diagram.
\end{defn}

A $\textit{virtual link diagram}$ consists in a finitely many closed 1-manifolds generically immersed in $\mathbb{R}^2$ such that each double point is labeled to be (1) a classical crossing which is indicated as usual in classical knot theory or (2) a virtual crossing which is indicated by a small circle around the double point.
The moves (I)-(III) and (I*)-(IV*) for virtual link diagrams illustrated in Figures \ref{F:classicalmoves} and \ref{F:virtualmoves} are called $\textit{generalized Reidemeister moves}$.
The $\textit{Detour move}$ illustrated in Figure \ref{dm} is equivalent to a combination of moves (I*)-(IV*) in Figure \ref{F:virtualmoves}.
Two virtual link diagrams are said to be \textit{equivalent} if they are related by a finite sequence of generalized Reidemeister moves.
We call the equivalence class of a virtual link diagram a $\textit{virtual link}$.
We refer the reader to \cite{VKT,Kau3} for further details about virtual links.

\begin{defn}
A \textit{virtual spatial graph diagram} is a graph generically immersed in $\mathbb{R}^2$, that is, there are two kinds of double points: virtual crossings and classical crossings.
A virtual spatial graph diagram is analogous to the spatial graph diagram of a spatial graph, with the addition of virtual crossings.
\end{defn}

\begin{figure}[!htbp]
\centering
\includegraphics[width=10cm]{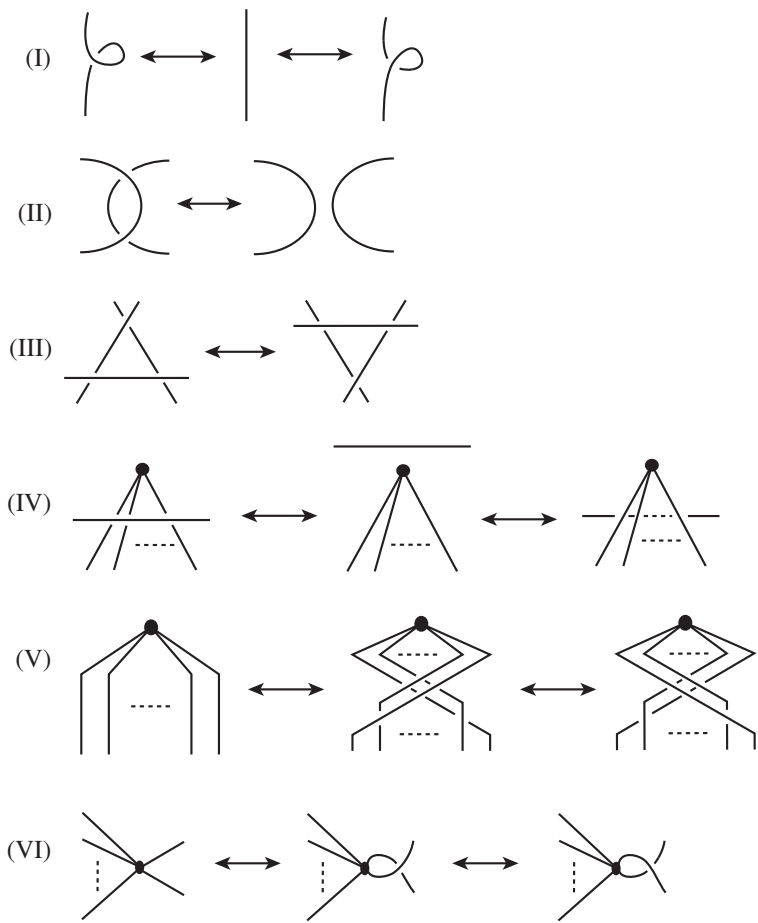}
\caption{Reidemeister\,moves for spatial graph diagrams.}
\label{F:classicalmoves}
\end{figure}

\begin{figure}
  \centering
  \includegraphics[width=8cm]{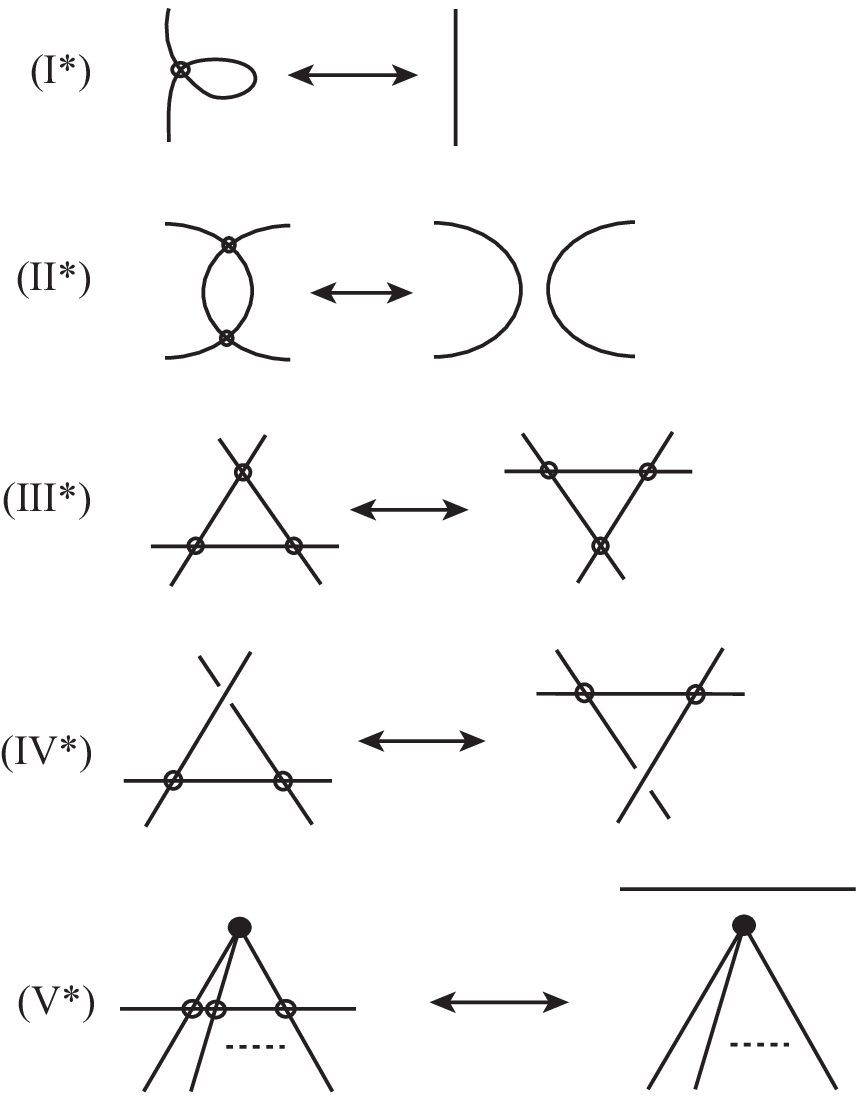}
 \caption{Reidemeister\,moves for virtual spatial graph diagrams.}
 \label{F:virtualmoves}
\end{figure}

\begin{figure}
\centering
\includegraphics[width=3.5in]{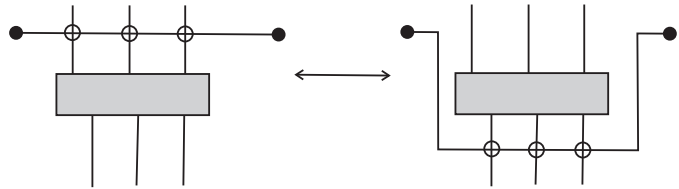}
\caption{Detour move.}
\label{dm}
\end{figure}

\begin{figure}
\centering
\includegraphics[width=9cm]{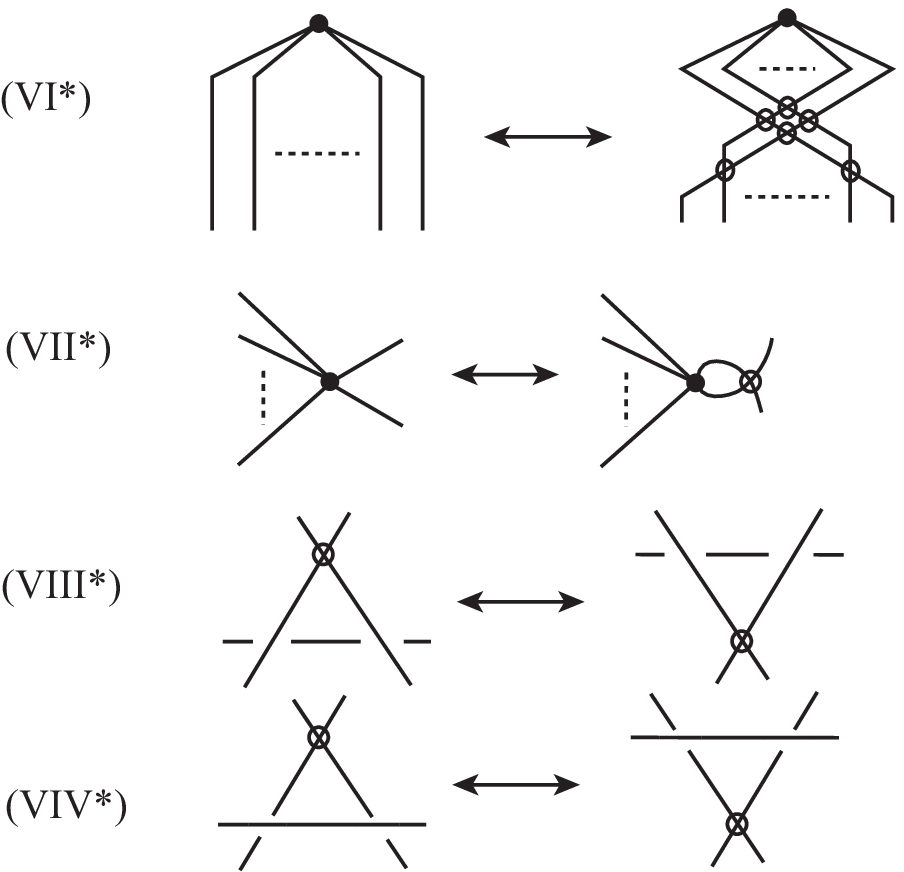}
\caption{Forbidden moves on virtual spatial graph diagrams.}
\label{F:forbiddenmoves}
\end{figure}

Figure \ref{F:classicalmoves} illustrates the Reidemeister moves for spatial graph diagrams. In \cite{Fleming07},
Fleming and Mellor added five moves with virtual crossings for virtual spatial graph diagrams as shown in Figure \ref{F:virtualmoves}.
\begin{defn}
In virtual spatial graph diagrams, there are the following four kinds of isotopy deformations.
\begin{itemize}
  \item Moves (II)-(III) and moves (I*)-(V*) generate \textit{{\bf regular deformation}}.
  \item Moves (I)-(III) and moves (I*)-(V*) generate \textit{{\bf ambient isotopic deformation}}.
  \item Moves (I)-(V) and moves (I*)-(V*) generate \textit{{\bf rigid vertex deformation}}, where the cyclic order of the half-edges around each vertex is fixed.
  \item Moves (I)-(VI) and moves (I*)-(V*) generate \textit{{\bf pliable vertex deformation}}, where the order of the half-edges around each vertex can be changed using move (VI).
\end{itemize}
\end{defn}

\begin{rem}
The four moves shown in Figure \ref{F:forbiddenmoves} are forbidden. They are not in the list of Reidemeister moves for virtual spatial graph diagrams although they are not a combination of Reidemeister moves for virtual spatial graph diagrams.
\end{rem}

Two virtual spatial graph diagrams are said to be \textit{regular, ambient, rigid vertex and pliable vertex isotopic} if they are related by a finite sequence of regular, ambient isotopic, rigid vertex and pliable vertex deformations, respectively.
We call the rigid vertex and pliable vertex isotopic equivalence class of a virtual spatial graph diagram as a \textit{virtual rigid vertex graph and virtual pliable vertex graph}, respectively.
That is, for this study, two types of virtual spatial graphs are considered: graphs with rigid vertices and graphs with nonrigid (topological) vertices.
Note that, virtual rigid vertex graph means that a vertex of the graph is regarded as a disk, but virtual pliable vertex graph means that the vertex of graph is a topological vertex such that the cyclic order of the half-edges around each vertex can be changed at will.

Let $S_g$ be a compact orientable surface with genus $g$ and $I$ be the unit interval $[0,1]$. It is also well known that a link diagram drawn on $S_g$, considered up to Reidemeister moves, is equivalent to a virtual link in the thickened surface $S_g\times I$, considered up to isotopy,
Similarly, we will give a topological definition of the virtual spatial graph.

\begin{defn}
Let $G=(V,E)$ be a graph embedded in the thickened surface $S_g\times I$.  We say $G$ is a \textit{virtual spatial graph}.
\end{defn}

We say that two such surface embeddings are \textit{stably equivalent} if one can be obtained from the other by  rigid vertex isotopy (or pliable vertex isotopy) in the thickened surface, homeomorphisms of surfaces,
and the addition or subtraction of handles not incident to images of curves in the graph.
That is, one can perform surgery on the surface along curves that do not intersect the virtual spatial graph diagram on
the surface.

Similar to virtual knot theory, we obtain the following result for the virtual spatial graph diagrams.
\begin{thm}
\label{T:stableequivalent}
Two virtual spatial graph diagrams are rigid vertex isotopic or pliable vertex isotopic if and only if their corresponding surface embeddings are stably equivalent. In each case, we take the corresponding diagrammatic relations on the surfaces, plus surface homeomorphism and stabilization.
\end{thm}
\begin{proof}
The proof is similar to the proof for the case of virtual link diagrams as in
\cite{Car2000}.
\end{proof}
\begin{defn}
\label{D:cyclic}
A \textit{cyclic graph} consists of an abstract graph $G$ and a cyclic ordering of the half-edges at each vertex.
\end{defn}
\begin{defn}\cite{Mo}
A $\textit{ribbon graph}$ $\textbf{G}=(V(\textbf{G}),E(\textbf{G}))$ is a surface with boundary, represented as the union of two sets of discs: a set $V(\textbf{G})$ of vertices and a set $E(\textbf{G})$ of edges such that:
\begin{itemize}
  \item [(1)] the vertices and edges intersect in disjoint line segments;
  \item [(2)] each such line segment lies on the boundary of precisely one vertex and precisely one edge; and
  \item [(3)] every edge contains exactly two such line segments.
\end{itemize}
\end{defn}

It is well-known that ribbon graphs are equivalent to cellularly embedded graphs in surfaces, orientable or non-orientable.
Ribbon graphs arise naturally from small neighborhoods of cellularly embedded graphs.
On the other hand, topologically, a ribbon graph is a surface with boundary.
Capping-off the holes results in a band decomposition in the surface, which gives rise to a cellularly embedded graph in the obvious way.
A ribbon graph is $\textit{orientable}$ if it is orientable when viewed as a punctured surface.
Two ribbon graphs are $\textit{equivalent}$ if there is a homeomorphism taking one to the other that preserves the vertex-edge structure.
The homeomorphism should be orientation preserving when the ribbon graphs are orientable.
Here we will not go into ribbon graph deeply, we refer to interested reader to \cite{Mo} and references in it.

\begin{figure}[!htbp]
\centering
\includegraphics[width=3.5in]{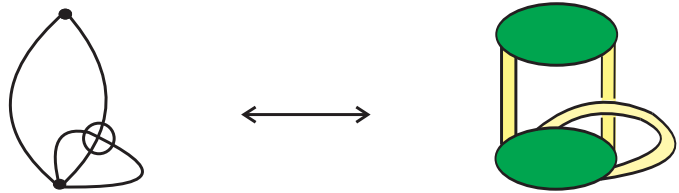}
\caption{The cyclic graph $G$ and its associated orientable ribbon graph $\textbf{G}_r$.}
\label{F:orientableribbon}
\end{figure}

\begin{figure}
\centering
\includegraphics[width=12cm]{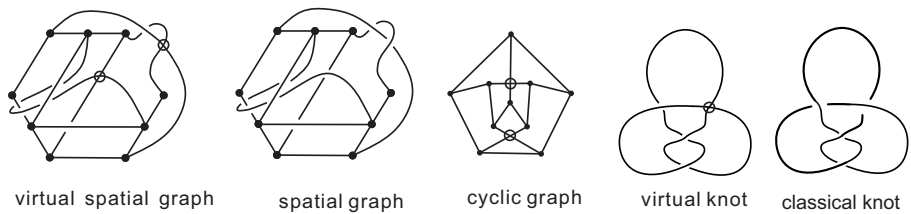}
\caption{Five types of objects.}
\label{F:fivegraphs}
\end{figure}

The theory of cyclic graphs is equivalent to the theory of graphs on orientable surfaces.
For an example of a cyclic graph and its associated orientable ribbon graph, see Figure \ref{F:orientableribbon}.
Note that there is a virtual crossing where two edges are traversing.
Figure \ref{F:fivegraphs} gives an example for the five kinds of virtual spatial graph diagrams. We define \textit{the number of boundary components of a cyclic graph $G$} to be the number of boundary components of the ribbon graph $\textbf{G}_r$ associated with $G$.

From a virtual spatial graph diagram $D$, an orientable ribbon graph $\textbf{D}_r$ (possibly non-connected) can be constructed naturally, i.e., a classical crossing in $D$ is considered as a disk of a degree 4 vertex.
Thus the vertex discs of $\textbf{D}_r$ correspond to classical crossings and vertices.
The edge discs correspond to arcs, where the arcs of a diagram start and end at classical crossings and vertices.
$D$ can be embedded into the closed surface $\textbf{F}_D$ formed by capping off the holes of $\textbf{D}_r$.
The genus of $\textbf{F}_D$ is called the genus of diagram $D$,
i.e., $g(D)=g(\textbf{F}_D)$. The reader can check that a formula of the genus of closed orientable surface $\textbf{F}_D$ is
\begin{eqnarray}
  g(\textbf{F}_D)&=&k(\textbf{D}_r)+\frac{|E(\textbf{D}_r)|-|V(\textbf{D}_r)|-bc(\textbf{D}_r)}{2},
  \label{E:genus}
\end{eqnarray}
where $|V(\textbf{D}_r)|$, $|E(\textbf{D}_r)|$, $k(\textbf{D}_r)$ and $bc(\textbf{D}_r)$ are the number of vertices, edges, connected components and boundary components of $\textbf{D}_r$, respectively. In $\textbf{D}_r$, classical crossings of $D$ are here seen as distinct from graphical vertices of $D$, virtual crossings of $D$ do not count as graphical vertices. Then the number of vertices of  $\textbf{D}_r$ includes both the number of graphical vertices and the number of classical crossings.
Let $|V(D)|$ and $c(D)$ be the number of vertices and classical crossings of $D$, respectively. Let $d(v)$ be the degree of vertex $v$ of $D$.
Then by Eq. (\ref{E:genus}), we get the genus of surface $\textbf{F}_D$ as follows:
\begin{eqnarray}
  g(\textbf{F}_D)&=&k(\textbf{D}_r)+c(D)+\frac{\sum_{v\in V(D)} d(v)}{4}-\frac{|V(D)|+c(D)+bc(\textbf{D}_r)}{2}.
\end{eqnarray}

Note that the formula is designed to not use edge counts.
Let $G$ be a virtual spatial graph.
The \textbf{genus} of $G$ is defined as $g(G)=min\{g(\textbf{F}_D)|D$ is any diagram of $G \}$.

\begin{rem}\label{R:generalresults}
It is well known that another motivation for the introduction of virtual knots is that all Gauss codes can be implemented.
Another motivation for the introduction of virtual spatial graphs is that all Gauss codes of graphs can be implemented.
For the definition of the Gauss code for spatial graphs and related results, see \cite{FlemingMellor,KKKP}.
\end{rem}

\section{The generalized Yamada polynomial of a virtual spatial graph}\label{S:GYP}

\noindent

In \cite{Fleming07}, Fleming and Meller followed Yamada's method in \cite{Yam} to define the Yamada polynomial of a virtual spatial graph diagram $D$.
After the deletion and contraction of edges, $D$ becomes a bouquet graph (cyclic graph with one vertex).
They defined that the Yamada polynomial of a bouquet graph is equal to the Yamada polynomial of the underlying bouquet graph,
i.e., the cyclic order of the edges incident to the vertex of bouquet is ignored. Such a calculation method cannot distinguish the non-equivalence between the two virtual spatial graphs in Figure \ref{twotheta}.

\begin{figure}
\centering
\includegraphics[width=8cm]{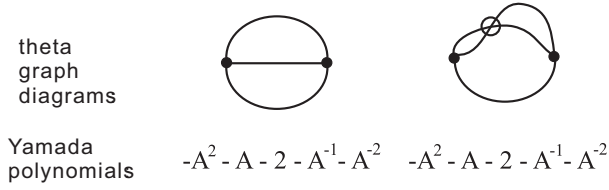}
\caption{Two theta graphs with the same Yamada polynomial.}
\label{twotheta}
\end{figure}

In this section, we want to use an idea commonly used in ribbon graph theory, that is, to count the number of boundary components of a cyclic graph to define a polynomial of a virtual spatial graph diagram.
In addition, we shall add several new variables to define a multivariate polynomial.
We call this the \textbf{generalized Yamada polynomial}.
We mainly are interested in the generalized Yamada polynomial of a virtual rigid vertex graph,
although the generalized Yamada polynomial can become a polynomial of a special type of a virtual pliable vertex graph, where the maximum degree of a virtual pliable vertex graph is at most 3.

We first define a polynomial with 7 variables associated with the virtual spatial graph diagram by recursion.
This polynomial must be specialized so that it can be a rigid vertex isotopy invariant.

Let $e$ be an edge of the underlying abstract graph of a virtual spatial graph diagram $D$.
We use $D-e$, $D/e$, and $D(\overline{e})$ to denote deletion, contraction, and marking of the edge $e$ in a virtual spatial graph diagram $D$, respectively.
$D(\overline{e})$ indicates that the edge $e$ is preserved in $D$ and the edge $e$ is marked with a short slash.

It is convenient to introduce a short-hand for the evaluation of the generalized Yamada
polynomial: the virtual spatial graph diagram $D$ itself will stand for the evaluation of $R(D)$.

\begin{defn}\label{defn}
Let $D$ be a virtual spatial graph diagram.
Let $\alpha,\beta,\gamma,x,y,z,\mu$ be the variables, and assume that the product $\alpha\beta\gamma xyz\mu \neq 0$.
The generalized Yamada polynomial of $D$, $R(D)=R(D;\alpha,\beta,\gamma,x,y,z,\mu)$, can be recursively defined as follows:
\begin{enumerate}
\item [(1)] If $D$ contains degree 2 vertices, and $D'$ is the graph diagram obtained after absorbing all degree 2 vertices (i.e., reserve half edges incident with the degree 2 vertex and erase the degree 2 vertex), then $R(D')=R(D)$.
\item [(2)] \Across = $\alpha$ \Bsmooth+ $\beta$ \Asmooth+ $\gamma$ \onevertex.
\item [(3)] \Bcross = $\alpha$ \Asmooth+ $\beta$ \Bsmooth+ $\gamma$ \onevertex.
\item [(4)] \onee\ =\ $x$ \deleteedge+ $y$ \markingedge\ .
\item [(5)] If each edge of a cyclic graph $D$ is a marked edge, then $R(D)=z^{k(D)}\mu^{bc(D)-k(D)}$, where $k(D)$ and $bc(D)$ are the number of connected components and boundary components of $D$, respectively.
In particular, if $D$ is an $n$-th order empty graph ($n$ disjoint vertices), then $R(D)=z^n$ and $R(\emptyset)=1$.
\end{enumerate}
\end{defn}

When a virtual spatial graph is a 2-regular graph, it becomes a virtual knot or link.
To define the polynomial invariants of a virtual spatial graph diagram, in the edges of the graph, absorbing or adding degree 2 vertices does not change the value of the polynomial.
Each diagram in the equations means that only the drawn part is different and the other parts are the same. Note that the marker is on y-coefficient graph in Definition \ref{defn} (4). There can also be some virtual crossings on the edge $e$, but for simplicity and convenience, we do not draw them in Definition \ref{defn} (4).
For Definition \ref{defn} (5), for example, suppose $D$ is an all-marked connected cyclic graph with 5 boundaries components, then $R(D)=z\mu^4$.

There are two natural ways of combining two virtual spatial graph diagrams $D_1$ and $D_2$.
The first, the \textit{disjoint union} $D_1\cup D_2$, is exactly the same as for graphs.
The second, the \textit{one-point join} $D_1\vee D_2$, is formed by taking disjoint virtual spatial graph diagrams $D_1$ and $D_2$, choosing any vertices $v_1\in V(D_1)$ and $v_2\in V(D_2)$, and uniting the vertices $v_1$ and $v_2$, combining their cyclic orders so that all the edges of $v_1$ and $v_2$ are consecutive in the new order.
The following proposition is obvious.

\begin{proposition}\label{P:cupvee}
Let $D_1$ and $D_2$ be two virtual spatial graph diagrams. Then
\begin{enumerate}
\item [(1)] $R(D_1\cup D_2)=R(D_1) R(D_2),$
\item [(2)] $R(D_1\vee D_2)=z^{-1} R(D_1) R(D_2).$
\end{enumerate}
\end{proposition}
\begin{proof}
(1) is trivial. (2) follows from the fact that the one-point join of two cyclic graphs ensures that two boundary components will be merged into one component via Definition \ref{defn} (5).
\end{proof}

In a virtual spatial graph diagram, a non-loop edge that does not intersect other edges at classical or virtual crossings will be called a \textit{normal edge}.
Otherwise, it is called an \textit{abnormal edge}. It is easy to perform the edge contraction operation for the normal edge, but the contraction of an abnormal edge is a less intuitive task because of the need to take into account the vertex cyclic order. This is also the reason why we introduce the marking edge operation.

\begin{lem}\label{markcontract}
Let $D$ be a virtual spatial graph diagram and $e$ be a normal edge of $D$.
Then $R(D(\overline{e}))=R(D/e).$
\end{lem}
\begin{proof}
According to the Definition \ref{defn} (5) of the generalized Yamada polynomial of a virtual spatial graph diagram, the contribution of marking a normal edge \makrboundary ~to the boundary component is equivalent to the contraction of the edge.
\end{proof}

\begin{rem}
For the normal edge, marking operations and contraction operations are consistent.
Using $R(D/e)$ to represent $R(D(\overline{e}))$ will simplify the process of diagrammatic calculation of the generalized Yamada polynomial of a virtual spatial graph diagram $D$.
\end{rem}

Next, in order to ensure that this polynomial $R(D;\alpha,\beta,\gamma,x,y,z,\mu)$ is invariant under regular isotopy, the value of some variables and relationship between the variables are determined.

\begin{proposition}\label{P:calculate}
\begin{itemize}
\item[(1)] $\selfloop=(x+y\mu)z$
\item[(2)] $\loopvertex=(x+y\mu) \downarc$
\item[(3)] $\twovertex =  x^2 \Bsmooth + (2xy +y^2\mu) \onevertex$
\end{itemize}
\end{proposition}
\begin{proof}
\begin{itemize}
\item[(1)]
\begin{eqnarray*}
&\selfloop&=x~ \vertex  +  y~\loopmaker\\
& &=(x+y\mu)z
\end{eqnarray*}
\item[(2)] According to Proposition \ref{P:cupvee} (1) and above identity, we have
\begin{eqnarray*}
&\loopvertex&= z^{-1}\times \selfloop\times \downarc \\
& &=(x+y\mu)\downarc
\end{eqnarray*}
\item[(3)] \begin{eqnarray*}
&\twovertex &= x \oneedge  +  y\markedge \\
& &= x^2 \Bsmoothvertex +xy \onerightmark +xy \oneleftmark +y^2 \twomark\\
& &=  x^2 \Bsmooth + (2xy +y^2\mu) \onevertex
\end{eqnarray*}
\end{itemize}
\end{proof}

Proposition \ref{P:calculate} implies that the generalized Yamada polynomial of the standard diagram $\bigcirc$
 of the trivial knot is $(x+y\mu)z$.

\begin{lem}\label{RM2}
To ensure that $R(D)$ is invariant under Reidemeister moves (II) and (III),
then $\beta=\alpha^{-1},$ $\gamma=x^{-1}$, $y=1$, $z=-x^{-1}$ and $\mu=-(\alpha+\alpha^{-1}+2)x.$
\end{lem}
\begin{proof}
On the one hand,
\begin{eqnarray*}\label{r2invariant}
\ReidMove
&=& \alpha\beta \Asmooth  + [\alpha^2+\beta^2+\alpha\beta(xz+y\mu z)]\Bsmooth
     +(\alpha+\beta)\gamma \onevertex\\
& &+ \alpha\gamma\downvertexloop + \beta\gamma\upvertexloop+ \gamma^2 \twovertex \\
&=&\alpha\beta \Asmooth  + [\alpha^2+\beta^2+\alpha\beta(xz+y\mu z)+
   \gamma(\alpha+\beta)(x+y\mu)+\gamma^2x^2]\Bsmooth\\
& &+[(\alpha+\beta)\gamma+\gamma^2(2xy+y^2\mu)]\onevertex  \\
\end{eqnarray*}
In order to have $R(\ReidMove)=R(\Asmooth)$, we take the following relations:
\begin{equation}
\left\{ \begin{array}{ll}
  \alpha\beta = 1 \\
  \alpha^2+\beta^2+\alpha\beta(xz+y\mu z)+\gamma(\alpha+\beta)(x+y\mu)+\gamma^2x^2=0 \\
  (\alpha+\beta)\gamma+\gamma^2(2xy+y^2\mu) =0
\end{array}
\right.  \label{eq:rm2}
\end{equation}
Thus $\beta=\alpha^{-1}.$

On the other hand,
\begin{eqnarray*}
\threevertex
&=& \alpha^2 \aav  + \alpha^{-2} \bbv + \abv + \bav + \alpha\gamma\av+\alpha^{-1}\gamma \bv
   +\alpha^{-1}\gamma \vb \\
& & +\alpha\gamma\va+\gamma^2\vv  \\
&=& \alpha^2 \aav  + \alpha^{-2} \bbv + \abv + (x+y\mu)\onearc + \alpha\gamma(x \twothree+y\five)\\
& &+ \alpha^{-1}\gamma [x^2\onearc +(2xy+y^2\mu)\bbv] +\alpha^{-1}\gamma (x \threetwo+y\five)\\
& &+\alpha\gamma[x^2\onearc +(2xy+y^2\mu)\aav]+\gamma^2[x^3\onearc +(y^3\mu+3xy^2) \five\\
& &+ x^2y(\bbv+\aav+\crossarc)]  \\
&=&[\alpha^2+\alpha\gamma(2xy+y^2\mu) +\gamma^2x^2y] \aav
  + [\alpha^{-2}+\alpha^{-1}\gamma(2xy+y^2\mu) +\gamma^2x^2y] \bbv \\
& &+ \abv + \alpha\gamma x \twothree+\alpha^{-1}\gamma x \threetwo+[ \alpha\gamma y +\alpha^{-1}\gamma y+\gamma^2(y^3\mu+3xy^2)] \five \\
& &+\gamma^2x^2y\crossarc+ [(x+y\mu)+\alpha^{-1}\gamma x^2+\alpha\gamma x^2+\gamma^2x^3] \onearc
\end{eqnarray*}
\begin{eqnarray*}
\upthree
&=& \alpha \twothree  + \alpha^{-1} \threetwo +\gamma x\abv + \gamma y \five\\
\end{eqnarray*}
In order to have $R(\threevertex)=R(\upthree)$, we take
\begin{equation}
\left\{ \begin{array}{ll}
  \gamma x=1\\
  x+y\mu+\alpha^{-1}\gamma x^2+\alpha\gamma x^2+\gamma^2x^3 =0 \\
  \alpha^2+\alpha\gamma(2xy+y^2\mu) +\gamma^2x^2y=0 \\
   \alpha^{-2}+\alpha^{-1}\gamma(2xy+y^2\mu) +\gamma^2x^2y=0\\
    \alpha\gamma y +\alpha^{-1}\gamma y +\gamma^2(y^3\mu+3xy^2)=\gamma y
\end{array}
\right.  \label{eq:rm3}
\end{equation}
Solving the system of Eq. (\ref{eq:rm3}), we firstly obtain that $\gamma=x^{-1}$.
Based on  $\gamma=x^{-1}$, we can simplify Eq. (\ref{r2invariant}) and Eq. (\ref{eq:rm3}) to be the following equations:

\begin{equation}
\left\{ \begin{array}{ll}
  \alpha^2+\alpha^{-2}+(xz+y\mu z)+x^{-1}(\alpha+\alpha^{-1})(x+y\mu)+1=0 \\
  (\alpha+\alpha^{-1})+2y+x^{-1}y^2\mu =0\\
  (\alpha+\alpha^{-1})x+2x+y\mu=0 \\
  \alpha^2+2\alpha y+\alpha x^{-1}y^2\mu +y=0 \\
   \alpha^{-2}+2\alpha^{-1} y+\alpha^{-1} x^{-1}y^2\mu +y=0 \\
    (\alpha+\alpha^{-1})y +x^{-1}(y^3\mu+3xy^2)=y
\end{array}
\right.  \label{eq:simplify1}
\end{equation}
Solving the second and third identities of the system of Eq. (\ref{eq:simplify1}), we get that $\mu=-\frac{(\alpha+\alpha^{-1})+2y}{x^{-1}y^2}=-\frac{(\alpha+\alpha^{-1})x+2x}{y}.$
Meanwhile it can be solved that
\begin{equation}
\left\{ \begin{array}{ll}
  y=1\\
 \mu =(-\alpha-\alpha^{-1}-2)x
\end{array}
\right.  \label{eq:rm4}
\end{equation}
Hence we can figure out that $z=-x^{-1}$ according to the first identity of Eq. (\ref{eq:simplify1}), and all other identities of Eq. (\ref{eq:simplify1}) are correct when $y=1$ and $\mu =(-\alpha-\alpha^{-1}-2)x.$ This ensures that $R(D)$ is invariant under move (II).

\noindent If the vertex degree is at least 3, then by the definition and induction hypothesis of the generalized Yamada polynomial, we have
\begin{align*}
  R(\ndv) & =R(\ndvae)-xR(\ndvde) \\
   & = R(\ndvaes)-xR(\ndvdes) \\
   &=R(\ndvs)\\
\end{align*}
This ensures that $R(D)$ is invariant under move (IV).
From the definition of $R(D)$ and its invariance under the moves (II) and (IV), we can obtain that
\begin{eqnarray*}
\Gte&=&\alpha\graphone +\alpha^{-1}\graphtwo+ \gamma\graphthree \\
&=&\alpha\graphfour +\alpha^{-1}\graphfive+ \gamma\graphsix \\
&=&\Gts
\end{eqnarray*}
Hence, $R(D)$ is invariant under move (III).
\end{proof}

\begin{thm}\label{regiso}
When variables satisfy that
$$(\alpha,\beta,\gamma,x,y,z,\mu)=(\alpha,\alpha^{-1},x^{-1},x,1,-x^{-1},
-(\alpha+\alpha^{-1}+2)x),$$
then $R(D)$ is a regular isotopy invariant. That is, the generalized Yamada polynomial $R(D;\alpha,x)$ of a virtual spatial graph diagram is a two-variable polynomial and a regular isotopy ((II)(III) and Figure \ref{F:virtualmoves}) invariant.
\end{thm}
\begin{proof}
It is obvious that $R(D)$ is invariant under moves except (IV*) in Figure \ref{F:virtualmoves}.
Furthermore, $R(D)$ is invariant under moves (IV*) which is similar to the case of move (III).
Thus $R(D)$ is a regular isotopy invariant.
\end{proof}

We reformulate the Definition \ref{defn} as follows:
\begin{defn}\label{definition1}
Let $D$ be a virtual spatial graph diagram. Then the generalized Yamada polynomial $R(D;\alpha,x)$ satisfies the following relations:
\begin{enumerate}
\item [(1)] If $D$ contains degree 2 vertices, and $D'$ is the graph diagram obtained after absorbing all degree 2 vertices, then $R(D')=R(D).$
\item [(2)] \Across = $\alpha$ \Bsmooth+ $\alpha^{-1}$ \Asmooth+ $x^{-1}$ \onevertex.
\item [(3)] \Bcross = $\alpha$ \Asmooth+$\alpha^{-1}$ \Bsmooth+ $x^{-1}$ \onevertex.
\item [(4)] \onee\ =\ $x$ \deleteedge+ \markingedge\ .
\item [(5)] If each edge of a cyclic graph $D$ is a marked edge, then $R(D)=(-x)^{bc(D)-2k(D)}(\alpha+\alpha^{-1}+2)^{bc(D)-k(D)}$, where $k(D)$ and $bc(D)$ are the number of connected components and boundary components of $D$, respectively. In particular, if $D$ is an $n$-th order empty graph ($n$ disjoint vertices), then $R(D)=(-x)^{-n}$ and $R(\emptyset)=1$.
\end{enumerate}
\end{defn}

In the following, for a virtual spatial graph diagram $D$, the generalized Yamada polynomial of $D$ will always mean $R(D;\alpha,x)$. Let $e$ be an edge of a virtual spatial graph  $G$. We say that $e$ is an \textit{isthmus} of $G$ or its diagram $D$ if $G-e$ as abstract graph has more connected components than $G$ as abstract graph.

\begin{proposition}\label{pro1}
If a virtual spatial graph diagram $D$ contains a normal isthmus, then $R(D)=0.$
\end{proposition}
\begin{proof}
Let $e$ be a normal isthmus of $D$. Then there exists $D_1$, $D_2$ such that $D-e=D_1 \cup D_2$ and $D/e=D_1 \vee D_2$.
By the definition of generalized Yamada polynomial, when $xz=-1$, $R(D)=x R(D-e)+R(D/e)=(xz+1)R(D_1 \vee D_2)=0$.
\end{proof}

Just as \cite{Yam}, the behaviors of $R(D;\alpha,x)$ under moves (I), (V) and (VI) are illustrated in Figure \ref{F:yamada}.
The generalized Yamada polynomial $R(D;\alpha,x)$ is changed under moves (I), (V) and (VI).

\begin{figure}
\centering
\includegraphics{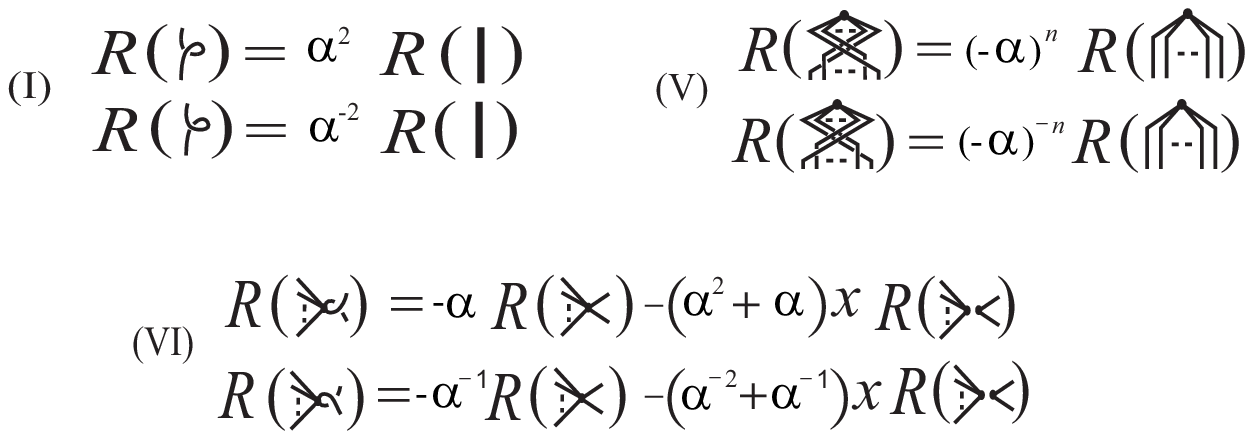}
\caption{The behaviors of $R(D)$ under moves (I), (V) and (VI). In move (V), the degree of vertex is $n$.}
\label{F:yamada}
\end{figure}

Let $G$ be a virtual spatial graph and $D$ be a diagram of $G$.
Based on the formulas of Figure \ref{F:yamada},
the generalized Yamada polynomial can be normalized as $\overline{R}(G;\alpha,x)=(-\alpha)^{-m}R(D;\alpha,x)$, where $m$ is the minimum degree of $\alpha$ in $R(D)$.
Then we can obtain the following result.
\begin{thm}
Let $G$ be a virtual spatial graph. Then $\overline{R}(G;\alpha,x)$ is a rigid vertex isotopy invariant.
When the maximum degree of $G$ is at most 3, $\overline{R}(G;\alpha,x)$ is a pliable vertex isotopy invariant.
\end{thm}
\begin{proof}
The moves (I) and (V) do not change the value of $\overline{R}(G;\alpha,x)$.
But under the move (VI) it will change the value of $\overline{R}(G;\alpha,x)$, so it is not a pliable vertex isotopy invariant in a general case.
However, there are exceptions.
When the maximum degree of the graph is less than or equal to 3, the result of the move (VI) at a degree 3 vertex is equivalent to the result of the moves (V) and (IV).
Thus, when the maximum degree of the graph is at most 3, $\overline{R}(G;\alpha,x)$ is a pliable vertex isotopy invariant.
\end{proof}

Let $L$ be an oriented virtual link and $D$ be a diagram of $L$.
The writhe of $D$, which is denoted by $w(D)$, is defined to the sum of the signs of classical crossings of $D$.
If $D$ has no classical crossings, then $w(D)=0$.
We defined a polynomial for $L$ as follows.

\begin{eqnarray}
\widetilde{R}(L;\alpha,x)&=&\alpha^{-2w(D)}R(D;\alpha,x).
\end{eqnarray}

Then, we have the following result.
\begin{thm}
Let $L$ be an oriented virtual link.
Then $\widetilde{R}(L;\alpha,x)$ is an ambient isotopy invariant.
\end{thm}

According to the Definition \ref{definition1}, the following result is obvious.
\begin{proposition}
Let $D^*$ be the mirror image of a virtual spatial graph diagram $D$. Then $R(D^*;\alpha,x)=R(D;\alpha^{-1},x)$.
\end{proposition}
This proposition implies the following two theorems.
\begin{thm}
Let $G$ be a virtual rigid vertex graph.
If $G$ is amphicheiral (i.e., rigid vertex isotopic to the mirror image of itself), then $\overline{R}(G;\alpha,x)=(-\alpha)^{d}\overline{R}(G;\alpha^{-1},x)$,
where $d$ is the degree of $\alpha$ in $\overline{R}(G;\alpha,x)$.
\end{thm}

\begin{thm}
Let $G$ be a virtual pliable vertex graph whose maximum degree is at most 3.
If $G$ is amphicheiral (i.e., pliable vertex isotopic to the mirror image of itself), then $\overline{R}(G;\alpha,x)=(-\alpha)^{d}\overline{R}(G;\alpha^{-1},x)$,
where $d$ is the degree of $\alpha$ in $\overline{R}(G;\alpha,x)$.
\end{thm}
Next we give the state-space expansion of the generalized Yamada polynomial of a virtual spatial graph diagram.
\begin{figure}
\centering
\includegraphics[width=2.5in]{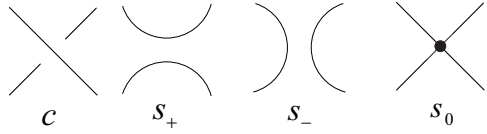}
\caption{The spin set.}
\label{F:spin}
\end{figure}
Let $D$ be a virtual spatial graph diagram and the number of classical crossings of $D$ is $c$.
For a classical crossing of $D$, we define $s_+$, $s_-$ and $s_0$, called the \textit{spin} of $D$, as shown in Figure \ref{F:spin}.
Let $S$ be the cyclic graph obtained from $D$ by replacing each classical crossing with a spin of $\{s_+, s_-, s_0\}$.
$S$ is called a state of $D$, and $\mathscr{S}(D)$ denotes the set of all states of $D$.
If the number of spins $s_+$ and $s_-$ in the state $S$ is $a$ and $b$, respectively, then the number of spins $s_0$ is $c-a-b$.
For a cyclic graph $S=(V(S),E(S))$, we use $|V(S)|$ and $|E(S)|$ to denote the number of vertices and edges in the graph $S$, respectively.
When the context is clear, $|V(S)|$ and $|E(S)|$ are abbreviated as $|V|$ and $|E|$, respectively.
For each $F\subseteq E(S)$, let $\langle F\rangle$ be the spanning subgraph whose vertex set is $V(S)$ and the edge set is $F$.
Let $|F|$ be the number of edges in $F$.
Let $k(F)$ be the number of connected components of $\langle F\rangle$ (that is, the zero-th Betti number of the graph).
Let $bc(F)$ be the number of boundary components of $\langle F\rangle$.
Let $E(S)\backslash F$ be the complement of $F$ in $E(S)$, i.e., $E(S)\backslash F=E(S)-F$.

By Definition \ref{definition1} (4) and (5), we have the \textit{spanning subgraph expansion} of the generalized Yamada polynomial of a cyclic graph $S$ as follows:
$$R(S;\alpha,x)=\sum_{F\subseteq E(S)}
x^{|E(S)\backslash F|}z^{k(F)}\mu^{bc(F)-k(F)}.$$
That is,

\begin{equation}\label{E:spanningsubgraph}
\begin{split}
R(S;\alpha,x)
&=\sum_{F\subseteq E(S)}(-1)^{bc(F)}x^{|E|-|F|+bc(F)-2k(F)}(\alpha+\alpha^{-1}+2)^{bc(F)-k(F)}\\
&=\sum_{F\subseteq E(S)}(-1)^{bc(F)}x^{|E|-|V|+|V|-|F|+bc(F)-2k(F)}(\alpha+\alpha^{-1}+2)^{bc(F)-k(F)}\\
&=\sum_{F\subseteq E(S)}(-1)^{bc(F)}x^{|E|-|V|-2g(F)}(\alpha+\alpha^{-1}+2)^{bc(F)-k(F)}\\
&=x^{|E|-|V|}\sum_{F\subseteq E(S)}(-1)^{bc(F)}x^{-2g(F)}(\alpha+\alpha^{-1}+2)^{bc(F)-k(F)},
\end{split}
\end{equation}
where $|V|-|F|+bc(F)=2k(F)-2g(F)$ is the formula for the genus in an orientable ribbon graph (cyclic graph).

Then we can get the \textit{state-space expansion} of the generalized Yamada polynomial of a virtual spatial graph diagram $D$ as follows:
\begin{eqnarray}\label{E:Dspanningsubgraph}
R(D;\alpha,x)=\sum_{S\in \mathscr{S}(D)}\alpha^{a-b}x^{a+b-c}R(S;\alpha,x).
\end{eqnarray}

When $D$ is a virtual link diagram, we note that each cyclic graph $S$ is a 4-regular graph with $|V|=c-a-b$, $|E|=2(c-a-b)$.
By Eq. (\ref{E:spanningsubgraph}), we can simplify Eq. (\ref{E:Dspanningsubgraph}) as follows:
\begin{eqnarray}\label{E:Dspanningsubgraph1}
R(D;\alpha,x)=\sum_{S\in \mathscr{S}(D)}\alpha^{a-b}\sum_{F\subseteq E(S)}(-1)^{bc(F)}x^{-2g(F)}(\alpha+\alpha^{-1}+2)^{bc(F)-k(F)}.
\end{eqnarray}
Obviously we have the following result.
\begin{corollary}\label{C:powerofx}
Let $D$ be a virtual link diagram. Then the powers of $x$ in $R(D;\alpha,x)$ are even.
\end{corollary}

In \cite{Yam}, the Yamada polynomial $R(D)(A)$ of a spatial graph diagram $D$ has the following state-space expansion:
\begin{eqnarray}\label{E:cyp}
R(D)(A)=\sum_{S\in \mathscr{S}(D)}A^{a-b}\sum_{F\subseteq E(S)}(-1)^{k(F)+n(F)}(A+A^{-1}+2)^{n(F)}.
\end{eqnarray}

Note that each state of a spatial graph diagram is a plane graph.
If $S$ is a plane graph, then $g(F)=0$ and $k(F)+n(F)=bc(F)$ for any spanning subgraph $\langle F\rangle$ of $S$.
Hence we have the following corollary by Eq. (\ref{E:Dspanningsubgraph}) and  Eq. (\ref{E:cyp}).

\begin{corollary}\label{C:classicalYP}
Let $D$ be a spatial graph diagram. Then $R(D;A,1)$ coincides with the Yamada polynomial for $D$.
\end{corollary}

For the two theta spatial graph diagrams in the Figure \ref{twotheta}, their generalized Yamada polynomials $R(G;A,1)$ are: $-A^2-A-2-A^{-1}-A^{-2}$ and $2(A+A^{-1}+1)$.
Thus their normalized generalized Yamada polynomials $\overline{R}(G;A,1)$ are: $-A^4-A^3-2A^2-A-1$ and $2(A^2+1+A)$.
Hence they are neither rigid vertex isotopic nor pliable vertex isotopic virtual spatial graphs.

\begin{rem}
If $D$ is a virtual spatial graph diagram, the generalized Yamada polynomial $R(D;A,1)$ usually does not coincide with the Yamada polynomial for $D$.
\end{rem}

\section{The relationship between generalized Yamada polynomial and a link polynomial}\label{S:Dubrovnik}
\noindent

Kauffman \cite{Kau2-variable} had defined a regular isotopy invariant for a classic link diagram $l$ (2-regular spatial graph).
$D(l;a,z)$ is a special case of the Kauffman polynomial.
$D(l;a,z)$ is called Dubrovnik polynomial.
$D(l;a,z)$ can be obtained from the following skein relation.
\begin{enumerate}
\item [(1)] $D(\Across)-D(\Bcross)= z [D(\Bsmooth)-D(\Asmooth)]$
\item [(2)] $D(\positivecurl)=aD(\downarc)$
\item [(3)] $D(\negativecurl)=a^{-1}D(\downarc)$
\item [(4)] $D(\bigcirc K)=[1+\frac{(a-a^{-1})}{z}]D(K)$,~~~$D(\bigcirc)=1+\frac{(a-a^{-1})}{z}.$
\end{enumerate}
The above definition is different from the original definition, with $D(\bigcirc)=1$ in the original definition.
According to the Definition \ref{definition1} of the generalized Yamada polynomial $R(l;\alpha,x)$ of the virtual spatial graph diagram, the following skein relations can be obtained.
\begin{enumerate}
\item [(1)] $R(\Across)-R(\Bcross)= (\alpha-\alpha^{-1}) [R(\Bsmooth)-R(\Asmooth)]$
\item [(2)] $R(\positivecurl)=\alpha^2 R(\downarc)$
\item [(3)] $R(\negativecurl)=\alpha^{-2}R(\downarc)$
\item [(4)] $R(\bigcirc K)=(1+\alpha+\alpha^{-1})R(K)$,~~~$R(\bigcirc)=1+\alpha+\alpha^{-1}$
\end{enumerate}
Thus $R(l;\alpha,x)$ satisfies the recursive definition of the Dubrovnik polynomial.

\begin{lem}\label{l:detectclassical}
Let $l$ be a classical link diagram, then $R(l;\alpha,x)=D(l;\alpha^2,\alpha-\alpha^{-1}),$ where $a=\alpha^2$ and $z=\alpha-\alpha^{-1}$ in the Dubrovnik polynomial.
That is, for a classical link diagram $l$, there is no $x$ term in $R(l;\alpha,x)$.
\end{lem}

\begin{proof}
Note that for classical links, variable $x$ does not appear because we can just calculate $R(l;\alpha,x)$ from skein relations above.
\end{proof}

\begin{rem}\label{R:detectclassical}
The above Lemma \ref{l:detectclassical} can often be used to determine if a virtual link is non-classical.
\end{rem}

It is well known that the Dubrovnik polynomial cannot be generalized directly to virtual links because its skein relations are not sufficient to calculate the polynomials of virtual link diagrams.
Note that some virtual diagrams can not be unknotted or unlinked by switching classical crossings.
Nevertheless, the generalized Yamada polynomial $R(l;\alpha,x)$ specializes to a version of the Dubrovnik polynomial for virtual link diagrams.
That is, we can first resolve classical crossings via skein relations until we can not continue to use the skein relations.
Then we turn to formulas (2)(3) in Definition \ref{definition1} to resolve the remaining classical crossings.
In this way the graphical expansion of the generalized Yamada polynomial provides a method to decide evaluations that are ambiguous to just the Dubrovnik skein relation.
These structures may act as a guide to finding a full formulation of the Dubrovnik polynomial for virtual knots and links.

By calculating, we have
$R(0_1;\alpha,x)=R(\bigcirc)=\alpha+\alpha^{-1}+1.$
Let \virtualtrefoil ~be a diagram of virtual trefoil knot (the mirror of virtual knot $2_1$ in \cite{Gr}), then
$R(\virtualtrefoil;\alpha,x)=\alpha^3+3\alpha^2+5\alpha-4\alpha^{-1}-3\alpha^{-2}+\alpha^{-4}-x^{-2}(\alpha-\alpha^{-1})$, which contains $x$.
Thus virtual trefoil knot is non-trivial and non-classical.
Let \virtualizedtrefoil ~be a diagram of virtualized trefoil knot (virtual knot $3_7$ in \cite{Gr}), then
\begin{eqnarray*}
R(\virtualizedtrefoil;\alpha,x)&=&-\alpha^{-4}-2\alpha^{-3}+3\alpha^{-2}+2\alpha^{-1}+\alpha-2\alpha^{2}+2\alpha^{3}+\alpha^{4}-\alpha^{5}\\
& &+x^{-2}(\alpha^{-2}-3\alpha^{-1}+3\alpha-\alpha^{2}),
\end{eqnarray*}
which contains $x$.
Thus virtualized trefoil knot is non-trivial and non-classical.
We are sure that the generalized Yamada polynomial detects non-classicality for many link diagrams with a single virtualization that have unit Jones polynomial (it was generally proved by Dye \emph{et al.} in \cite{DKK}), but it is a project to compute more of them. We would like to find a general theorem about this situation.
Note that $D(\virtualhopflink)$ cannot be calculated with the above definition of Dubrovnik polynomial.
However, we can obtain the polynomial of a virtual Hopf link by the recursive definition of the generalized Yamada polynomial $R(l;\alpha,x)$.
Let \virtualhopflink ~be a diagram of virtual Hopf link, then $R(\virtualhopflink;\alpha,x)=
(\alpha+\alpha^{-1}+1)^2+(\alpha+\alpha^{-1}+2)-x^{-2}$, which contains $x$.
Thus the virtual Hopf link is non-trivial and non-classical.

\section{Applications and calculations}\label{S:Application}
\noindent

In this Section, we consider a special case of $R(D;\alpha,x)$ which corresponds to an evaluation of the 2-cabled bracket polynomial with Jones-Wenzl projector $P_2$.

Kauffman defined the bracket polynomial of the classical link diagrams using the skein relations, giving a new model of the Jones polynomial.
In \cite{VKT}, Kauffman extended the bracket polynomial of classical links to a bracket polynomial for virtual links.
Thus, the classical Jones polynomial is extended to the Jones polynomial of virtual links.

For a virtual link diagram $D$, the $\textit{bracket polynomial}$ $\langle D\rangle=\langle D\rangle(A,B,d)$ \cite{VKT} can be defined recursively by using the following three relations:
\begin{equation*}
\centering
\includegraphics[width=2.0in]{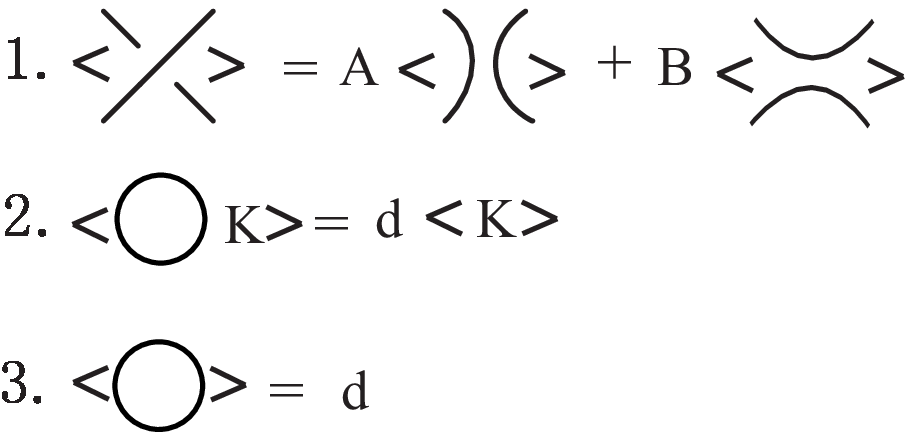}
\end{equation*}
Here it is understood that the three small diagrams are parts of otherwise identical larger diagrams and $\bigcirc$ denotes a connected diagram with no classical crossings.
There is also a state expansion for the bracket polynomial as follows.

\begin{thm}(\cite{VKT})
The bracket polynomial of a virtual link diagram $D$ is a polynomial in three variables $A,B,d$ defined by the formula
$$\langle D\rangle(A,B,d)=\sum_{\sigma\in \mathscr{S}(D)}A^{\alpha(\sigma)}B^{\beta(\sigma)}d^{|\sigma|-1},$$
\noindent where $\alpha(\sigma)$ and $\beta(\sigma)$ are the numbers of $A$-smoothings ($s_{+}$) and $B$-smoothings ($s_{-}$) (see the spin $s_{+}$ and $s_{-}$ in Figure \ref{F:spin}, respectively) in a state $\sigma$, respectively, $\mathscr{S}(D)$ denotes the set of states of $D$, and $|\sigma|$ is the number of closed curves in $\sigma$.
\end{thm}
In particular, it is well known that if $B=A^{-1}$, $d=-(A^2+A^{-2})$ then $\langle D\rangle$ is invariant under moves (II), (III) and moves (I*)-(IV*).
\vskip0.5cm

The $n$-strand Temperley-Lieb algebra $TL_n$ is an algebra in $\mathbb{C}[d]$;
the elements of the algebra are generated by 1-dimensional sub-manifolds of a rectangle.
Each sub-manifold intersects the top and bottom of the rectangle in $n$ vertices.
If the two sub-manifolds are isotopic keeping the boundary fixed, then we call them equivalent.
Deleting a simple closed curve is equivalent to generating a multiple of $d$.
The product between two sub-manifold\textcolor[rgb]{0.00,0.07,1.00}{s} is represented by a vertical stack of rectangles.
Let $\{1_n,e_1,e_2,...,e_{n-1}\}$ be the generator set of algebra $TL_n$ as illustrated in Figure \ref{F:TLAgenerators}.
Each element in $TL_n$ can be represented as a linear combination of the product of these generators.
For any $i,j=1,2,...,n-1$, the generators satisfy the following relations:
\begin{equation}\label{E:TLArelation}
\begin{split}
&1_ne_i=e_i=e_i1_n \\
&e_i^2=d~ e_i\\
&e_ie_j=e_je_i \ \ {if~} |i-j|\geq 2\\
&e_ie_{i\pm1}e_i=e_{i}\\
\end{split}
\end{equation}

\begin{figure}[!htbp]
\centering
\includegraphics[width=8cm]{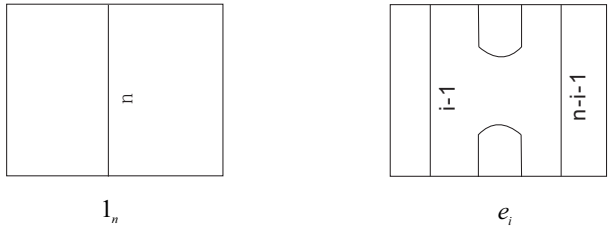}
\caption{The generator set of $TL_n$.}
\label{F:TLAgenerators}
\end{figure}

In $TL_2$, we have the Jones-Wenzl projector $P_2=1_2-\frac{1}{d}e_1$, and $P_2^{2}=P_2$, $P_2e_1=0$
as shown in Figure \ref{F:projector2}. Next we apply $P_2$ to the virtual spatial graph diagram to get a special parameter for the generalized Yamada polynomial.

\begin{figure}
\centering
\includegraphics[width=11cm]{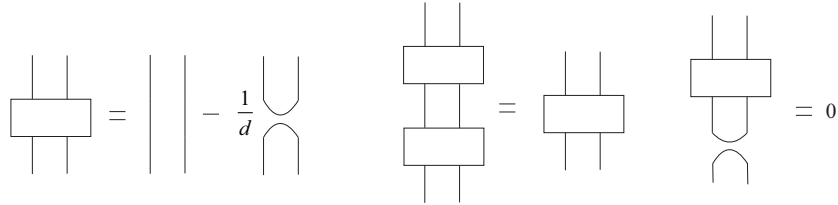}
\caption{The diagrammatic definition and properties of $P_2$.}
\label{F:projector2}
\end{figure}

With Jones-Wenzl projector $P_2$, we shall construct a new diagram from a virtual spatial graph diagram as follows.
First, for a virtual spatial graph diagram $D$, we shall resolve each vertex and crossing with $P_2$ and replace each arc of the graph as shown in Figure \ref{F:edgetranslate}.
After converting all vertices and crossing of $D$, we denote $D^{p}$ \textit{the virtual link diagram associated with $D$ with projector $P_2$}, see Figure \ref{F:G50} for an example.
Note that there are two $P_2$s on each arc of $D$ in $D^p$ and recall that the arcs of $D$ start and end at classical crossings and vertices.
In particular, when $D$ is $K_1$, i.e. $1-$th order empty graph, $D^{p}$ is a unknot (free loop).

\begin{figure}[ht]
\centering
\includegraphics[height=1.8cm]{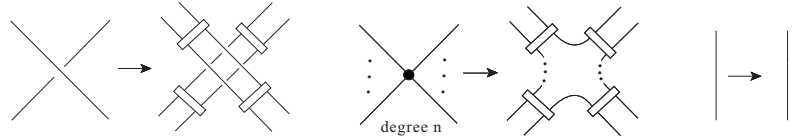}
\caption{Resolve each arc, each vertex and each classical crossing of a virtual spatial graph diagram.}
\label{F:edgetranslate}
\end{figure}

\begin{figure}[!htbp]
\centering
\includegraphics[width=8cm]{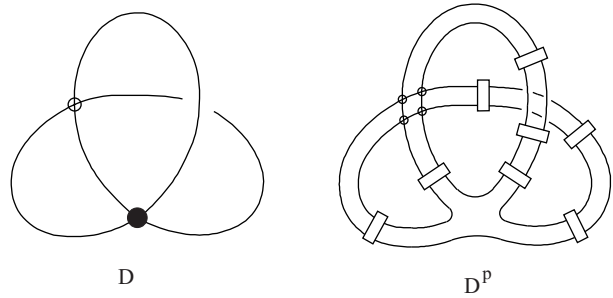}
\caption{With $P_2$, the virtual spatial graph $D$ is converted to $D^p$.}
\label{F:G50}
\end{figure}

We call $\langle D^p\rangle(A,A^{-1},d)$ the \textit{2-cabled bracket polynomial of $D$ with projector $P_2$}, where $d=-A^{2}-A^{-2}$.

In \cite{KauLins}, Kauffman and Lins discussed about applying $P_2$ to classical link diagrams and obtained a formula which is extended to virtual spatial graph diagrams here.

\begin{proposition}\label{P:PhiYamada}
\begin{eqnarray}\label{F:2strandformula}
\langle\cablecrossing\rangle &=& A^4 \langle\cableA\rangle +A^{-4}\langle\cableB\rangle+(A^{2}+A^{-2})\langle\cablev\rangle,
\end{eqnarray}
where the four small diagrams are parts of otherwise identical larger virtual link diagrams with projectors, and every term is the bracket polynomial of the diagram which is understood to fully expand each $P_{2}$ and represents a linear combination of some virtual link diagrams.
\end{proposition}
\begin{proof}
When there are virtual crossings in the diagram, Eq. (\ref{F:2strandformula}) is still true since Detour moves do not affect the action of projector $P_2$. So the proof is similar to  the original one in \cite{KauLins}.
\end{proof}

According to the definition of $P_2$ and the above Proposition \ref{P:PhiYamada}, the following theorem can be obtained.

\begin{thm}\label{T:PhiYamada1}
Let $D$ be a virtual spatial graph diagram and $D^p$ be the virtual link diagram associated with $D$ with projector $P_2$.
Then $$R(D;A^4,-d^{-1})=d\langle D^p\rangle(A,A^{-1},d),$$
where $d=-A^{2}-A^{-2}.$
\end{thm}

\begin{proof}
Note that during the process of the calculation of $d\langle D^p\rangle(A,A^{-1},d)$,  we shall first use the formula in Proposition \ref{P:PhiYamada} to deal with the classical crossings of $D$, then expand all projectors and record the number of closed curves.

It is sufficient to prove that the Definition \ref{definition1} of $R(D;A^4,-d^{-1})$ satisfies the calculation process of $d\langle D^p\rangle(A,A^{-1},d)$ .

(1) The degree 2 vertex has no effect on $R(D;A^4,-d^{-1})$ and the degree 2 vertex can also be removed in calculating $d\langle D^p\rangle(A,A^{-1},d)$ by using $P_2^2=P_2$.

(2) Reducing a classical crossing in $R(D;A^4,-d^{-1})$ is equivalent to reducing the corresponding four classical crossings in $D^p$ as the formula given in Proposition \ref{P:PhiYamada}.

(3) Reducing an edge in $R(D;A^4,-d^{-1})$ is equivalent to expanding the corresponding $P_2$ in $D^p$, that is, the deletion of an edge corresponds to choose $e_1$ of $P_2$ and the marking of an edge corresponds to choose $1_2$ of $P_2$.

(4) The relation $z=\mu=d$ also satisfies the value of a trivial knot and each circle in the calculation of $d\langle D^p\rangle$.
\end{proof}

\begin{rem}\label{R:Gp0}
According to Lemma \ref{l:detectclassical}, for a classical link diagram $l$, $R(l;A^4,-d^{-1})=D(l;A^8,A^4-A^{-4})=D(l;\alpha^2,\alpha-\alpha^{-1})_{|\alpha=A^4}$, which means that the power of each $A$ is a multiple of four.
Furthermore, for any virtual link diagram, we can obtain that the power of each $A$ of $R(l;A^4,-d^{-1})$ is a multiple of four by Corollary \ref{C:powerofx}.
\end{rem}

According to the relation $R(D;A^4,-d^{-1})=d\langle D^p\rangle(A,A^{-1},d)$, we write a program for calculating generalized Yamada polynomials based on \textit{Mathematica} code.

\noindent\textbf{Program 1:}\\
1. \text{rule1} $=\{Y[a\_, b\_, c\_, d\_, e\_, f\_, g\_, h\_,bc\_,de\_,fg\_,ha\_]$
$:>A^{-4}*X[a, b, c, d]*X[g, h, e, f] +A^{4}*X[a, b, g, h]*X[c, d, e, f] -dd*X[a, b, bc, ha]*X[c, d, bc, de]*X[e, f, fg, de]*X[g, h, ha, fg]\};$\\
2.  \text{rule2} $=\{X[a\_, b\_, c\_, d\_] :> del[a~ d]*del[b~ c]-1/dd*del[a~ b]*del[c~ d]\};$\\
3.  \text{rule3} $=\{del[a\_ b\_] del[b\_ c\_] :> del[a~c]\};$\\
4.  \text{rule4} $=\{(del[\_])^2 :> dd, del[\_^2] :> dd\};$\\
5.  \text{rule5} $=\{dd :> -A^2 - A^{-2}\};$\\
6.  \text{YamadaPoly}$[t\_]:=\\
\text{Simplify}[\text{Expand}[(t /. \text{rule1} /. \text{rule2} // \text{Expand}) //. \text{rule3} /.\text{rule4} /. \text{rule5}]$\\

\noindent where $X[a\_, b\_, c\_, d\_]$ denotes the 4 endpoints of a projector $P_2$ counterclockwise as shown in Figure \ref{F:codeingforD}(4)(usually $X[a\_, b\_, c\_, d\_]$ can be used to denote a classical crossing of a link diagram),
$Y[a\_, b\_, c\_, d\_, e\_, f\_, g\_, h\_,bc\_,de\_,fg\_,ha\_]$ denotes a classical crossing of a virtual spatial graph diagram as shown in Figure \ref{F:codeingforD}(4).
That is, there contains three steps for coding process in calculating generalized Yamada polynomial $R(D;A^4,-d^{-1})$ as shown in Figure \ref{F:codeingforD}(1)(2)(3).
(1) Put two projectors on each arc between two classical crossings and one projector on each arc between a classical crossing and a vertex in the diagram. (2) Use the letters $a,b,c,d$, etc. to mark the two sides of the arc between one classical crossing and another classical crossing and mark mark the region between every two half edges of each vertex in the diagram. (3) For each classical crossing in the diagram, by listing the information of the four semi-arcs forming the classical crossings,
we record the projector labels counter-clockwise which start from over-crossing half arc and firstly meets the A-region as follows:
$Y[d,c,b,a,e,$
$g,h,j,cb,ae,gh,jd]$,
where the newly last 4 labels will be used to
calculate the state graphs obtained after choosing spin $s_0$ for the classical crossing as shown in last term of identity of Proposition 5.2.
For the projector on the arc between two vertices in the graph,
we consecutively record the labels near a vertex counter-clockwise, and list the 4 endpoints of the projector as follows: $X[g,f,k,h]$.
\begin{figure}[!htbp]
\centering
\includegraphics[width=12cm]{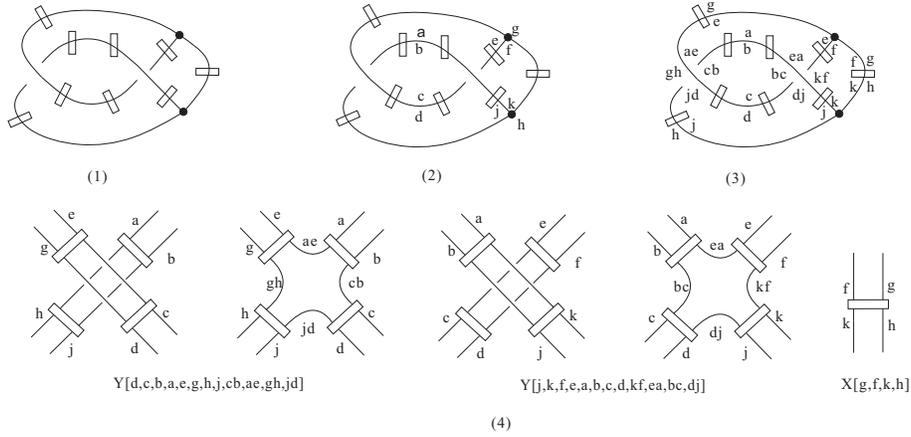}
\caption{The coding process in calculating generalized Yamada polynomial $R(D;$ $A^4,$ $-d^{-1}).$}
\label{F:codeingforD}
\end{figure}

An example is illustrated as in Figure \ref{F:codeingforD},
the virtual spatial graph diagram $D$ can be entered as:
\begin{equation*}
\begin{aligned}
&GraphG = \\
&Y[d,c,b,a,e,g,h,j,cb,ae,gh,jd]
Y[j,k,f,e,a,b,c,d,kf,ea,bc,dj]
X[g,f,k,h];
\end{aligned}
\end{equation*}
Once the graph $D$ is encoded, the command $YamadaPoly[t]$ will produce a value of generalized Yamada polynomial.
For example, the command $YamadaPoly[GraphG]$ will generate the value of generalized Yamada polynomial for the graph in Figure \ref{F:codeingforD}.
\vskip0.5cm
\textbf{Example 1:} Using $R(D; A^4, -d^{-1})$ to detect the virtual spatial graph diagram in Figure \ref{F:G50} and its mirror are not equivalent since $R(D; A^4, -d^{-1})$ is not equal to $R(D^*; A^4, -d^{-1})$ up to multiplication by some $(-A^4)^m$. Its coding is illustrated as in Figure \ref{F:exampleDp}.

\begin{figure}[!htbp]
\centering
\includegraphics[width=12cm]{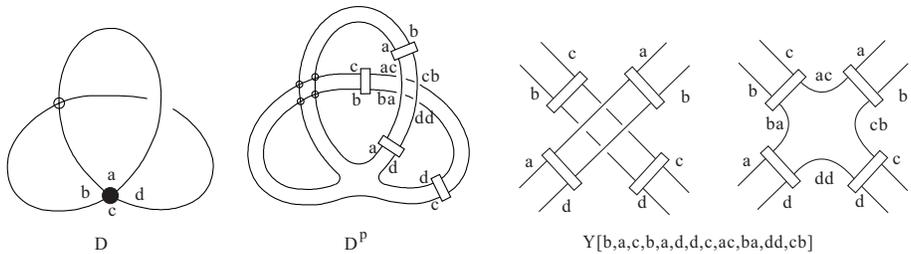}
\caption{The coding of the virtual spatial graph diagram in Figure \ref{F:G50}.}
\label{F:exampleDp}
\end{figure}

Input: \text{VirtualgraphD} =Y[b,a,c,b,a,d,d,c,ac,ba,dd,cb];\\
 \text{YamadaPoly}[\text{VirtualgraphD}]\\
Output: $-\frac{A^{-6}(1+3A^{4}+3A^{8}+2A^{12})}{1+A^4}$.
\vskip0.5cm

Input:  \text{VirtualgraphD1} = Y[c,b,a,d,d,c,b,a,ba,dd,cb,ac];\\
 \text{YamadaPoly}[\text{VirtualgraphD}]\\
Output: $-\frac{A^{2}(2+3A^{4}+3A^{8}+A^{12})}{1+A^4}$.

\begin{rem}
When the parameter $x\neq -d^{-1},\alpha\neq A^{4}$, the program is no longer valid.
In calculating a generalized Yamada polynomial, the difficulty is that the contribution of some circles in the boundary component is $z$ and the other is $\mu$.
But currently we cannot resolve this kind of difficulty in the Mathematica program.
It may be possible to use other methods to write programs which can be used to compute the full generalized Yamada polynomial.
\end{rem}

\section{Acknowledgements}

Xian'an Jin is supported by NSFC (No. 11671336), President's Funds
of Xiamen University (No. 20720160011).
Qingying Deng is supported by the China Scholarships Council (Grant Nos. 201706310009).
Louis H. Kauffman is supported by the Laboratory of Topology and Dynamics, Novosibirsk State University (contract no. 14.Y26.31.0025 with the Ministry of Education and Science of the Russian Federation).

 \end{document}